\documentclass[11pt,reqno]{amsart}
\usepackage[numbers,sort&compress]{natbib}
\usepackage[colorlinks,citecolor=blue]{hyperref}
\usepackage{amssymb}

\usepackage{epsfig,amssymb,amsmath,version}
\usepackage{amssymb,version,graphicx,fancybox,mathrsfs,multirow}
\usepackage{empheq}
\usepackage{url,hyperref}
\usepackage{float}
\usepackage{subfigure}
\usepackage{color,xcolor}
\usepackage{cases}
\usepackage{mathtools}

\usepackage[lined,boxed]{algorithm2e}
\usepackage{multirow}
\usepackage{algpseudocode}
\usepackage[numbers,sort&compress]{natbib}
\usepackage{lineno}
\usepackage{enumerate}

\makeatletter\@addtoreset{equation}{section}

\makeatletter
\@namedef{subjclassname@2010}{%
  \textup{2010} Mathematics Subject Classification}
\makeatother
\font\tenbi=cmmib10   at 11 pt \font\sevenbi=cmmib10 at 9pt
\font\fivebi=cmmib7 at 6pt
\newfam\bifam
\textfont\bifam=\tenbi \scriptfont\bifam=\sevenbi
\scriptscriptfont\bifam=\fivebi

\font\tendb=msbm10 at 12 pt \font\sevendb=msbm7
\newfam\dbfam
\textfont\dbfam=\tendb \scriptfont\dbfam=\sevendb

\newcommand{\dps}{\displaystyle}

\catcode`\@=11 \theoremstyle{plain}
\@addtoreset{equation}{section}   

\@addtoreset{figure}{section}
\renewcommand\thefigure{\thesection.\@arabic\c@figure}
\renewcommand{\thefigure}{\arabic{section}.\arabic{figure}}
\newtheorem{thm}{\bf Theorem}[subsection]

\newtheorem{cor}{\bf Corollary}

\newtheorem{lmm}{\bf Lemma}

\newenvironment{lemma}{\begin{lmm}}{\end{lmm}}
\theoremstyle{remark}
\newtheorem{rem}{\bf Remark}[section]
\theoremstyle{definition}

\numberwithin{table}{section}
\newcommand{\ba}{\begin{array}}\newcommand{\ea}{\end{array}}
\newcommand{\be}{\begin{eqnarray}}\newcommand{\ee}{\end{eqnarray}}
\newcommand{\beq}{\begin{equation*}}\newcommand{\eeq}{\end{equation*}}
\newcommand{\bex}{\begin{eqnarray*}}
\newcommand{\eex}{\end{eqnarray*}}

\newcommand{\bs}[1]{\boldsymbol{#1}}
\def\R{{\mathbb R}}
\newcommand{\mb}{\mathbf}

\def\O{\Omega}  \def\dO{\partial\Omega}

\def\x{{\bs x}} 

\def\f{{\bs f}}

\def\R{{\mathbb R}}

\def\u{{\mathbf{u}}}

{}

\begin{document}

\baselineskip=17pt

\title[Numerical methods for NSNPP equations]
{Efficient numerical methods for the Navier-Stokes-Nernst-Planck-Poisson equations$^\ast$}

\author{Xiaolan Zhou$^{1}$}
\author{Chuanju Xu$^{1,2}$}
\thanks{\hskip -12pt
$^\ast$This research is partially supported by NSFC grant 11971408.\\
$^{1}$School of Mathematical Sciences and Fujian Provincial Key Laboratory of Mathematical Modeling and High Performance Scientific Computing,
Xiamen University, Xiamen 361005, P.R. China.\\
${}^{2}$Corresponding author. Email: cjxu@xmu.edu.cn (C. Xu)}
\date{}

\begin{abstract}
We propose in this paper efficient first/second-order time-stepping schemes
for the evolutional Navier-Stokes-Nernst-Planck-Poisson equations.
The proposed schemes are constructed using an auxiliary variable reformulation and sophisticated treatment of
the terms coupling different equations.
By introducing a dynamic equation for the auxiliary variable and
reformulating the original equations into an equivalent system,
we construct first- and second-order semi-implicit linearized schemes for the
underlying problem.
The main advantages of the proposed method are: (1) the schemes are unconditionally stable
in the sense that a discrete energy keeps decay during the time stepping;
(2) the concentration components of the discrete solution preserve positivity and mass conservation;
(3) the delicate implementation shows that the proposed schemes
can be very efficiently realized, with computational complexity close to a semi-implicit scheme.
Some numerical examples are presented to demonstrate the accuracy and performance
of the proposed method. As far as the best we know, this is the first second-order method which satisfies all the above properties for the Navier-Stokes-Nernst-Planck-Poisson equations.

\end{abstract}


\maketitle

  \section{Introduction and motivation}
 \noindent

The Navier-Stokes-Nernst-Planck-Poisson (NSNPP) coupling system is a popular model for describing the electro-hydrodynamic phenomenon, which is originated in bio-electronic application. It is also known as the electro-fluid-dynamics, used to study the dynamics of electrically charged fluids, the motions of ionized particles or molecules and their interactions with electric fields and the surrounding fluid.
In electro fluid dynamics,
ions of different valences suspended in a fluid are carried by the fluid flow
and an electric potential, which results from both an applied potential on the boundary and the distribution of
charges carried by the ions. In addition, ionic diffusion is driven by the concentration gradients of the ions themselves.  In
turn, fluid flow is forced by the electrical field created by the ions.
These situations arise
frequently in a large number of physical, biophysical,
and industrial processes.
For more
details of the physical background issues of this system, we refer the reader to \cite{Rubinstein1990NSNPP, 2004Diffuse} and the references therein.

The mathematical property of the NSNPP system has been investigated in a number of papers.
Local existence of solutions in the whole space was obtained in \cite{Jerome2002}.
 Schmuck in \cite{schmuck2009} established global existence and uniqueness of weak solutions in a bounded domain in two and three dimensions for blocking boundary conditions on the ions and homogeneous Neumann boundary condition on the potential.
 Ryham \cite{Ryham2009} considered the homogeneous Dirichlet boundary conditions on the potential,
 and gave the global existence of weak solutions in two dimensions for large initial data and in three dimensions for small initial data and forces.
Bothe \cite{Bothe2014} studied the
 Robin boundary conditions for the electric potential, and showed the global existence and stability in two dimensions.
Zhao et al. \cite{Zhao2011Well} proved the local well-posedness for any initial data and global well-posedness for small initial data
in the critical Lebesgue spaces.
Deng et al. \cite{Chao2011Well} extended this result to Triebel-Lizorkin
space and Besov space with negative indices.
Zhang \cite{ZHANG2015102} proved the global existence for the Cauchy problem in two dimensions and established the $L^2$ decay estimates of solutions by using the Fourier splitting method.
 Constantin \cite{Constantin2018NSNPP, Constantin2020NSNPP,Constantin2021NSNPP} investigated the global existence of smooth solutions for different boundary conditions.

Numerical methods for the NSNPP system have also been subject of several works.
Yang et al. \cite{2001Electroosmotic} proposed an artificial compressibility method and
a finite difference/alternative direction method.
 Tsai et al. \cite{Tsai2005} employed this
method in capillary electrophoresis microchips, and tested some injection systems with different configurations.
Prohl and Schmuck \cite{Schmuck2010Convergent}
used finite element method for spatial discretization and an implicit time discretization
which preserves the non-negativity of the ionic concentrations. They also considered a projection method without non-negativity preserving.
He and Sun \cite{2018Mixed} proposed some time stepping and finite element
methods for NSNPP, which preserves the positivity and/or some form of
energy dissipation under certain conditions and specific spatial discretization.
The drawback of these methods is the need to solve nonlinear equations at each time step.
Liu and Xu \cite{liu-xu2017} proposed numerical methods of different orders
by combining several finite difference schemes in time and a spectral method
for the spatial discretization.
The proposed schemes result in several elliptic equations with
time-dependent coefficient to be solved at every time step.
The positivity-preserving of the first-order scheme was proved.

The scalar auxiliary variable approach, often called SAV \cite{SXY2018,SXY2017},
has received much attention recently. It has been proved to be a powerful tool to
design unconditionally stable schemes for a large class of problems
\cite{Cheng2018, Cheng2019, Hou2019, Zhou2019, Lin2020, 2020An, Li2020, Yao2021,LiMinghui2021}.
The aim of this paper is to make use of the auxiliary variable approach to construct highly efficient
time-stepping schemes for the NSNPP equations.
Precisely, our idea is to find a suitable auxiliary variable to treat the nonlinear terms involved in the equations,
and employ a splitting strategy to decouple different unknowns in the
Navier-Stokes part. A function transform approach for Nernst-Plank-Poisson will be also employed in the
construction.
We will show that the resulting scheme possesses the following properties:

- it is positivity preserving;

- it is mass conservative;

- it is unconditionally energy dissipative;

- it can be implemented in an efficient way: the computational complexity is equal to
solving several decoupled linear equations with constant coefficient at each time step.

The spatial discretization will make use of a spectral-Galerkin method in \cite{Shen1994EfficientSM}, for which fast solvers exist
for elliptic equations with constant coefficients.
We emphasize that the above attractive properties remain held at the full discrete level.

The remainder of this paper is structured as follows.  In Section 2, we first describe the NSNPP system, and the reformulation based on auxiliary variable approach.
In Section 3, we construct and analyze first/second order, linear, decoupled, and unconditionally stable scheme for the reformulate NSNPP equations.
We describe in Section 4 the implementation details of the proposed schemes, and show that the schemes can be efficiently implemented through solving a set of decoupled, linear elliptic equations with constant coefficients.
In Section 5, we present numerical examples to validate our schemes.
Some concluding remarks are given in Section 6.

 \section{Governing equations and reformulatation}

 \subsection{Navier-Stokes-Nernst-Planck-Poisson equations}
\noindent
 Let $\O\in\R^2 $ be a bounded Lipschitz domain and $T>0$.
 Given
  initial conditions $\mb u(\x,0), c_i(\x,0)$ $(i=1,...,m)$.
 We look for the velocity field $\mb u(\x,t)$, the pressure $p(\x,t)$, the mass concentration of ions $c_i(\x,t)$ $(i=1,...,m)$, and the electrostatic potential $\Phi(\x,t)$, satisfying the following Navier-Stokes-Nernst-Plank-Poisson equations
 (NSNPP) in $\Omega\times(0,T]$:
\begin{subequations}\label{nsnpp1}
\begin{align}
\dps& \partial_{t} \mb{u}+( \mb u\cdot\nabla)\mb u-\nu\Delta \mb{u}+\nabla p =- \big(\sum_{i=1}^{m} z_i c_i\big)\nabla \Phi , & {} &   \label{eq:1a} \\
 \dps& \nabla\cdot \mb{u} =0,   & {} &   \label{eq:1b}\\
 & \partial_t c_i =D_i\nabla\cdot\left( \nabla c_i +z_i c_i\nabla\Phi \right)-\nabla \cdot(\mb{u}c_i) ,\quad i=1,...,m,  & {} & \label{eq:1c}\\
&-\epsilon\Delta \Phi= \sum_{i=1}^{m} z_i c_i,& {} &  \label{eq:1d}
\end{align}
\end{subequations}
where $z_i\in \R$ are the ionic valences, $D_i$ denote the positive constant diffusivities,
$\epsilon>0$ is a small positive dimensionless number representing the ratio of the squared Debye length
to the physical characteristic length, $\nu>0$ is the kinematic viscosity.

We consider the blocking boundary conditions, i.e., vanishing of all normal fluxes
 for the ionic concentrations:
\begin{equation}\label{bc-noflux}
  \big(\mb{u}c_i -D_i( \nabla c_i +z_i c_i\nabla\Phi) \big)\cdot\mb n \big|_{\dO}=0,\ i=1,...,m,
\end{equation}
where $\mb n$ is outer normal on the boundary of $\dO$.
The boundary condition on the velocity $\mb u$ and the potential $\Phi$ is respectively
\begin{equation}\label{bcu}
  \mb u \big|_{\dO}=0,
\end{equation}
and
\begin{equation}\label{bc-phi}
 \frac{\partial\Phi}{\partial \mb n} \Big|_{\dO}=0.
 \end{equation}
Under the conditions \eqref{bcu} and \eqref{bc-phi},
noticing the identity $\nabla c=c \nabla \log c$,
it follows from \eqref{bc-noflux}:
\begin{equation}\label{bc-noflux2}
 c_i \frac{\partial( \log c_i  +z_i\Phi)}{\partial \mb n} \Big|_{\dO}
 =(\frac{\partial c_i}{\partial \mb n} +z_ic_i\frac{\partial\Phi}{\partial \mb n}) \Big|_{\dO}
= \frac{\partial c_i}{\partial \mb n}\Big|_{\dO}
 =0,\ i=1,...,m.
\end{equation}

It is readily seen that in the governing equations \eqref{eq:1a}-\eqref{eq:1d},
the pressure $p$ and electrostatic potential $\Phi$ are determined up to
an arbitrary constant. In order to fix this constant,
we impose the following zero mean conditions:
\begin{equation}\label{eq-pre}
  \int_{\O}p d\x=0, \quad  \int_{\O}\Phi d\x=0.
\end{equation}
We define two partial energy functionals $E_{ns}, E_{npp}$:
\begin{subequations}
\begin{align*}
\dps& E_{ns}=E_{ns}[\mb u]:=\frac{1}{2}\int_\O |\mb u |^2 d\x,\\
 \dps& E_{npp}= E_{npp}[\{c_i\},\Phi]
 :=\int_{\Omega}\Big[\sum_{i=1}^{m} c_i(\log c_i-1)+\frac{1}{2}\big(\sum_{i=1}^{m} z_i c_i\big)\Phi  \Big] d\x.
\end{align*}
\end{subequations}
The total energy functional is defined as the sum of these two partial functionals:
  \begin{equation}\label{energy}
E={E}[\mb u,\{c_i\},\Phi]:=E_{ns}+E_{npp}.
\end{equation}
 Using $\nabla\cdot \mb u=0$ and the boundary conditions \eqref{bcu}, \eqref{bc-phi}, and \eqref{bc-noflux2}, we have
\begin{align}\label{E_ns}
\dps  \frac{d E_{ns}}{dt} &=\int_{\O}\partial_t\mb u\cdot\mb ud\x
=\int_{\O}\Big[-(\mb u\cdot\nabla) \mb u+\nu\Delta \mb u-\nabla p-\big(\sum_{i=1}^{m} z_i c_i\big)\nabla\Phi\Big]\cdot \mb ud\x \notag \\
 &=
  -\nu\int_{\O}|\nabla\mb u|^2 d\x - \int_{\O}\big(\sum_{i=1}^{m} z_i c_i\big)\nabla\Phi\cdot\mb ud\x.
\end{align}
Taking the inner product of \eqref{eq:1c} with $\log c_i + z_i\Phi$, summing up for $i=1,...,m$,
and
using \eqref{eq:1d}, we obtain the following equality:
\begin{align}\label{E_npp}
\dps \frac{d E_{npp}}{dt}
&= \dps \sum_{i=1}^{m}\int_{\Omega}-D_i c_i|\nabla(\log c_i+ z_i \Phi) |^2d\x + \sum_{i=1}^{m}\int_{\O}-\nabla\cdot(\mb u c_i)(\log c_i+z_i\Phi)d\x  .
\end{align}
Furthermore, for the second term in the right hand side, we have
\begin{align}\label{eq:dd}
\dps \sum_{i=1}^{m}\int_{\O}-\nabla\cdot(\mb u c_i)(\log c_i+z_i\Phi)d\x&= \sum_{i=1}^{m}\int_{\O}(\mb u c_i)\cdot\nabla(\log c_i+z_i\Phi)d\x \notag\\
&=\sum_{i=1}^{m}\int_{\O}\mb u\cdot \nabla c_i d\x+\int_{\O}(\sum_{i=1}^{m}z_i c_i)\mb u \cdot\nabla\Phi d\x=\int_{\O}(\sum_{i=1}^{m}z_i c_i)\mb u \cdot\nabla\Phi d\x.
\end{align}
Then, combining \eqref{E_ns}, \eqref{E_npp}, and \eqref{eq:dd} gives
\begin{align}\label{energy-d}
 \dps \frac{d E}{dt}
  =&-\nu\int_{\O}|\nabla\mb u|^2 d\x-\sum_{i=1}^{m}\int_{\Omega}D_i c_i|\nabla(\log c_i+ z_i \Phi) |^2d\x.
\end{align}
\begin{lemma}
 The NSNPP problem \eqref{nsnpp1}-\eqref{bc-phi}
 satisfies the following properties:

 1) Mass conservation:
  \be\label{mcon}
  \int_\O c_i(\x,t)d\x=\int_\O c_i(\x,0)d\x, \ \ \forall t>0.
  \ee

2) Positivity: If the initial condition $c_i(\x, 0) > 0\ a.e.\ \x \in\O$, then $c_i(\x, t) > 0\  a.e. \ \x \in\O,\ t\in[0,T]$.

3) Energy dissipation:
\be\label{elaw}
\frac{d E}{dt}\leq 0.
\ee
\end{lemma}
\begin{proof} The proof is simple, and can be found in the literature. We give a proof sketch here
for the convenience of the reader.

1) Integrating \eqref{eq:1c} over $\O$, and using \eqref{eq:1b}, we obtain immediately \eqref{mcon}.

2) A proof of the positivity was given in \cite{schmuck2009}.

3) \eqref{energy-d} leads to the desired result \eqref{elaw}.
\end{proof}

\subsection{Auxiliary variable reformulation}
\noindent
Observing that $\int_\Omega \sum_{i=1}^{m} c_i(\log c_i-1) d\x$ is convex,
so there is a constant $C_0>0$, such that $E_{npp}[\{c_i\},\Phi]+C_0 \geq 1$.
We introduce the time-dependent auxiliary variable (AV) as follows:
\begin{align}\label{sav}
 \dps& r(t):=\sqrt{E_{npp}[\{c_i\},\Phi]+C_0}.
\end{align}
Then we have
\bex
\frac{dr}{dt}=\dps \frac{1}{2\sqrt{E_{npp}+C_0}} \frac{dE_{npp}}{dt}.
\eex
Insert \eqref{E_npp} into the above equation, add the zero-valued item $\int_\O\mb u \cdot\nabla \mb u\cdot\mb ud\x$, and multiply by the factor $\frac{r (t)}{\sqrt{E_{npp}[\{c_i\},\Phi]+C_0}}$, the governing equation for the auxiliary variable $r(t)$ can be obtained as follows:
\begin{align*}
 \dps \frac{dr(t)}{dt}= \dps- \frac{1}{2\sqrt{E_{npp}+C_0}}\bigg[\frac{r(t)}{\sqrt{E_{npp}[\{c_i\},\Phi]+C_0}} \sum_{i=1}^{m}\int_{\Omega}
 D_i c_i|\nabla\mu_i |^2d\x -\notag\\
   \int_\O\big(\sum_{i=1}^{m} z_i c_i\big)\nabla\Phi\cdot \mb{u} d\x
 -\int_\O(\mb u \cdot\nabla) \mb u\cdot\mb ud\x\bigg].
\end{align*}
For the Nernst-Plank-Poisson part, in order to preserve the positivity of the concentrations $c_i$,
we consider the variable transformation technique, which has already been used
in a number of papers; see \cite{nexp03wang,exp16-xu,exp18-xu,HuangShen2021}:
\begin{equation}\label{e1}
  c_i=T(\sigma_i):=\exp(\sigma_i), \quad i=1,...,m,
\end{equation}
where $\{\sigma_i\}_{i=1}^m$ are the new unknown functions to be determined.
Using this variable change, the boundary conditions \eqref{bc-noflux2}
are switched to
\begin{equation}\label{bc-noflux3}
 \frac{\partial \sigma_i}{\partial \mb n} \Big|_{\dO}
=0,\ i=1,...,m.
\end{equation}

With the auxiliary variable $r$ and the additional unknown variables
$\{\sigma_i\}_{i=1}^m$, we have the complete set of equations as follows:
\begin{subequations}\label{nsnpp3}
\begin{align}
\dps& \partial_{t} \mb{u}+\frac{r(t)}{\sqrt{E_{npp}(t)+C_0}}\Big((\mb u\cdot\nabla)\mb u+\big(\sum_{i=1}^{m} z_i c_i\big)\nabla\Phi\Big)-\nu\Delta \mb{u}+\nabla p =0, \label{eq:2a}\\
 \dps& \nabla\cdot \mb{u} =0,  \label{eq:2b} \\
 & \partial_t \sigma_i = D_i\Delta \sigma_i + D_i\big(|\nabla \sigma_i|^2+z_i(\nabla\sigma_i\cdot\nabla\Phi+\Delta\Phi)\big)-\nabla\cdot(\sigma_i \mb u), \label{eq:2c}
\\
 & c_i=\exp(\sigma_i),\quad i=1,...,m,\label{eq:2f}\\
 &-\epsilon\Delta \bar{\Phi}=\sum_{i=1}^{m} z_i c_i, \label{eq:2e}\\
 \dps& \frac{dr(t)}{dt}=-\dps \frac{1}{2\sqrt{E_{npp}(t)+C_0}}
 \Big[\frac{r(t)}{\sqrt{E_{npp}(t)+C_0}}\sum_{i=1}^{m}\int_{\Omega}D_i c_i|\nabla(\log c_i+ z_i \Phi) |^2d\x  \notag\\
 &\qquad\qquad  -\int_{\Omega}\big(\sum_{i=1}^{m} z_i c_i\big)\nabla\Phi\cdot \mb{u}d\x-  \int_{\Omega}(\mb{u}\cdot\nabla)\mb{u}\cdot \mb{u}d\x \Big], \label{eq:2g}\\
& \Phi=\frac{r(t)}{\sqrt{E_{npp}(t)+C_0}}\bar{\Phi},\label{eq:2h}
  \end{align}
\end{subequations}
subject to the boundary conditions \eqref{bcu}, \eqref{bc-phi},
\eqref{bc-noflux3}, and the initial conditions:
\be
  &&\u(\x,0)=\u_0(\x), \label{uic}\\
  &&r(0)=\sqrt{E_{npp}[\{c_i(\x,0)\},\ \Phi(\x,0)]+C_0}. \label{ric}
\ee
Noticing that $\int_\O (\mb u\cdot\nabla)\mb u\cdot \mb u d\x\equiv 0$ and
$\frac{r(t)}{\sqrt{E_{npp}(t)+C_0}}\equiv 1$, it is not difficult to check that the reformulated system \eqref{nsnpp3}
is strictly equivalent to the original system \eqref{nsnpp1} at the continuous level.

Taking the $L^{2}$ inner products of equation \eqref{eq:2a} with $ \mb u$, equation \eqref{eq:2g} with $2r$ respectively, then summing the resultants together, we obtain the following energy dissipation law:
\bex
\frac{d}{dt}\Big( \int_\O \frac{1}{2}|\mb u|^2d\x+r^2  \Big)
  =-\nu\int_{\O}|\nabla\mb u|^2 d\x-\Big|\frac{r(t)}{\sqrt{E_{npp}(t)+C_0}}\Big|^2\sum_{i=1}^{m}\int_{\Omega}D_i c_i|\nabla\mu_i |^2d\x\le 0.
\eex

Next we focus on the equivalent system \eqref{nsnpp3}
and construct numerical methods for this system.

\section{Unconditionally stable time-stepping schemes}

 Let $\Delta t$ be the time step size, $n\geq 0$ denotes the time step index, and $(\cdot)^n$ denotes the variable $(\cdot)$ at the time step $n$.
Let
\begin{equation}
  \mb u^0=\mb u(\x,0),\quad c_i^0=c_i(\x,0), \quad r^0=r(0).
\end{equation}
$\Phi^0$ and $p^0$ are obtained by solving the equations \eqref{eq:1d} and \eqref{eq:1a} at $t = 0$ under the constraint \eqref{eq-pre}, which read respectively in their weak forms:
\begin{equation}
\epsilon\int_{\Omega} \nabla \Phi^{0} \cdot \nabla q d \x
=\int_{\Omega}\big(\sum_{i=1}^{m} z_i c_i^0\big)q  d \x, \quad \forall q\in H^{1}(\Omega).
\end{equation}
\begin{equation}
\int_{\Omega} \nabla p^{0} \cdot \nabla q d \x
=\int_{\Omega}\Big(-\big(\sum_{i=1}^{m} z_i c_i^0\big)\nabla\Phi^0
-\mb u^0\cdot\nabla\mb u^0\Big) \cdot \nabla q d \x, \quad \forall q \in H^{1}(\Omega).
\end{equation}

\subsection {Time stepping schemes}

We now propose different schemes using first/second order BDF for the equivalent system \eqref{nsnpp3}.

\textbf{Scheme1}: The first scheme is constructed by
making use of BDF1 and
some first order approximations to different terms in
\eqref{eq:2a}-\eqref{eq:2h}: given the initial data
$(u^0,p^0,\Phi^0,\{c_i^0\}, r^0)$, compute $(u^{n+1}, p^{n+1},\Phi^{n+1},\{c_i^{n+1}\}, r^{n+1})$ for $n\ge 0$ successively by solving
\begin{subequations} \label{scheme1}
  \begin{align}
   \dps & \frac{\sigma_i^{n+1}-\sigma_i^{n}}{\Delta t}-D_i\Delta\sigma_i^{n+1}= \notag\\ \dps &\hspace{1cm}D_i\big[|\nabla \sigma_i^{n}|^2+z_i(\nabla\sigma_i^{n}\cdot\nabla\Phi^{n}+\Delta\Phi^{n})\big]-\nabla\cdot(\sigma_i^{n} \mb u^{n}),
   \text{ with } \left.\frac{\partial \sigma_{i}^{n+1}}{\partial \mb{n}}\right|_{\partial \Omega}=0, \label{c1st-1} \\
   \dps & \bar{c}_i^{n+1}=\exp(\sigma_i^{n+1}),\label{ci1}\\
   \dps& \lambda_i^{n+1}\int_\O \bar{c}_i^{n+1}-\int_\O  c_i^{n}=0, \label{ci2}\\
  \dps& c_i^{n+1}=\lambda_i^{n+1}\bar{c}_i^{n+1},\ \ i=1,...,m,\label{ci3}   \\
  \dps &-\epsilon\Delta\bar{\Phi}^{n+1}= \sum_{i=1}^{m}z_ic_i^{n+1}, \quad \text{ with }
    \int_\O \bar{\Phi}^{n+1}d\x =0, \quad \left.\frac{\partial \bar{\Phi}^{n+1}}{\partial \mb{n}}\right|_{\partial \Omega}=0, \label{c1st-3}\\
  &\frac{\tilde{\mb{u}}^{n+1}-\mb{u}^{n}}{\Delta t}+
    \frac{r^{n+1}}{\sqrt{\bar{E}_{npp}^{n+1}+C_0}}\Big[({\mb u}^n\cdot\nabla){\mb u}^n+(\sum_{i=1}^{m}z_ic_i^{n+1})\nabla\bar{\Phi}^{n+1}\Big]- \nonumber\\
    &\hspace{6cm} \nu\Delta \tilde{\mb{u}}^{n+1}+\nabla p^{n}=0, \text{ with }\left.\tilde{\mb u}^{n+1}\right|_{\partial \Omega}=0,
     \label{c1st-4}\\
    &\Delta \psi^{n+1}=\frac{1}{\Delta t}\nabla\cdot\tilde{\u}^{n+1}, \text{ with }\left.\frac{\partial \psi^{n+1}}{\partial \mb{n}}\right|_{\partial \Omega}=0, \nonumber\\
    &\hspace{5cm}  \u^{n+1}=\tilde{\u}^{n+1}-\Delta t\nabla\psi^{n+1}, \quad p^{n+1}=\psi^{n+1}+p^n, \label{c1st-5}\\
    &\frac{r^{n+1}-{r}^{n}}{\Delta t} =-\frac{1}{2\sqrt{\bar{E}_{npp}^{n+1}+C_0}}\bigg[ \frac{r^{n+1}}{\sqrt{\bar{E}_{npp}^{n+1}+C_0}}\int_\O\sum_{i=1}^{m}D_i\big(c_i^{n+1}|\nabla\bar{\mu}_{i}^{n+1}|^2d\x
     \notag \\
   & \hspace{3cm} -\int_\O(\sum_{i=1}^{m}z_ic_i^{n+1})\nabla\bar{\Phi}^{n+1}\cdot\tilde{\mb{u}}^{n+1}d\x
   - \int_\O(\mb{u}^n\cdot\nabla)\mb{u}^n\cdot \tilde{\mb{u}}^{n+1}d\x\bigg], \label{c1st-8}\\
    & \Phi^{n+1}=\frac{r^{n+1}}{\sqrt{\bar{E}_{npp}^{n+1}+C_0}} \bar{\Phi}^{n+1}.\label{c1st-9}
  \end{align}
\end{subequations}
For simplification of presentation, we have used in the above scheme two additional notations $\bar{\mu}_i^{n+1}$ and $\bar{E}_{npp}^{n+1}$:
    \begin{align}\label{eq:h1}
  \dps& \bar{\mu}_i^{n+1}=\log c_i^{n+1} + z_i\bar{\Phi}^{n+1} ,\ i=1,...,m,\\
  \dps& \bar{E}_{npp}^{n+1}=E_{npp}[\{c_i^{n+1}\},\bar{\Phi}^{n+1}].
    \end{align}

\textbf{Scheme2}: The same idea can be used to construct second order schemes.
For ease of notation
we will use $f^{*,n+1}$ to mean $2f^n-f^{n-1}$.
We propose the following second order scheme:
\begin{subequations}\label{scheme2}
  \begin{align}
    \dps & \frac{3\sigma_i^{n+1}-4\sigma_i^{n}+\sigma_i^{n-1}}{2\Delta t}-D_i\Delta\sigma_i^{n+1}=D_i\Big[|\nabla \sigma_i^{*,n+1}|^2+ \nonumber\\
    \dps & \hspace{1.5cm} z_i(\nabla\sigma_i^{*,n+1}\cdot\nabla\Phi^{*,n+1}+\Delta\Phi^{*,n+1})\Big]-\nabla\cdot(\sigma_i^{*,n+1}\mb {\mb u}^{*,n+1}), \text{ with } \left.\frac{\partial \sigma_{i}^{n+1}}{\partial \mb{n}}\right|_{\partial \Omega}=0, \label{2nd-1} \\
  \dps & \bar{c}_i^{n+1}=\exp(\sigma_i^{n+1}),\label{ci1-2}\\
   \dps& \lambda_i^{n+1}\int_\O \bar{c}_i^{n+1}-\int_\O  c_i^{n}=0, \label{ci2-2}\\
  \dps& c_i^{n+1}=\lambda_i^{n+1}\bar{c}_i^{n+1},\ \ i=1,...,m, \label{ci3-2}   \\
  \dps &-\epsilon\Delta \bar{\Phi}^{n+1}= \sum_{i=1}^{m}z_ic_i^{n+1} \ \ \text{ with }
    \big(\bar{\Phi}^{n+1},1\big) =0, \quad \left.\frac{\partial \bar{\Phi}^{n+1}}{\partial \mb{n}}\right|_{\partial \Omega}=0,\label{2nd-3}\\
  \dps & \frac{3\tilde{\mb{u}}^{n+1}-4\mb{u}^{n}+\mb{u}^{n-1}}{2\Delta t}+
    \frac{r^{n+1}}{\sqrt{\bar{E}_{npp}^{n+1}+C_0}}\Big[{\mb u}^{*,n+1}\cdot\nabla{\mb u}^{*,n+1}\!+\!(\sum_{i=1}^{m}z_ic_i^{n+1})\nabla\bar{\Phi}^{n+1}\Big]-
        \nonumber\\
    &\hspace{7cm} \nu\Delta \tilde{\mb{u}}^{n+1}+\nabla p^{n}=0, \text{ with }\left.\tilde{\mb u}^{n+1}\right|_{\partial \Omega}=0,\label{2nd-4}\\
    &\Delta \psi^{n+1}=\frac{3}{2\Delta t}\nabla\cdot\tilde{\u}^{n+1}\text{ with } \left.\frac{\partial \psi^{n+1}}{\partial \mb{n}}\right|_{\partial \Omega}=0,\nonumber\\
    &\hspace{3cm} \u^{n+1}=\tilde{\u}^{n+1}-\frac{2}{3\Delta t}\nabla\psi^{n+1},\quad   p^{n+1}=\psi^{n+1}+p^{n}-\nu \nabla\cdot\tilde{\u}^{n+1}
\label{2nd-5}\\
    &\frac{3r^{n+1}-4r^{n}+r^{n-1}}{2\Delta t}=-\frac{1}{2\sqrt{\bar{E}_{npp}^{n+1}+C_0}} \bigg[\frac{r^{n+1}}{\sqrt{\bar{E}_{npp}^{n+1}+C_0}}\int_\O\sum_{i=1}^{m}D_i\big(c_i^{n+1}|\nabla\bar{\mu}_{i}^{n+1}|^2d\x \notag \\
   & \hspace{2cm} -\int_\O(\sum_{i=1}^{m}z_ic_i^{n+1})\nabla\bar{\Phi}^{n+1}\cdot \tilde{\mb{u}}^{n+1}d\x- \int_\O(\mb{u}^{*,n+1}\cdot\nabla)\mb{u}^{*,n+1}\cdot\tilde{\mb{u}}^{n+1}d\x\bigg],\label{2nd-8}\\
    & \Phi^{n+1}=\frac{r^{n+1}}{\sqrt{\bar{E}_{npp}^{n+1}+C_0}} \bar{\Phi}^{n+1}.\label{2nd-9}
  \end{align}
\end{subequations}

Notice that the discretization of the
Navier-Stokes equations, i.e., \eqref{2nd-4}-\eqref{2nd-5} in \textbf{Scheme2}
  without the auxiliary variable $r$, is the so called rotational pressure correction method (RPC) \cite{timmermans1996approximate,guermond2004error}.
For the Navier-Stokes equations alone, it has been proved in \cite{LSL2020new}
that the time semi-discretization using the RPC and auxiliary variable approach is unconditionally stable.
However, our analysis shows, as we will see in the next section,
that the unconditional stability of the full discrete scheme using spectral method for the
spatial discretization necessitates modifying the RPC. That is why we propose an alternative of
second order scheme using a modification of RPC (termed as MRPC hereafter) below.

\textbf{Scheme2b}: Almost same as the scheme \textbf{Scheme2}, except of
\eqref{2nd-4}-\eqref{2nd-5}, which is modified as
\begin{subequations}
  \begin{align}
  \dps & \frac{3\tilde{\mb{u}}^{n+1}-4\mb{u}^{n}+\mb{u}^{n-1}}{2\Delta t}+
    \frac{r^{n+1}}{\sqrt{\bar{E}_{npp}^{n+1}+C_0}}\Big[{\mb u}^{*,n+1}\cdot\nabla{\mb u}^{*,n+1}\!+\!(\sum_{i=1}^{m}z_ic_i^{n+1})\nabla\bar{\Phi}^{n+1}\Big]
    \nonumber\\
    \dps & \hspace{5cm}\!-\nu\Delta \tilde{\mb{u}}^{n+1}\!+\!\nabla (p^{n}{+\nu \nabla\cdot\tilde{\u}^{n}})=0, \text{ with }\left.\tilde{\mb u}^{n+1}\right|_{\partial \Omega}=0, \label{2nd-4b}\tag{3.7f'}\\
    &\Delta \psi^{n+1}=\frac{3}{2\Delta t}\nabla\cdot\tilde{\u}^{n+1}\text{ with } \left.\frac{\partial \psi^{n+1}}{\partial \mb{n}}\right|_{\partial \Omega}=0,\nonumber\\
    &\hspace{1.5cm}  \u^{n+1}=\tilde{\u}^{n+1}-\frac{2}{3\Delta t}\nabla\psi^{n+1},\quad p^{n+1}=\psi^{n+1}+p^{n}+\nu\nabla\cdot\tilde{\u}^n-\nu \nabla\cdot\tilde{\u}^{n+1}.
\label{2nd-5b} \tag{3.7g'}
  \end{align}
\end{subequations}
The motivation of this modification will be explained in Remark \ref{rmk4.1},
and become more clear in the stability
analysis that follows.
%

\begin{rem}\label{rmk4.1}
Before analyzing the stability property, it is worth noting a number of points about the above schemes,
comparing with the existing schemes for the NSNPP system.
\begin{itemize}

   \item[i)] We will show below that the proposed schemes satisfy the following properties:
  (1) concentration positivity preserving;
  (2) mass conservation;
  (3) unconditional stability;
  (4) resulting in decoupled linear equations with constant coefficient to be solved at each time step.
In particular, to the best of our knowledge, the scheme \eqref{scheme2} is the first in the literature
developed for the NSNPP system,
which is second order convergent and possesses all these advantages.

  \item[ii)] The schemes constructed in \cite{liu-xu2017} led to some linear equations to be solved at each time step.
However these equations involve time-dependent coefficient, thus need re-computations of the coefficient matrices
every time step. Moreover there is a lack of stability analysis for these schemes.
A time-stepping scheme was also proposed in \cite{2018Mixed},
but it is fully coupled and nonlinear.

  \item[iii)] The key to achieving the desirable properties in the schemes \eqref{scheme1} and \eqref{scheme2}
  lies in the introduction of suitable auxiliary variables, which allow to decouple different unknown variables and to
  eliminate the undesirable nonlinear terms in the proof of the stability.

  \item[iv)]  Compared to \eqref{2nd-4}-\eqref{2nd-5} in \textbf{Scheme2},
  we have technically introduced the additional term $\nu\nabla\cdot\tilde{\u}^n$ in
\eqref{2nd-4b}-\eqref{2nd-5b} in \textbf{Scheme2b}. This term is helpful in establishing the unconditional
stability of the full discrete version of \textbf{Scheme2b}.
However our numerical test (see Table \ref{ex1-tab2} for example),
shows that this additional term has essentially no impact on the stability and
convergence order. Therefore its presence in the scheme seems to be of a purely technical nature.
\end{itemize}
\end{rem}

\subsection{Unconditional stability}
In this subsection, we prove the unconditional stability of the proposed schemes by analyzing the decay behavior of the discrete energy. For simplification of notation, we will use $\dps \xi^{n+1}$ to denote $\frac{r^{n+1}}{\sqrt{\bar{E}_{npp}^{n+1}+C_0}}$.
We start with proving the main properties of the first order scheme in the following theorem.
\begin{thm}\label{thm1}
    Given $\{c_i^n\}, \Phi^{n}, \mb u^{n}, p^{n}, r^{n}, \sigma_i^n=\log c_i^n$.
    Suppose $c_i^n>0 $ and $\int_\O c_i^nd\x=\int_\O c_i^0d\x, i=0,...,m$.
Then the solution $(\{c_i^{n+1}\},\mb u^{n+1}, p^{n+1},r^{n+1})$ of the discrete problem
    \eqref{scheme1} satisfies the following properties:
 \begin{enumerate}
   \item[i)] Positivity preserving: $c_i^{n+1}>0$.
   \item[ii)] Mass conserving: $\int_\O c_i^{n+1}d\x=\int_\O c_i^nd\x,\ i=1,...,m$.
   \item[iii)] Energy dissipation:
   \begin{equation}\label{eq:2-8eng}
    \mathcal{E}^{n+1}-\mathcal{E}^n\leq 0,
\end{equation}
where
\bex
  \mathcal{E}^{n+1}:=\|{\mb u}^{n+1}\|^2+
  \Delta t^2\| \nabla p^{n+1}\|^{2} +|r^{n+1}|^2.
\eex
\item[iv)] The quantities
$\|\u^{n}\|,\ |r^{n}|$, and $|\xi^{n}|$ are bounded for all $n\geq 0$.
 \end{enumerate}

\end{thm}

\begin{proof}
We first prove the positivity preserving and mass conservation.
By \eqref{ci1}, we have
$\bar c_i^{n+1}>0$. Then it follows from \eqref{ci2} and the positivity of the concentration at the previous step, i.e, $c_i^n>0$,
$i=1,\dots,m$:
\bex
\lambda_i^{n+1}=\frac{\int_\O c_i^nd\x}{\int_\O \bar{c}_i^{n+1}d\x} > 0, \quad i=1,\dots,m.
\eex
Thus we derive from \eqref{ci3} the positivity of the concentration:
\bex
{c}_i^{n+1}=\lambda_i^{n+1}\bar{c}_i^{n+1}>0, \quad i=1,\dots, m.
\eex
The mass conservation ii) is given by
\bex
\int_\O c_i^{n+1}d\x =\frac{\int_\O c_i^nd\x}{\int_\O \bar{c}_i^{n+1}d\x} \int_\O \bar{c}_i^{n+1}d\x = \int_\O c_i^nd\x, \quad i=1,\dots,m.
\eex
Now we turn to prove the energy dissipation.
   Taking the $L^{2}$ inner product of equation \eqref{c1st-4} with $ 2\Delta t \tilde{\mb u}^{n+1}$,
by using the identify
$
2\left(a^{k+1}, a^{k+1}-a^{k}\right)
=|a^{k+1}|^{2}+|a^{k+1}- a^{k}|^{2}
-|a^{k}|^{2}$,
 we get
\begin{align}\label{eq:proof1-1}
    &\|\tilde{\mb u}^{n+1}\|^2-\|{\mb u}^{n}\|^2+2\Delta t\big(\nabla p^n,\tilde{\mb u}^{n+1}\big)+2\Delta t
    \xi^{n+1}\big({\mb u}^n\cdot\nabla{\mb u}^n+(\sum_{i=1}^{m}z_ic_i^{n+1})\nabla\bar{\Phi}^{n+1},\tilde{\mb{u}}^{n+1}\big) \notag\\
      & = -2\Delta t\nu \|\nabla \tilde{\mb{u}}^{n+1}\|^{2}-\|\tilde{\mb u}^{n+1}-{\mb u}^{n}\|^2.
\end{align}
Furthermore, we deduce from \eqref{c1st-5}
\begin{equation}\label{eq:2-6rr}
    \left\|\mathbf{u}^{n+1}\right\|^{2}+\Delta t^2\left\|\nabla p^{n+1}\right\|^{2}=\left\|\tilde{\mathbf{u}}^{n+1}\right\|^{2}+2 \Delta t\left(\nabla p^{n}, \tilde{\mathbf{u}}^{n+1}\right)+\Delta t^2\left\|\nabla p^{n}\right\|^{2}.
\end{equation}
Summing up \eqref{eq:proof1-1} and \eqref{eq:2-6rr}, we obtain
\begin{align}\label{eq:proof1-2}
    &\|{\mb u}^{n+1}\|^2-\|{\mb u}^{n}\|^2
    + \Delta t^2\left\|\nabla p^{n+1}\right\|^{2}- \Delta t^2\left\|\nabla p^{n}\right\|^{2}\notag\\
    &+ 2\Delta t
    \xi^{n+1}\big({\mb u}^n\cdot\nabla{\mb u}^n +(\sum_{i=1}^{m}z_ic_i^{n+1})\nabla\bar{\Phi}^{n+1},\tilde{\mb{u}}^{n+1}\big)   = -2\Delta t\nu \|\nabla \tilde{\mb{u}}^{n+1}\|^{2} -\|\tilde{\mb u}^{n+1}-{\mb u}^{n}\|^2.
\end{align}
The last (undesirable) term in the left hand side can be cancelled by subtracting \eqref{eq:proof1-2} from
\eqref{c1st-8} multiplied by $2\Delta t r^{n+1}$:
\begin{align*}
    &\|{\mb u}^{n+1}\|^2-\|{\mb u}^{n}\|^2+
   \Delta t^2\left\|\nabla p^{n+1}\right\|^{2}- \Delta t^2\left\|\nabla p^{n}\right\|^{2} + |r^{n+1}|^2- |r^{n}|^2 \notag\\
      & = -2\Delta t\nu\|\nabla \tilde{\mb{u}}^{n+1}\|^{2}-\|\tilde{\mb u}^{n+1}-{\mb u}^{n}\|^2
      - \Delta t |\xi^{n+1}|^2 \sum_{i=1}^{m}D_i(c_i^{n+1},|\nabla\bar{\mu}_{i}^{n+1}|^2) - |r^{n+1}-r^{n}|^2.
\end{align*}
Then the energy dissipation law \eqref{eq:2-8eng} follows from the fact that the right hand side of the above equality is non-positive.
\\
Finally the boundedness of $\|\u^{n}\|,\ |r^{n}|$, and $|\xi^{n}|$ follows directly from the
boundedness of $\mathcal{E}^{n},\ \forall n\geq 0$.
\\
The proof is completed.
 \end{proof}

The similar results for the second order scheme is given in the next theorem.
\begin{thm}\label{thm2}
 Given $\{c_i^k\}, \Phi^{k}, \mb u^{k}, p^{k}, r^{k}, \sigma_i^k=\log c_i^k, \ k=n,n-1$. Suppose $c_i^k>0$ and $\int_\O c_i^kd\x=\int_\O c_i^0d\x,\ i=0,...,m;\ k=n,n-1$.
Then the solution $(\{c_i^{n+1}\},\tilde{\mb u}^{n+1},\mb u^{n+1}, p^{n+1},r^{n+1})$ of the discrete problem \eqref{scheme2}, both \textbf{Scheme2} and \textbf{Scheme2b}, satisfies the following properties:
 \begin{enumerate}
   \item[i)] Positivity preserving: $c_i^{n+1}>0$.
   \item[ii)] Mass conserving: $\int_\O c_i^{n+1}d\x=\int_\O c_i^nd\x,\ i=1,...,m$.
   \item[iii)] Energy dissipation:
   \begin{equation}\label{eq:3-8eng}
    \mathcal{E}^{n+1}-\mathcal{E}^n\leq 0,
\end{equation}
where
in  \textbf{Scheme2},
\begin{align}\label{eq:3-8-2nd}
  \mathcal{E}^{n+1}:= \frac{1}{2}\|{\mb u}^{n+1}\|^2+\frac{1}{2}\|2{\mb u}^{n+1}-{\mb u}^{n}\|^2+
  \frac{2}{3}\Delta t^2\| \nabla p^{n+1}+\nu\omega^{n+1}\|^{2}\notag\\
  +\nu\|\omega^{n+1}\|^2 +\frac{1}{2}|r^{n+1}|^2+\frac{1}{2}|2r^{n+1}-r^n|^2
\end{align}
with $\{w^{n+1}\}$ being recursively defined by
\begin{equation}\label{w}
  \omega^0=0,\ \omega^{n+1}=\omega^{n}+\nabla\cdot\tilde{\mb u}^{n+1};
\end{equation}
and in \textbf{Scheme2b},
\begin{align}\label{eq:3-8-2ndb}
  \mathcal{E}^{n+1}:= \frac{1}{2}\|{\mb u}^{n+1}\|^2+\frac{1}{2}\|2{\mb u}^{n+1}-{\mb u}^{n}\|^2+
  \frac{2}{3}\Delta t^2\| \nabla (p^{n+1}+\nu \nabla\cdot\tilde{\mb u}^{n+1})\|^{2}\notag\\
  +\frac{1}{2}|r^{n+1}|^2+\frac{1}{2}|2r^{n+1}-r^n|^2. \tag{3.15'}
\end{align}
\item[iv)] The quantities
$\|\u^{n}\|,\ |r^{n}|$, and $|\xi^{n}|$ are bounded for all $n\geq 0$.
 \end{enumerate}
\end{thm}

\begin{proof}
The positivity preserving i) and mass conservation ii) can be proved in the exactly same way as in Theorem \ref{thm1}
for the first order scheme.\\
Next we demonstrate the discrete energy dissipation law iii). We first do it for \textbf{Scheme2}.
 Taking the $L^{2}$ inner product of equation \eqref{2nd-4} with $ 2\Delta t \tilde{\mb u}^{n+1}$, we obtain
  \begin{equation}\label{eq:proof2-1}
    \begin{split}
    &\big(3\tilde{\mb{u}}^{n+1}-4\mb{u}^{n}+\mb{u}^{n-1}, \tilde{\mb u}^{n+1}\big)+2\Delta t\big(\nabla p^n,\tilde{\mb u}^{n+1}\big)\\
      &\hspace{2cm} +2\Delta t
    \xi^{n+1}\big({\mb u}^n\cdot\nabla{\mb u}^n+(\sum_{i=1}^{m}z_ic_i^{n+1})\nabla\bar{\Phi}^{n+1},\tilde{\mb{u}}^{n+1}\big) = -2\Delta t\nu \|\nabla \tilde{\mb{u}}^{n+1}\|^{2}.
    \end{split}
\end{equation}
Applying the well-known identity
\begin{equation*}
\begin{aligned}
&2(a^{k+1}, 3 a^{k+1}-4 a^{k}+a^{k-1})=|a^{k+1}|^{2}+|2 a^{k+1}-a^{k}|^{2} \\
&\hspace{6cm} +|a^{k+1}-2 a^{k}+a^{k-1}|^{2}-|a^{k}|^{2}-|2 a^{k}-a^{k-1}|^{2},
\end{aligned}
\end{equation*}
we have
\begin{align}\label{eq:proof2-2}
&\left(3 \tilde{\mathbf{u}}^{n+1}-4 \mathbf{u}^{n}+\mathbf{u}^{n-1}, \tilde{\mathbf{u}}^{n+1}\right)=\left(3\left(\tilde{\mathbf{u}}^{n+1}-\mathbf{u}^{n+1}\right)+3 \mathbf{u}^{n+1}-4 \mathbf{u}^{n}+\mathbf{u}^{n-1}, \tilde{\mathbf{u}}^{n+1}\right) \notag\\
&=3\left(\tilde{\mathbf{u}}^{n+1}-\mathbf{u}^{n+1}, \tilde{\mathbf{u}}^{n+1}\right)+\left(3 \mathbf{u}^{n+1}-4 \mathbf{u}^{n}+\mathbf{u}^{n-1}, \mathbf{u}^{n+1}\right) +\left(3 \mathbf{u}^{n+1}-4 \mathbf{u}^{n}+\mathbf{u}^{n-1}, \tilde{\mathbf{u}}^{n+1}-\mathbf{u}^{n+1}\right) \notag\\
& =\frac{3}{2}\left(\left\|\tilde{\mathbf{u}}^{n+1}\right\|^{2}-\left\|\mathbf{u}^{n+1}\right\|^{2}
+\left\|\tilde{\mathbf{u}}^{n+1}-\mathbf{u}^{n+1}\right\|^{2}\right)
+\frac{1}{2}\big(\left\|\mathbf{u}^{n+1}\right\|^{2}+\left\|2 \mathbf{u}^{n+1}-\mathbf{u}^{n}\right\|^{2}\big) \notag\\
&\qquad -\frac{1}{2}\big(\left\|\mathbf{u}^{n}\right\|^{2}+\left\|2 \mathbf{u}^{n}-\mathbf{u}^{n-1}\right\|^{2}\big) + \frac{1}{2}\left\|\mathbf{u}^{n+1}-2 \mathbf{u}^{n}+\mathbf{u}^{n-1}\right\|^{2}.
\end{align}
In the last equality, we have also used the fact
\begin{align*}
& \left(3 \mathbf{u}^{n+1}-4 \mathbf{u}^{n}+\mathbf{u}^{n-1}, \tilde{\mathbf{u}}^{n+1}-\mathbf{u}^{n+1}\right)   =\big(3 \mathbf{u}^{n+1}-4 \mathbf{u}^{n}+\mathbf{u}^{n-1}, \frac{2\Delta t}{3} \nabla(p^{n+1}-p^{n})\big) \\
& = \big(\nabla\cdot\big(3 \mathbf{u}^{n+1}-4 \mathbf{u}^{n}+\mathbf{u}^{n-1}\big), \frac{2\Delta t}{3} (p^{n+1}-p^{n})\big) =0.
\end{align*}
Furthermore, let $ H^{n+1}:=p^{n+1} +\nu\omega^{n+1}$ with $w^{n+1}$ being given in \eqref{w}, we have
\bex
p^{n+1}-p^n+\nu\nabla \cdot\tilde{\mb u}^{n+1}=H^{n+1}-H^{n},
\eex
and rewrite \eqref{2nd-5} as
\begin{equation}\label{eq:proof2-4}
  \sqrt{\frac{3}{2}}\mathbf{u}^{n+1}+\sqrt{\frac{2}{3}}\Delta t \nabla H^{n+1}=\sqrt{\frac{3}{2}}\tilde{\mathbf{u}}^{n+1}+\sqrt{\frac{2}{3}}\Delta t \nabla H^{n}.
\end{equation}
Now, taking the inner product of \eqref{eq:proof2-4} with itself on both sides and noticing that $(\nabla H^{n+1},\mb u^{n+1})=(H^{n+1},\nabla\cdot\mb u^{n+1})=0$, we obtain
\begin{equation}\label{eq:proof2-5}
\begin{split}
    &\frac{3}{2}\left\|\mathbf{u}^{n+1}\right\|^{2}+\frac{2}{3}\Delta t^{2}\left\|\nabla H^{n+1}\right\|^{2}=\frac{3}{2}\left\|\tilde{\mathbf{u}}^{n+1}\right\|^{2}+\frac{2}{3}\Delta t^2\left\|\nabla H^{n}\right\|^{2}\\
    &\hspace{7cm}+2 \Delta t\left(\nabla p^{n}, \tilde{\mathbf{u}}^{n+1}\right)+2\Delta t \nu \left(\nabla\omega^{n}, \tilde{\mathbf{u}}^{n+1}\right).
\end{split}
\end{equation}
On the other side, it follows from \eqref{w}
\begin{align}\label{eq:proof2-s}
 2 \Delta t\nu \left(\tilde{\mathbf{u}}^{n+1}, \nabla\omega^{n}\right)
=&-2 \nu \Delta t\left(\nabla \cdot\tilde{\mathbf{u}}^{n+1},\omega^{n}\right) \notag\\
=&-2 \nu \Delta t\left(\omega^{n+1}-\omega^{n},\omega^{n}\right) \notag\\
=&  \nu\Delta t\left(\left\|\omega^{n}\right\|^{2}-\left\|\omega^{n+1}\right\|^{2}+\left\|\omega^{n+1}-\omega^{n}\right\|^{2}\right) \notag\\
=&  \nu \Delta t\left\|\omega^{n}\right\|^{2}- \nu\Delta t\left\|\omega^{n+1}\right\|^{2}+ \nu\Delta t\left\|\nabla \cdot \tilde{\mb u}^{n+1}\right\|^{2} .
\end{align}
Using the identity
\bex
\|\nabla \times \mathbf{v}\|^{2}+\|\nabla \cdot \mathbf{v}\|^{2}=\|\nabla \mathbf{v}\|^{2}, \quad \forall \mathbf{v} \in \mathbf{H}_{0}^{1}(\Omega),
\eex
we get
\begin{equation}\label{eq:w1}
\begin{gathered}
2 \Delta t\nu \left(\tilde{\mathbf{u}}^{n+1}, \nabla\omega^{n}\right)= \nu \Delta t\left\|\omega^{n}\right\|^{2}- \nu \Delta t\left\|\omega^{n+1}\right\|^{2}
+ \nu \Delta t\left\|\nabla \tilde{\mathbf{u}}^{n+1}\right\|^{2}- \nu \Delta t\left\|\nabla \times \tilde{\mathbf{u}}^{n+1}\right\|^{2}.
\end{gathered}
\end{equation}
Combining \eqref{eq:proof2-5} and \eqref{eq:w1} yields
\begin{align}\label{eq:wp}
2 \Delta t\left(\nabla p^{n}, \tilde{\mathbf{u}}^{n+1}\right)=\frac{3}{2}\left\|\mathbf{u}^{n+1}\right\|^{2}-
\frac{3}{2}\left\|\tilde{\mathbf{u}}^{n+1}\right\|^{2}+\frac{2}{3}\Delta t^2\left\|\nabla H^{n+1}\right\|^{2}-\frac{2}{3}\Delta t^2\left\|\nabla H^{n}\right\|^{2} \notag\\
\qquad + \nu \Delta t\left\|\omega^{n+1}\right\|^{2}-\nu \Delta t\left\|\omega^{n}\right\|^{2}
- \nu \Delta t\left\|\nabla \tilde{\mathbf{u}}^{n+1}\right\|^{2}+ \nu \Delta t\left\|\nabla \times \tilde{\mathbf{u}}^{n+1}\right\|^{2}.
\end{align}
Finally, multiplying \eqref{2nd-8-f} by $2\Delta t r^{n+1}_N$, and using \eqref{eq:proof2-1}, \eqref{eq:proof2-2}, and \eqref{eq:wp}, we obtain
\begin{align}\label{eq:proof2-6}
    &\frac{1}{2}\big(\left\|\mathbf{u}^{n+1}\right\|^{2}+\left\|2 \mathbf{u}^{n+1}-\mathbf{u}^{n}\right\|^{2}\big)
    -\frac{1}{2}\big(\left\|\mathbf{u}^{n}\right\|^{2}-\left\|2 \mathbf{u}^{n}-\mathbf{u}^{n-1}\right\|^{2}\big)+\frac{2}{3}\Delta t^2\left\|\nabla H^{n+1}\right\|^{2}-\frac{2}{3}\Delta t^2\left\|\nabla H^{n}\right\|^{2} \notag\\
  &\quad +\nu \Delta t\left\|\omega^{n+1}\right\|^{2}-\nu \Delta t\left\|\omega^{n}\right\|^{2} +\frac{1}{2}(|r^{n+1}|^2+|2r^{n+1}-r^{n}|^2)-\frac{1}{2}(|r^{n}|^2-|2r^{n}-r^{n-1}|^2) \notag\\
      & = -\Delta t\nu \|\nabla \tilde{\mb{u}}^{n+1}\|^{2} - \Delta t |\xi^{n+1}|^2\int_{\O}\sum_{i=1}^{m}D_ic_i^{n+1}|\nabla\bar{\mu}_{i}^{n+1}|^2
   d\x
      -\frac{3}{2}\left\|\tilde{\mathbf{u}}^{n+1}-\mathbf{u}^{n+1}\right\|^{2} \notag\\
    &\qquad  -\frac{1}{2}\left\|\mathbf{u}^{n+1}-2 \mathbf{u}^{n}+\mathbf{u}^{n-1}\right\|^{2}-
     \nu \Delta t\left\|\nabla \times \tilde{\mathbf{u}}^{n+1}\right\|^{2}-\frac{1}{2}|r^{n+1}-2r^n+r^{n-1}|^2.
\end{align}
This gives the energy dissipation law \eqref{eq:3-8eng} and \eqref{eq:3-8-2nd}.\\
For \textbf{Scheme2b},  denote
\bex&\label{barp}
\bar{p}^{n}:=p^{n} +\nu \nabla\cdot\tilde{\mb u}^{n}.
\eex
Taking the $L^{2}$ inner product of equation \eqref{2nd-4b} with $ 2\Delta t \tilde{\mb u}^{n+1}$ gives the counterpart of
\eqref{eq:proof2-1}:
  \begin{equation}\label{eq:proof2-1b}
    \begin{split}
    &\big(3\tilde{\mb{u}}^{n+1}-4\mb{u}^{n}+\mb{u}^{n-1}, \tilde{\mb u}^{n+1}\big)+2\Delta t\big(\nabla \bar{p}^n,\tilde{\mb u}^{n+1}\big)
     \\
      &\hspace{2cm} +2\Delta t
    \xi^{n+1}\Big({\mb u}^n\cdot\nabla{\mb u}^n+\big(\sum_{i=1}^{m}z_ic_i^{n+1}\big)\nabla\bar{\Phi}^{n+1},\tilde{\mb{u}}^{n+1}\Big) = -2\Delta t\nu \|\nabla \tilde{\mb{u}}^{n+1}\|^{2}.
    \end{split} \tag{3.16'}
\end{equation}
The equality \eqref{eq:proof2-2} remains true for \textbf{Scheme2b} by using the fact
\begin{align*}
\left(3 \mathbf{u}^{n+1}-4 \mathbf{u}^{n}+\mathbf{u}^{n-1}, \tilde{\mathbf{u}}^{n+1}-\mathbf{u}^{n+1}\right)
 &=\big(3 \mathbf{u}^{n+1}-4 \mathbf{u}^{n}+\mathbf{u}^{n-1}, \frac{2\Delta t}{3} \nabla(\bar{p}^{n+1}-\bar{p}^{n})\big) \\
& = \big(\nabla\cdot\big(3 \mathbf{u}^{n+1}-4 \mathbf{u}^{n}+\mathbf{u}^{n-1}\big), \frac{2\Delta t}{3} (\bar{p}^{n+1}-\bar{p}^{n})\big) =0.
\end{align*}
Equation \eqref{2nd-5b} can be rewritten as
\bex
  \sqrt{\frac{3}{2}}\mathbf{u}^{n+1}+\sqrt{\frac{2}{3}}\Delta t \nabla  \bar{p}^{n+1}=\sqrt{\frac{3}{2}}\tilde{\mathbf{u}}^{n+1}+\sqrt{\frac{2}{3}}\Delta t \nabla\bar{p}^{n}.
\eex
Now, taking the inner product of the above equality with itself and noticing that $(\nabla \bar{p}^{n+1},\mb u^{n+1})=0$, we obtain
\begin{equation}\label{eq:proof2-5b}
    \frac{3}{2}\left\|\mathbf{u}^{n+1}\right\|^{2}+\frac{2}{3}\Delta t^{2}\left\|\nabla \bar{p}^{n+1}\right\|^{2}=\frac{3}{2}\left\|\tilde{\mathbf{u}}^{n+1}\right\|^{2}+\frac{2}{3}\Delta t^2\left\|\nabla \bar{p}^{n}\right\|^{2}+2 \Delta t\left(\nabla \bar{p}^{n}, \tilde{\mathbf{u}}^{n+1}\right).\tag{3.19'}
\end{equation}
Finally, multiplying \eqref{2nd-8} by $2\Delta t r^{n+1}$, and using \eqref{eq:proof2-1b}, \eqref{eq:proof2-2}, and \eqref{eq:proof2-5b}, we get
\begin{align}\label{eq:proof2-6b}
    &\frac{1}{2}\big(\left\|\mathbf{u}^{n+1}\right\|^{2}+\left\|2 \mathbf{u}^{n+1}-\mathbf{u}^{n}\right\|^{2}\big)
    -\frac{1}{2}\big(\left\|\mathbf{u}^{n}\right\|^{2}+\left\|2 \mathbf{u}^{n}-\mathbf{u}^{n-1}\right\|^{2}\big)+\frac{2}{3}\Delta t^{2}\left\|\nabla \bar{p}^{n+1}\right\|^{2}\notag\\
  &\hspace{2cm} -\frac{2}{3}\Delta t^2\left\|\nabla \bar{p}^{n}\right\|^{2} +\frac{1}{2}(|r^{n+1}|^2+|2r^{n+1}-r^{n}|^2)-\frac{1}{2}(|r^{n}|^2-|2r^{n}-r^{n-1}|^2) \notag\\
      &= -2\Delta t\nu \|\nabla \tilde{\mb{u}}^{n+1}\|^{2} - \Delta t |\xi^{n+1}|^2\int_{\O}\sum_{i=1}^{m}D_ic_i^{n+1}|\nabla\bar{\mu}_{i}^{n+1}|^2
   d\x
      -\frac{3}{2}\left\|\tilde{\mathbf{u}}^{n+1}-\mathbf{u}^{n+1}\right\|^{2} \notag\\
    &\hspace{5cm}  -\frac{1}{2}\left\|\mathbf{u}^{n+1}-2 \mathbf{u}^{n}+\mathbf{u}^{n-1}\right\|^{2}-\frac{1}{2}|r^{n+1}-2r^n+r^{n-1}|^2.
\end{align}
The right sides of \eqref{eq:proof2-6} and \eqref{eq:proof2-6b} are both non-positive, which directly leads to the energy dissipation law \eqref{eq:3-8eng}.
\\
The boundedness of $\|\u^{n}\|,\ |r^{n}|$, and $|\xi^{n}|$ follows directly from the
boundedness of $\mathcal{E}^{n},\ \forall n\geq 0$, as shown in \eqref{eq:3-8eng}.
The proof is completed.
\end{proof}

\section{Full Discretization and Implementation}

In this section, we consider a spectral method for the spatial discretization, and analyze the stability of the
full discrete problems.

\subsection{Full Discretization} To fix the idea,
we take $\Omega=(-1,1)^2$. Let $I\!\!\!\!P_N(\O)$ be the space of polynomials of degree $\leq N$ with respect to each variable in $\O$. We introduce the following approximation spaces:
  \begin{align*}
  &X_N=I\!\!\!\!P_N(\O),\quad {\mb X}_N=X_N^d,\\
  &X_N^0= H_0^1(\O)\cap I\!\!\!\!P_N(\O)=
  \Big\{ v\in I\!\!\!\!P_N(\O): v\big|_{\dO} =0 \Big\},\quad {\mb X}_N^0=(X_N^0)^d,\\
  & M_{N}=\Big\{ v\in I\!\!\!\!P_{N}(\O):  \int_\O vd\x=0\Big\}.
  \end{align*}
Let $(\cdot, \cdot)_N$ denotes the discrete inner product using the $N+1$-point
Legendre-Gauss-Lobatto quadrature. Let $\|\cdot\|_{0, N}=(\cdot, \cdot)^{1/2}_N$.

\textbf{Scheme1-SM}:
The full discretization of the first-order scheme in time/spectral method in space reads:
with the solutions known at the previous time steps,
find $\sigma_{i,N}^{n+1}, c_{i,N}^{n+1}\in X_N, i=1,...,m; \ \bar\Phi^{n+1}_N, \Phi^{n+1}_N\in M_N;\ \tilde{\mb u}^{n+1}_N\in {\mb X}^0_N;\ {\mb u}^{n+1}_N\in {\mb X}_N;\ \psi^{n+1}_N,p^{n+1}_N\in M_{N-2}$, such that
\begin{subequations} \label{scheme1b}
  \begin{align}
   \dps & \big(\frac{\sigma_{i,N}^{n+1}-\sigma_{i,N}^{n}}{\Delta t},w_N\big)_N-D_i\big(\nabla\sigma_{i,N}^{n+1},\nabla  w_N\big)_N \notag\\
   &\hspace{1cm} = \big(D_i\big[|\nabla \sigma_{i,N}^{n}|^2+z_i(\nabla\sigma_{i,N}^{n}\cdot\nabla\Phi^{n}_{N}+\Delta\Phi_{N}^{n})\big]-\nabla\cdot(\sigma_i^{n} \mb u^{n}_N),w_N\big)_N, \quad \forall w_N\in X_N,  \label{c1st-1-f} \\
   \dps & \bar{c}_{i,N}^{n+1}=\exp(\sigma_{i,N}^{n+1}), \label{ci1-f}\\
   \dps& \lambda_{i,N}^{n+1}\big( \bar{c}_{i,N}^{n+1},1\big)_N-\big( c_{i,N}^{n},1\big)_N=0, \label{ci2-f}\\
  \dps& c_{i,N}^{n+1}=\lambda_{i,N}^{n+1}\bar{c}_{i,N}^{n+1},\ \ i=1,...,m,\label{ci3-f}   \\
  \dps &\epsilon\big(\nabla\bar{\Phi}^{n+1}_N,\nabla q_N\big)_N= \big(\sum_{i=1}^{m}z_ic_{i,N}^{n+1},q_N\big)_N, \quad \forall q_N \in M_N, \label{c1st-3-f}\\
  &\big(\frac{\tilde{\mb{u}}^{n+1}_N-\mb{u}^{n}_N}{\Delta t},\mb v_N\big)_N+
    \frac{r^{n+1}_N}{\sqrt{\bar{E}_{npp,N}^{n+1}+C_0}}\Big(({\mb u}^n_N\cdot\nabla){\mb u}^n_N+(\sum_{i=1}^{m}z_ic_{i,N}^{n+1})\nabla\bar{\Phi}^{n+1}_N,\mb v_N\Big)_N \notag\\
    & \quad\qquad\qquad\qquad\qquad\qquad\qquad +\nu\big(\nabla \tilde{\mb{u}}^{n+1}_N,\nabla \mb v_N\big)_N+\big(\nabla p^{n}_N,\mb v_N\big)_N=0, \quad \forall \mb v_N\in {\mb X}_N^0,\label{c1st-4-f}\\
    &\left(\nabla \psi_{N}^{n+1}, \nabla q_{N}\right)_N=\frac{1}{ \Delta t}\left(\tilde{\mb u}_{N}^{n+1}, \nabla q_{N}\right)_N, \quad \forall q_{N} \in M_{N-2},  \nonumber\\
  &\hspace{6cm} \mb u_N^{n+1}=\tilde{\mb u}_N^{n+1}-\Delta t\nabla \psi_N^{n+1},\quad  p_{N}^{n+1}=\psi_{N}^{n+1}+p_{N}^{n}, \label{c1st-6-f}\\
    &\frac{r^{n+1}_N-{r}^{n}_N}{\Delta t} =-\frac{1}{2\sqrt{\bar{E}_{npp,N}^{n+1}+C_0}}\bigg[ \frac{r^{n+1}_N}{\sqrt{\bar{E}_{npp,N}^{n+1}+C_0}}\sum_{i=1}^{m}D_i
    \big(c_{i,N}^{n+1},|\nabla\bar{\mu}_{i,N}^{n+1}|^2\big)_N
    - \notag \\
   & \qquad\qquad\qquad\qquad\qquad \qquad \big((\sum_{i=1}^{m}z_ic_{i,N}^{n+1})\nabla\bar{\Phi}^{n+1}_N,\tilde{\mb{u}}^{n+1}_N \big)_N - \big((\mb{u}^n_N\cdot\nabla)\mb{u}^n_N, \tilde{\mb{u}}^{n+1}_N\big)_N\bigg], \label{c1st-8-f2}\\
    & \Phi^{n+1}_N=\frac{r^{n+1}_N}{\sqrt{\bar{E}_{npp,N}^{n+1}+C_0}} \bar{\Phi}^{n+1}_N,\label{c1st-9-f}
  \end{align}
\end{subequations}
where $\bar{\mu}_{i,N}^{n+1}$ and $\bar{E}_{npp,N}^{n+1}$ are defined by
    \begin{align*}
  \dps& \bar{\mu}_{i,N}^{n+1}=\log c_{i,N}^{n+1} + z_i\bar{\Phi}^{n+1}_N ,\ i=1,...,m,\\
  \dps& \bar{E}_{npp,N}^{n+1}=E_{npp}[\{c_{i,N}^{n+1}\},\bar{\Phi}^{n+1}_N].
    \end{align*}

\textbf{Scheme2-SM}: The spectral method in space for the second-order semi-discrete problem \eqref{scheme2} reads: find $\sigma_{i,N}^{n+1}, c_{i,N}^{n+1}\in X_N, i=1,...,m; \ \bar\Phi^{n+1}_N, \Phi^{n+1}_N\in M_N;\ \tilde{\mb u}^{n+1}_N\in {\mb X}^0_N;\ {\mb u}^{n+1}_N\in {\mb X}_N;\ \psi^{n+1}_N\in M_{N-2}$ and $p^{n+1}_N$, such that
\begin{subequations} \label{scheme2-full}
  \begin{align}
   \dps & \big(\frac{3\sigma_{i,N}^{n+1}-4\sigma_{i,N}^{n}+\sigma_{i,N}^{n-1}}{2\Delta t},w_N\big)_N-D_i\big(\nabla\sigma_{i,N}^{n+1},\nabla  w_N\big)_N = \big(D_i\big[|\nabla \sigma_{i,N}^{*,n+1}|^2+ \notag\\
   &\qquad\quad
  z_i(\nabla\sigma_{i,N}^{*,n+1}\cdot\nabla\Phi^{*,n+1}_{N}+\Delta\Phi_{N}^{*,n+1})\big]-\nabla\cdot(\sigma_{i,N}^{*,n+1} \mb u^{*,n+1}_N),w_N\big)_N, \quad \forall w_N\in X_N,  \label{2nd-1-f} \\
   \dps & \bar{c}_{i,N}^{n+1}=\exp(\sigma_{i,N}^{n+1}), \label{2nd-2-f}\\
   \dps& \lambda_{i,N}^{n+1}\big( \bar{c}_{i,N}^{n+1},1\big)_N-\big( c_{i,N}^{n},1\big)_N=0, \label{2nd-3-f}\\
  \dps& c_{i,N}^{n+1}=\lambda_{i,N}^{n+1}\bar{c}_{i,N}^{n+1},\ \ i=1,...,m,\label{2nd-4-f}   \\
  \dps &\epsilon\big(\nabla\bar{\Phi}^{n+1}_N,\nabla q_N\big)_N= \big(\sum_{i=1}^{m}z_ic_{i,N}^{n+1},q_N\big)_N, \quad \forall q_N \in M_N, \label{2nd-5-f}\\
  &\big(\frac{3\tilde{\mb{u}}^{n+1}_N-4\mb{u}^{n}_N+\mb{u}^{n-1}_N}{2\Delta t},\mb v_N\big)_N+\nu\big(\nabla \tilde{\mb{u}}^{n+1}_N,\nabla \mb v_N\big)_N+\big(\nabla p^{n}_N ,\mb v_N\big)_N\notag\\
    & \quad +
    \frac{r^{n+1}_N}{\sqrt{\bar{E}_{npp,N}^{n+1}+C_0}}\Big(({\mb u}^{*,n+1}_N\cdot\nabla){\mb u}^{*,n+1}_N+\big(\sum_{i=1}^{m}z_ic_{i,N}^{n+1}\big)\nabla\bar{\Phi}^{n+1}_N,\mb v_N\Big)_N =0, \quad \forall \mb v_N\in {\mb X}_N^0,\label{2nd-6-f}\\
    &\left(\nabla \psi_{N}^{n+1}, \nabla q_{N}\right)_N=\frac{3}{2 \Delta t}\left(\tilde{\mb u}_{N}^{n+1}, \nabla q_{N}\right)_N, \quad \forall q_{N} \in M_{N-2}, \nonumber\\
    &\hspace{3cm}  \mb u_{N}^{n+1}=\tilde{\mb u}_{N}^{n+1}-\frac{2 \Delta t}{3} \nabla \psi_{N}^{n+1} ,\quad p_{N}^{n+1}=\psi_{N}^{n+1}+p_{N}^{n}-\nu\Pi_{N-2}\nabla\cdot\tilde{\mb u}^{n+1}_N,\label{2nd-7-f}\\
    &\frac{3r^{n+1}_N-4r^{n}_N+r^{n-1}_N}{2\Delta t}=-\frac{1}{2\sqrt{\bar{E}_{npp,N}^{n+1}+C_0}}\bigg[ \frac{r^{n+1}_N}{\sqrt{\bar{E}_{npp,N}^{n+1}+C_0}}\sum_{i=1}^{m}D_i
    \big(c_{i,N}^{n+1},|\nabla\bar{\mu}_{i,N}^{n+1}|^2\big)_N
     \notag \\
   & \qquad\qquad\qquad\qquad \qquad -\big((\sum_{i=1}^{m}z_ic_{i,N}^{n+1})\nabla\bar{\Phi}^{n+1}_N,\tilde{\mb{u}}^{n+1}_N \big)_N - \big((\mb{u}^{*,n+1}_N\cdot\nabla)\mb{u}^{*,n+1}_N, \tilde{\mb{u}}^{n+1}_N\big)_N\bigg], \label{2nd-8-f}
   \\ & \Phi^{n+1}_N=\frac{r^{n+1}_N}{\sqrt{\bar{E}_{npp,N}^{n+1}+C_0}} \bar{\Phi}^{n+1}_N,
   \label{2nd-9-f}
  \end{align}
\end{subequations}
where $\Pi_{N-2}$ is the $L^{2}$-projection operator onto $M_{N-2}$.

\textbf{Scheme2b-SM}: same as \textbf{Scheme2-SM} with the exception of \eqref{2nd-6-f} and \eqref{2nd-7-f}, which are replaced by
  \begin{align}
  &\big(\frac{3\tilde{\mb{u}}^{n+1}_N-4\mb{u}^{n}_N+\mb{u}^{n-1}_N}{2\Delta t},\mb v_N\big)_N+\nu\big(\nabla \tilde{\mb{u}}^{n+1}_N,\nabla \mb v_N\big)_N+\big(\nabla( {p}_N^n+\nu\nabla\cdot\tilde{\u}^n_N),\mb v_N\big)_N\notag\\
    & \quad +
    \frac{r^{n+1}_N}{\sqrt{\bar{E}_{npp,N}^{n+1}+C_0}}\Big(({\mb u}^{*,n+1}_N\cdot\nabla){\mb u}^{*,n+1}_N+\big(\sum_{i=1}^{m}z_ic_{i,N}^{n+1}\big)\nabla\bar{\Phi}^{n+1}_N,\mb v_N\Big)_N =0, \quad \forall \mb v_N\in {\mb X}_N^0,\label{2nd-6-f'}\tag{4.2f'}\\
    &\left(\nabla \psi_{N}^{n+1}, \nabla q_{N}\right)_N=\frac{3}{2 \Delta t}\left(\tilde{\mb u}_{N}^{n+1}, \nabla q_{N}\right)_N, \quad \forall q_{N} \in M_{N-2}, \nonumber\\
    &\hspace{1cm}
     \mb u_{N}^{n+1}=\tilde{\mb u}_{N}^{n+1}-\frac{2 \Delta t}{3} \nabla \psi_{N}^{n+1},\quad
     p_{N}^{n+1}={\psi}_N^{n+1}+{p}_N^n+\nu\nabla\cdot\tilde{\u}^n_N-\nu\nabla\cdot\tilde{\mb u}^{n+1}_N
     .\label{2nd-7-f'}\tag{4.2g'}
  \end{align}

It is generally believed that the stability of the full discretization follows directly from the one of the
time semi-discretization once the spatial discretization is of Galerkin-type. However, as we are going to see, while this
is true for the first order scheme \textbf{Scheme1-SM}, the full discrete version of the second order schemes needs
care. In fact we are unable to establish the stability for \textbf{Scheme2-SM}.
The unconditional stability of the schemes \textbf{Scheme1-SM} and \textbf{Scheme2b-SM} is given in the following two theorems.
\begin{thm}\label{thm3}
    Given $\{c_{i,N}^n\}, \Phi^{n}_N, \mb u^{n}_N, p^{n}_N, r^{n}_N$.
    Suppose $c_{i,N}^n>0$, and $( c_{i,N}^n ,1)_N=( c_{i,N}^0,1)_{N}, i=1,...,m$.
Then the solution $(\{c_{i,N}^{n+1}\}$,
$\mb u^{n+1}_N$, $p^{n+1}_N$,$r^{n+1}_N)$ of \textbf{Scheme1-SM} 
satisfies the following properties:
 \begin{enumerate}
   \item[i)] Positivity preserving: $c_{i,N}^{n+1}>0,\ i=1,...,m$.
   \item[ii)] Mass conserving: $(c_{i,N}^{n+1},1)_N=( c_{i,N}^n,1)_N,\ i=1,...,m$.
   \item[iii)] Energy dissipation:
   \begin{equation}\label{eq:2-8eng-f}
    \mathcal{E}^{n+1}_N-\mathcal{E}^n_N\leq 0,
\end{equation}
where
\bex
  \mathcal{E}^{n+1}_N:=\|{\mb u}^{n+1}_N\|^2_{0,N}+
  \Delta t^2\| \nabla p^{n+1}_N\|^{2}_{0,N} +|r^{n+1}_N|^2.
\eex
\item[iv)] The quantities
$\|u^{n}_N\|_{0,N},\ |r^{n}_N|$, and $|\xi^{n}_N|$ are bounded for all $n\geq 0$.
 \end{enumerate}
\end{thm}

\begin{proof}
Almost exactly as
established in Theorem \ref{thm3} for the time semi-discrete scheme \eqref{scheme1}, the properties given in Theorem \ref{thm3} can be likewise proved.
We present here only the key step:
eq.\eqref{c1st-6-f} gives
$$(\mb u^{n+1}_N,\nabla q_N^{n+1})_N=0, \ \forall q_N\in M_{N-2}.$$
We then derive from the above and \eqref{c1st-6-f} that
\begin{equation*}
    \left\|\mathbf{u}^{n+1}_N\right\|^{2}_{0,N}+\Delta t^2\left\|\nabla p^{n+1}_N\right\|^{2}_{0,N}=\left\|\tilde{\mathbf{u}}^{n+1}_N\right\|^{2}_{0,N}+2 \Delta t\left(\nabla p^{n}_N, \tilde{\mathbf{u}}^{n+1}_N\right)_N+\Delta t^2\left\|\nabla p^{n}_N\right\|^{2}_{0,N}
\end{equation*}
which is the discrete version of \eqref{eq:2-6rr}.
The other details are more or less the same as the semi-discrete version, which are omitted.
\end{proof}

\begin{thm}\label{thm4}
 Given $\{c_{i,N}^k\}, \Phi^{k}_N, \tilde{\u}^{k}_N, \u^{k}_N, {p}^{k}_N, r^{k}_N, k=n,n-1$. Suppose $c_{i,N}^k>0$ and $( c_{i,N}^k,1)_N=( c_{i,N}^0,1)_N,\ i=0,...,m;\ k=n,n-1$.
Then the solution $(\{c_{i,N}^{n+1}\},\tilde{\mb u}^{n+1}_N,\mb u^{n+1}_N, p^{n+1}_N,r^{n+1}_N)$ of the full discrete problem \textbf{Scheme2b-SM} satisfies the following properties:
 \begin{enumerate}
   \item[i)] Positivity preserving: $c_{i,N}^{n+1}>0 \text{ in } \O$.
   \item[ii)] Mass conservation: $(c_{i,N}^{n+1},1)_N=( c_{i,N}^n,1)_N,\ i=1,...,m$.
   \item[iii)] Energy dissipation:
   \begin{equation}\label{eq:3-8eng-f}
    \mathcal{E}^{n+1}_N-\mathcal{E}^n_N\leq 0,
\end{equation}
where
\begin{align*}
  \mathcal{E}^{n+1}_N:= \frac{1}{2}\|{\mb u}^{n+1}_N\|^2_{0,N}+\frac{1}{2}\|2{\mb u}^{n+1}_N-{\mb u}^{n}_N\|^2_{0,N}+
  \frac{2}{3}\Delta t^2\| \nabla (p^{n+1}_N+\nu \nabla\cdot\tilde{\mb u}^{n+1}_N)\|^{2}_{0,N}\notag\\ +\frac{1}{2}|r^{n+1}_N|^2+\frac{1}{2}|2r^{n+1}_N-r^n_N|^2.
\end{align*}
\item[iv)] The quantities
$\|\u^{n}_N\|_{0,N},\ |r^{n}_N|$, and $|\xi^{n}_N|$ are bounded for all $n\geq 0$.
 \end{enumerate}
\end{thm}
\begin{proof}
It follows from \eqref{2nd-7-f'} that
$$(\mb u^{n+1}_N,\nabla q_N)_N=0, \ \forall q_N\in M_{N-2}.$$
We denote $\bar{p}^n_N:=p^n_N+\nu\nabla\cdot\tilde{\u}^n_N, n\geq 2$. Let $\bar{p}^1_N=p^1_N\in M_{N-2}$ be the first step pressure solution of \textbf{Scheme1-SM}, then we have
from \eqref{2nd-7-f'} that $\bar{p}^n_N \in M_{N-2}, n \geq 1$, and thus
\bex
  \sqrt{\frac{3}{2}}\mathbf{u}^{n+1}_{N}+\sqrt{\frac{2}{3}}\Delta t \nabla \bar{p}^{n+1}_{N}=\sqrt{\frac{3}{2}}\tilde{\mathbf{u}}^{n+1}_{N}+\sqrt{\frac{2}{3}}\Delta t \nabla \bar{p}^{n}_{N}.
\eex
Now, taking the discrete inner product of the above equality with itself on both sides, we obtain
\bex
    \frac{3}{2}\left\|\mathbf{u}^{n+1}_{N}\right\|^{2}_{0,N}+\frac{2}{3}\Delta t^{2}\left\|\nabla \bar{p}^{n+1}_{N}\right\|^{2}_{0,N}=\frac{3}{2}\left\|\tilde{\mathbf{u}}^{n+1}_{N}\right\|^{2}_{0,N}+\frac{2}{3}\Delta t^2\left\|\nabla \bar{p}^{n}_{N}\right\|^{2}_{0,N}+2 \Delta t\left(\nabla \bar{p}^{n}_{N}, \tilde{\mathbf{u}}^{n+1}_{N}\right)_{N}.
\eex
This is the discrete analog of \eqref{eq:proof2-5b}. The remaining of the proof is similar to the semi-discrete case \textbf{Scheme2b}. We omit the details.
\end{proof}

However we got stuck in proving the stability of \textbf{Scheme2-SM}.
The obstacle comes from the discrete counterpart of \eqref{eq:proof2-s}.
Let
\be\label{abcd0}
\omega_N^{n+1}=\omega_N^n+\Pi_{N-2} \nabla\cdot\tilde{\u}^{n+1}_N,\ n\ge 0; \ \omega_N^0=0.
\ee
Then $p^n_N+\nu\omega_N^n\in M_{N-2}$. It follows from \eqref{2nd-7-f}:
$$(\mb u^{n+1}_N,\nabla q_N)_N=0, \ \forall q_N\in M_{N-2},$$
and
\begin{equation}\label{ddc2}
\begin{aligned}
    &\frac{3}{2}\left\|\mathbf{u}^{n+1}_N\right\|^{2}_{0,N}+\frac{2}{3}\Delta t^{2}\left\|\nabla (p^{n+1}_N+\nu\omega^{n+1}_N)\right\|^{2}_{0,N}\\
    &=\frac{3}{2}\left\|\tilde{\mathbf{u}}^{n+1}_N\right\|^{2}_{0,N}+\frac{2}{3}\Delta t^2\left\|\nabla ((p^{n}_N+\nu\omega^{n}_N))\right\|^{2}_{0,N}+2 \Delta t\left( \tilde{\mathbf{u}}^{n+1}_N,\nabla p^{n}_N\right)+2\Delta t\nu(\tilde{\mathbf{u}}^{n+1}_N,\nabla\omega^{n}_N)_N.
\end{aligned}\end{equation}
Although we can use the discrete Stokes formula
$
\left(\tilde{\mathbf{u}}^{n+1}_N, \nabla\omega^{n}_N\right)_N
=
- \left(\nabla\cdot\tilde{\mathbf{u}}^{n+1}_N,\omega^{n}_N\right)_N
$
to treat the last term of \eqref{ddc2}.
In order to get the discrete counterpart of \eqref{eq:proof2-s}, i.e.,
\bex
 2 \Delta t\nu \left(\tilde{\mathbf{u}}^{n+1}_N, \nabla\omega^{n}_N\right)_N
= \nu \Delta t\left\|w^{n}_N\right\|_N^{2}- \nu\Delta t\left\|w^{n+1}_N\right\|_N^{2}
+ \nu\Delta t\left\|\nabla \cdot \tilde{\mb u}^{n+1}_N\right\|_N^{2},
\eex
we would need the relationship
\bex
\nabla\cdot\tilde{\mathbf{u}}^{n+1}_N = w^{n+1}_N-w^{n}_N,
\eex
which is not true according to \eqref{abcd0}.

\subsection{Implementation}
Implementation of the first-order scheme is essentially the same as that of the second-order schemes,
we only present the implementation of the second-order full discretization \textbf{Scheme2b-SM} as follows.
\begin{description}
  \item[Step 1] Solve the elliptic problem: Find $\sigma_{i,N}^{n+1}\in {X}_{N}$ such that
  \bex
  \frac{3}{2\Delta t}( \sigma_{i,N}^{n+1},v_N )_N + D_i( \nabla\sigma_{i,N}^{n+1},\nabla v_N)_N= (\tilde{g}_{i}^{*,n+1}, v_N)_N \quad \forall v_N \in {X}_{N},\quad i=1,...,m,
\eex
  with $\dps \tilde{g}_{i}^{*,n+1}=\frac{4\sigma_{i,N}^{n}-\sigma_{i,N}^{n-1}}{2\Delta t} + g_i(\sigma_{i,N}^{*,n+1},\Phi^{*,n+1}_{N},\mb u^{*,n+1}_{N})$. Then compute $c^{n+1}_{i,N}$ by
\bex
\bar{c}_{i,N}^{n+1}=\exp(\sigma_{i,N}^{n+1}), \ \ \
\lambda_i^{n+1}=\frac{(c_{i,N}^{n},1)_N}{( \bar{c}_{i,N}^{n+1},1)_N}, \ \ \
c_{i,N}^{n+1}=\lambda_{i,N}^{n+1}\bar{c}_{i,N}^{n+1}.
\eex
    \item[Step 2] Find $\bar{\Phi}^{n+1}_N\in M_N$   such that
  \bex
    &(\nabla\bar{\Phi}^{n+1}_N ,\nabla q_N)_N= \frac{1}{\epsilon}(\sum_{i=1}^{m} z_ic_{i,N}^{n+1},q_N )_N, \quad \forall q_N\in M_{N}.
  \eex
\item[Step 3] Find $\tilde{\mb u}_{1,N}^{n+1}, \tilde{\mb u}_{2,N}^{n+1} \in {\mb X}_N^0$ such that
\bex
   &&\frac{3}{2\Delta t}(\tilde{\mathbf{u}}_{1,N}^{n+1},\mb v_N )_N + \nu( \nabla\tilde{\mathbf{u}}_{1,N}^{n+1},\nabla \mb v_N )_N=  (\tilde{\f}^{n+1}_N,\mb v_N)_N, \quad \forall \mb v_N \in {\mb X}_N^0, \\
   && \frac{3}{2\Delta t}(\tilde{\mathbf{u}}_{2,N}^{n+1},\mb v_N )_N + \nu( \nabla\tilde{\mathbf{u}}_{2,N}^{n+1},\nabla \mb v_N )_N=  -({\mb w}_N,\mb v_N )_N, \quad \forall \mb v_N \in {\mb X}_N^0,
 \eex
  where
  \begin{align*}
   &\tilde{\f}^{n+1}_N=
  \frac{1}{2\Delta t}(4\mb{u}^n_{N}-\mb{u}^{n-1}_{N})-\nabla \bar{p}^n_{N} ,\\
   & \mb w^{n+1}_N={\mb u}^{*,n+1}_N\cdot\nabla{\mb u}^{*,n+1}_N+(\sum_{i=1}^{m}z_ic_{i,N}^{n+1})\nabla\bar{\Phi}^{n+1}_N.
   \end{align*}
    \item[Step 4] Compute $\xi^{n+1}_N$: once $\tilde{\mathbf{u}}_{i,N}^{n+1}(i=1,2)$ are known, we determine explicitly $\xi^{n+1}_N$ from \eqref{2nd-8-f} as follows:
     \bex
      \xi^{n+1}_N=\frac{ 2r^n_N-\frac{1}{2}r^{n-1}_{N}+ \frac{\Delta t}{ 2\sqrt{\bar{E}^{n+1}_{npp,N}+C_0}} (\mb w^{n+1}_{N},\tilde{\mb u}^{n+1}_{1,N})_N }{\frac{3}{2}\sqrt{\bar{E}^{n+1}_{npp,N}+C_0}+\frac{\Delta t}{ 2\sqrt{\bar{E}^{n+1}_{npp,N}+C_0}}  \Big[( \sum_{i=1}^{m}D_i( c^{n+1}_{i,N},|\bar{\mu}^{n+1}_{i,N}|^2)_N-({\mb w}_{N},\tilde{\mb u}^{n+1}_{2,N})_N\Big]   }.
      \eex
Then compute  $r^{n+1}_N$, $\Phi^{n+1}_N$, and $\tilde{\mb u}^{n+1}_N$ by
    \begin{align*}
    & r^{n+1}_N=\xi^{n+1}_N \sqrt{\bar{E}^{n+1}_{npp,N}+C_0},\\
    & \Phi^{n+1}_N= \xi^{n+1}_N \bar{\Phi}^{n+1}_N, \\
   &\tilde{\mb
   u}^{n+1}_N=\tilde{\mathbf{u}}_{1,N}^{n+1}+\xi^{n+1}_N\tilde{\mathbf{u}}_{2,N}^{n+1}.
 \end{align*}
 \item[Step 5] Find $\psi_{N}^{n+1} \in M_{N-2}$ such that
\bex
\left(\nabla \psi_{N}^{n+1}, \nabla q_{N}\right)=\frac{3}{2 \Delta t}\left(\tilde{\mb u}_{N}^{n+1}, \nabla q_{N}\right), \quad \forall q_{N} \in M_{N-2}.
\eex
Then update the velocity and pressure by
\begin{align*}
&\mb u_{N}^{n+1}=\tilde{\mb u}_{N}^{n+1}-\frac{2 \Delta t}{3} \nabla \psi_{N}^{n+1}, \\
&\bar{p}_{N}^{n+1}=\psi_{N}^{n+1}+\bar{p}_{N}^{n}, \\
&p^{n+1}_N=\bar{p}^{n+1}_{N}-\nu\nabla\cdot\tilde{\mb u}^{n+1}_N.
\end{align*}
\end{description}

The above algorithm shows that the computational complexity of the proposed method is equal to
solving several decoupled linear elliptic equations with constant coefficient at each time step.
In the spatial discretization, we use the Legendre modal basis \cite{Shen1994EfficientSM}, for which fast solvers exist
for elliptic equations with constant coefficients in a rectangular domain.

\section{Numerical Results}

In this section, we present several numerical examples to validate the proposed method.
We first present the numerical results for the equation with two species to examine the accuracy and stability
of the schemes.
We will also investigate the convergence order, mass conservation, positivity preserving, and energy dissipation.
Then, we present an example with three species.
 In all our calculations, we fix $C_0 = 100$.

\subsection{Case with two ions}

\subsubsection*{Example 1}
We employ a manufactured analytic solution to the NSNPP equations to investigate the temporal  and spatial
convergence rates of the developed schemes.
Some suitable forcing terms are added in the equations such that the problem \eqref{nsnpp1}
admits the following exact solution:
\begin{align*}
& \mb u=(\pi{\sin}(2 \pi y ) {\sin}^{2}(\pi x)\sin^2 t,-\pi\sin (2 \pi x) {\sin}^{2} (\pi y ) \sin^2t) \notag\\
& p=\sin( \pi x) \sin (\pi y)\sin^2 t \notag\\
& c_1=1.1+\cos(\pi x)\cos(\pi y)\sin^2 t\\
& c_2=1.1-\cos(\pi x)\cos(\pi y)\sin^2 t \notag\\
& \Phi=\frac{1}{\pi^2}\cos(\pi x)\cos(\pi y)\sin^2 t. \notag
\end{align*}
Set the parameters $z_1 = 1, z_2 = -1, D_1 = D_2 = 1$, and
  $\epsilon = 1,\ \nu=0.1$.

Firstly, we investigate the spatial
convergence order by checking the error behavior of numerical solutions with respect to the polynomial degree.
In Figures \ref{fig:ex1-ord1-space} and \ref{fig:ex1-ord2-space}, we present the errors as a function of the polynomial degree $N$ for two values of $\Delta t: 0.001$ and $0.0001$ at time $T=0.1$ for \textbf{Scheme1-SM} and \textbf{Scheme2-SM}, respectively.
The straight lines error curves
in this semi-log representations indicates the exponential convergence until the temporal error dominates.

In the temporal convergence tests, we fix the spatial polynomial degree $N=64$, which is large enough such that the spatial discretization error is negligible compared to the temporal error.
We present in Figure \ref{fig:ex1-ord1} (\textbf{Scheme1-SM}) and Figure \ref{fig:ex1-ord2} (\textbf{Scheme2-SM}) the $L^2$--errors with respect to $\Delta t$ in log-log scale.
The expected convergence rate in time is clearly observed, as predicted by the theoretical analysis.

\begin{figure}[H]
\centering
\subfigure{
\includegraphics[width=5.6cm]{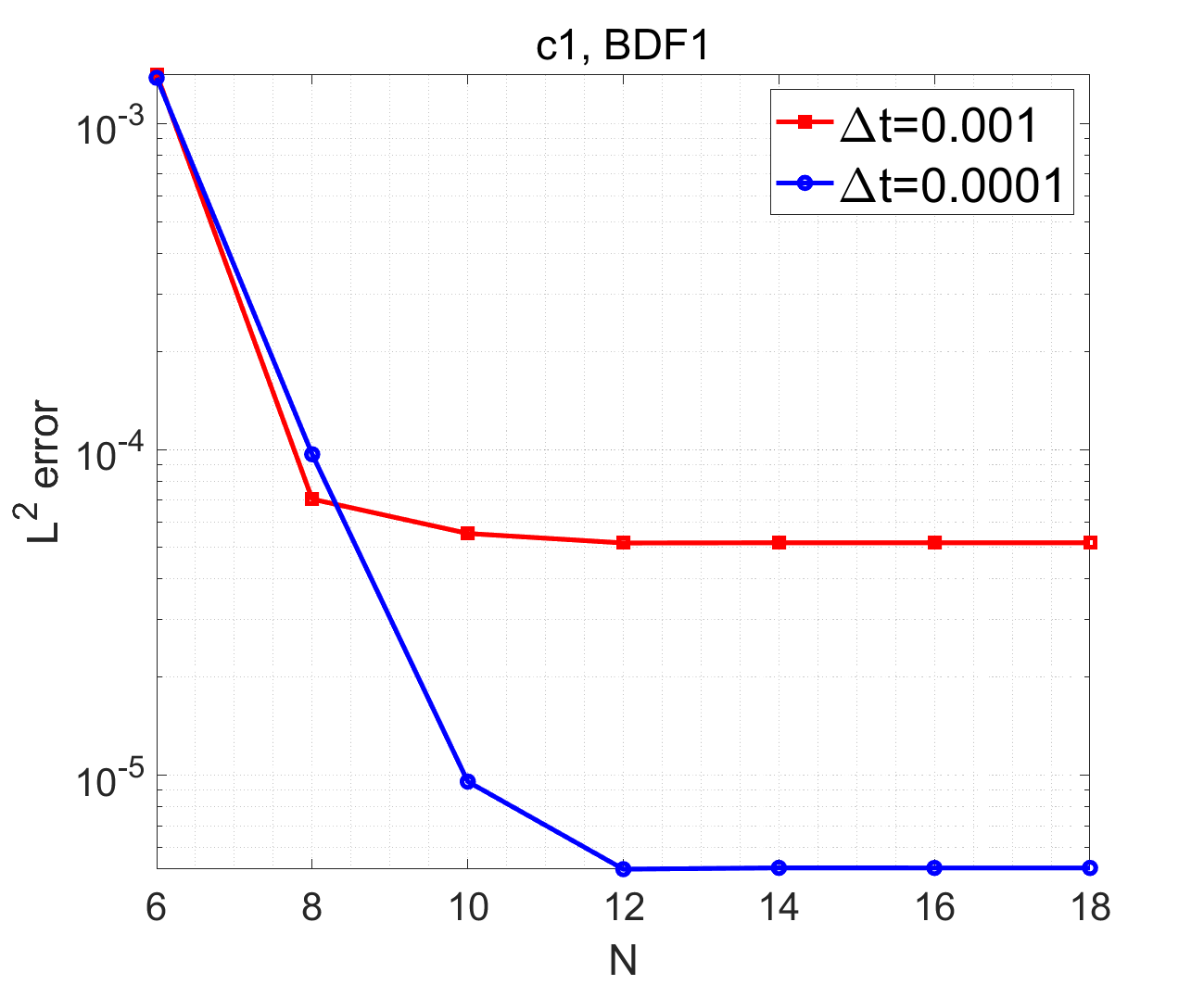}
}
\hspace{-7mm}
\subfigure{
\includegraphics[width=5.6cm]{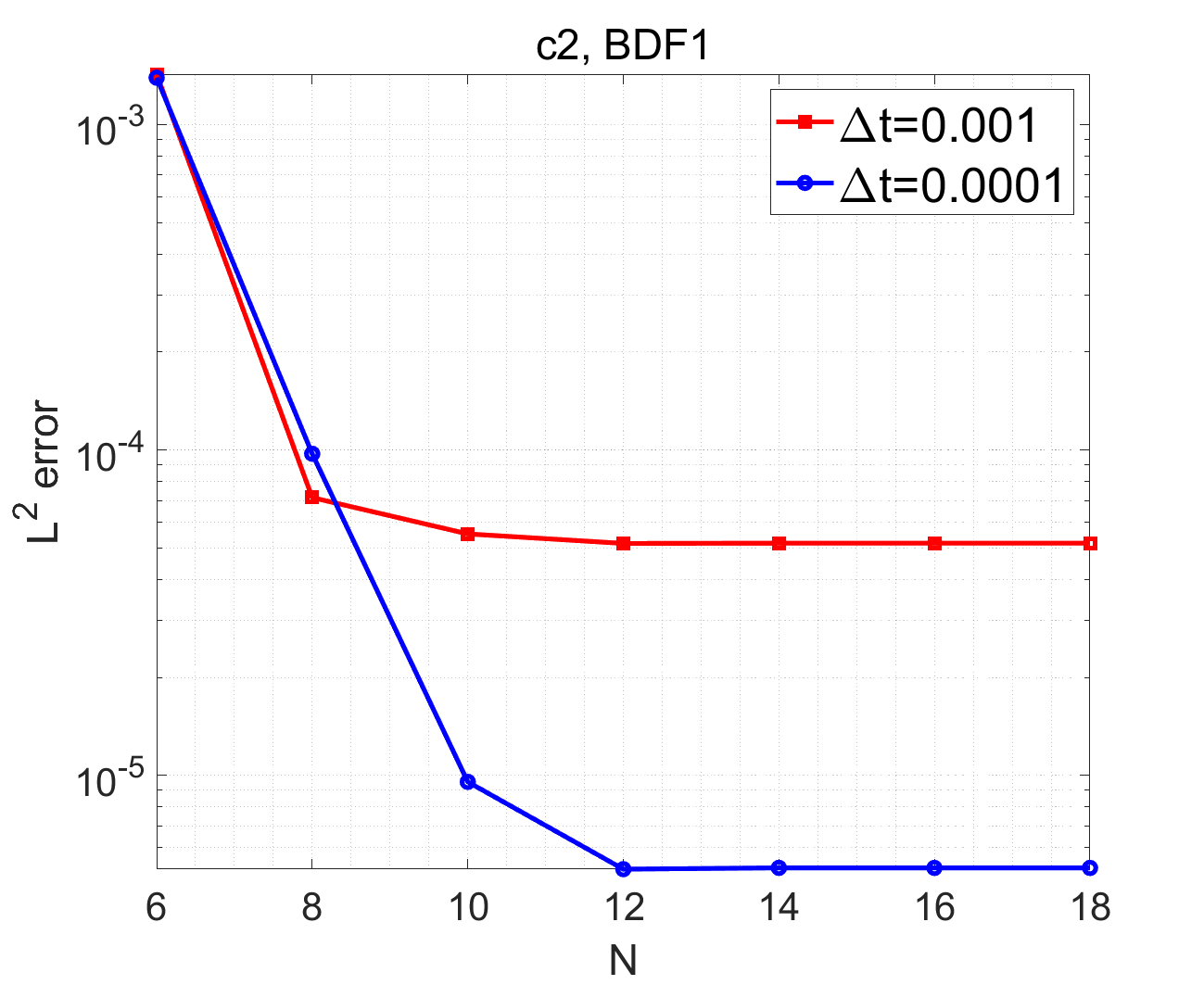}
}
\hspace{-7mm}
\subfigure{
\includegraphics[width=5.6cm]{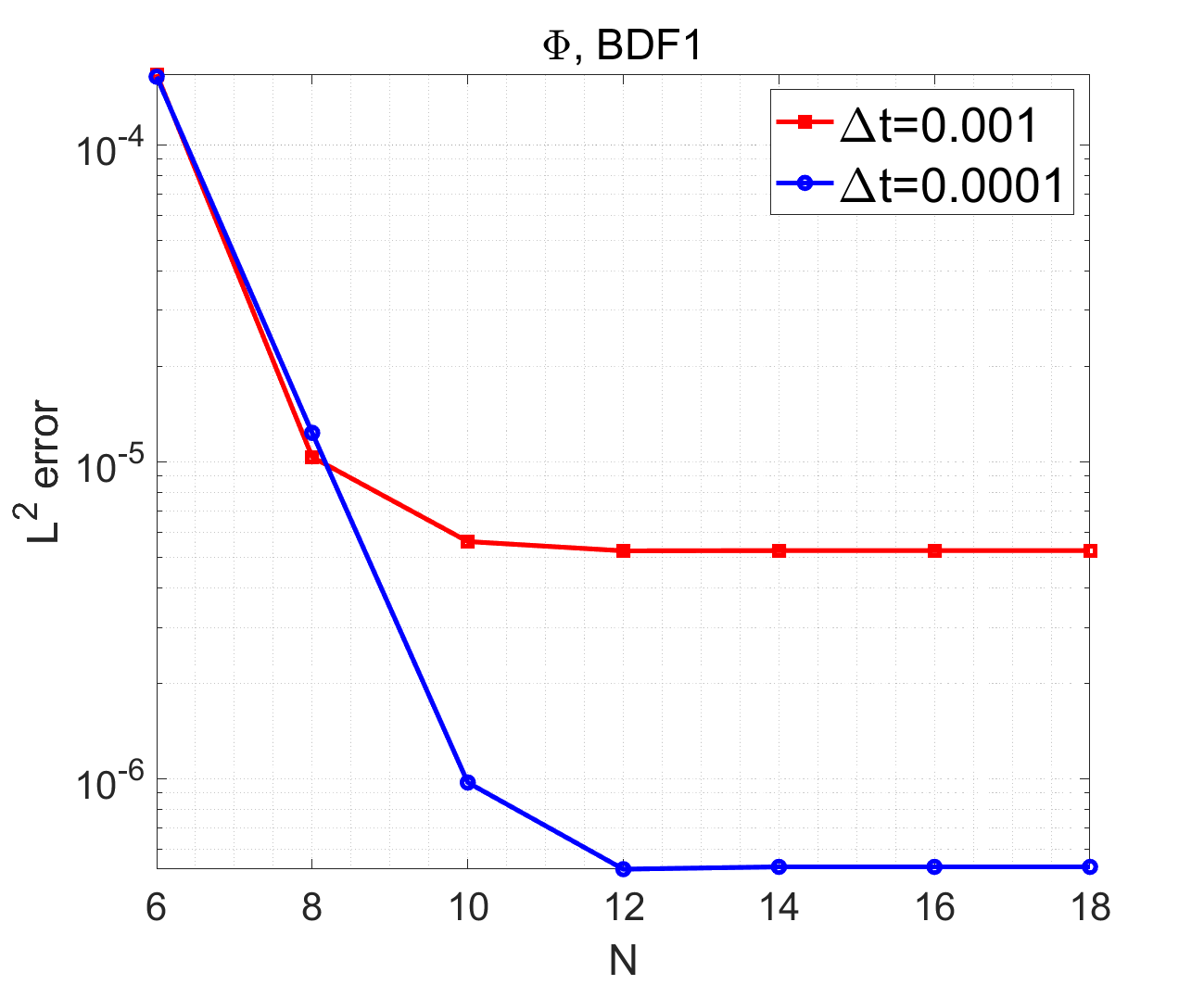}
}
\hspace{-7mm}
\subfigure{
\includegraphics[width=5.6cm]{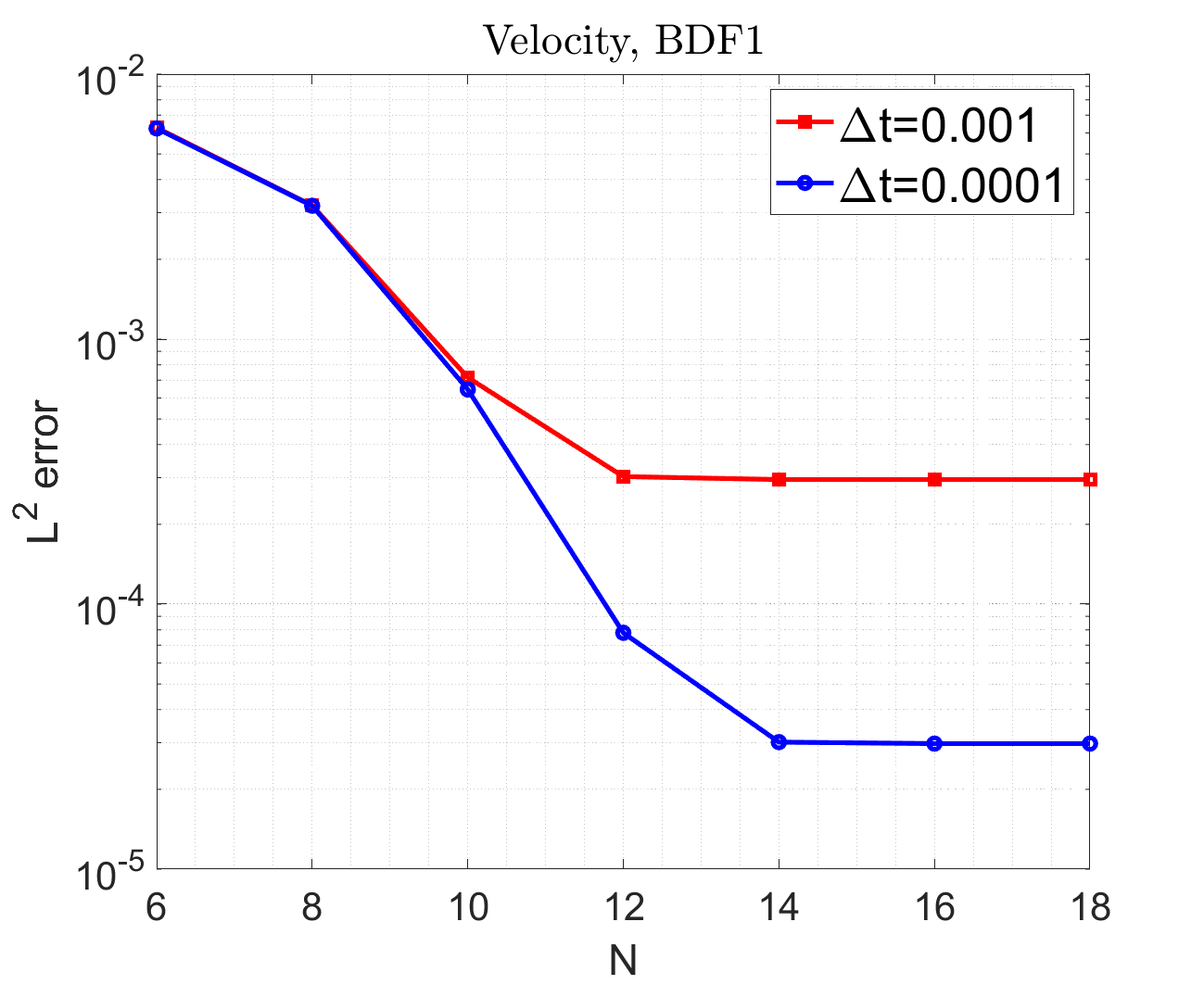}
}
\hspace{-7mm}
\subfigure{
\includegraphics[width=5.6cm]{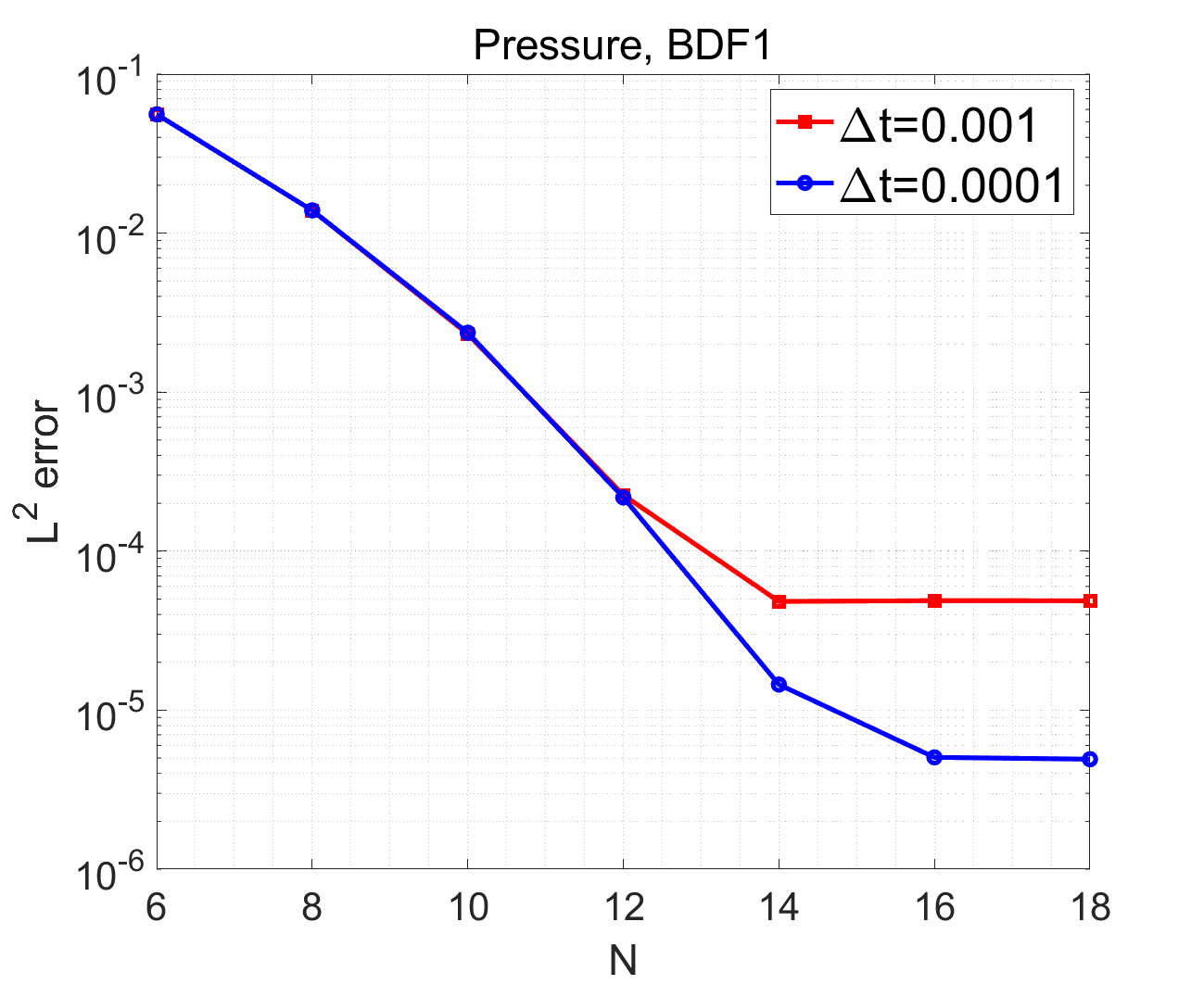}
}
\vspace{-5mm}
\caption{\small (Example 1) Errors in $L^2$-norm as functions of $N$ in semi-log scale for $\Delta t= 0.001,\ 0.0001$ using the 1st-order scheme.}\label{fig:ex1-ord1-space}
\end{figure}
\begin{figure}[H]
\centering
\subfigure{
\includegraphics[width=5.6cm]{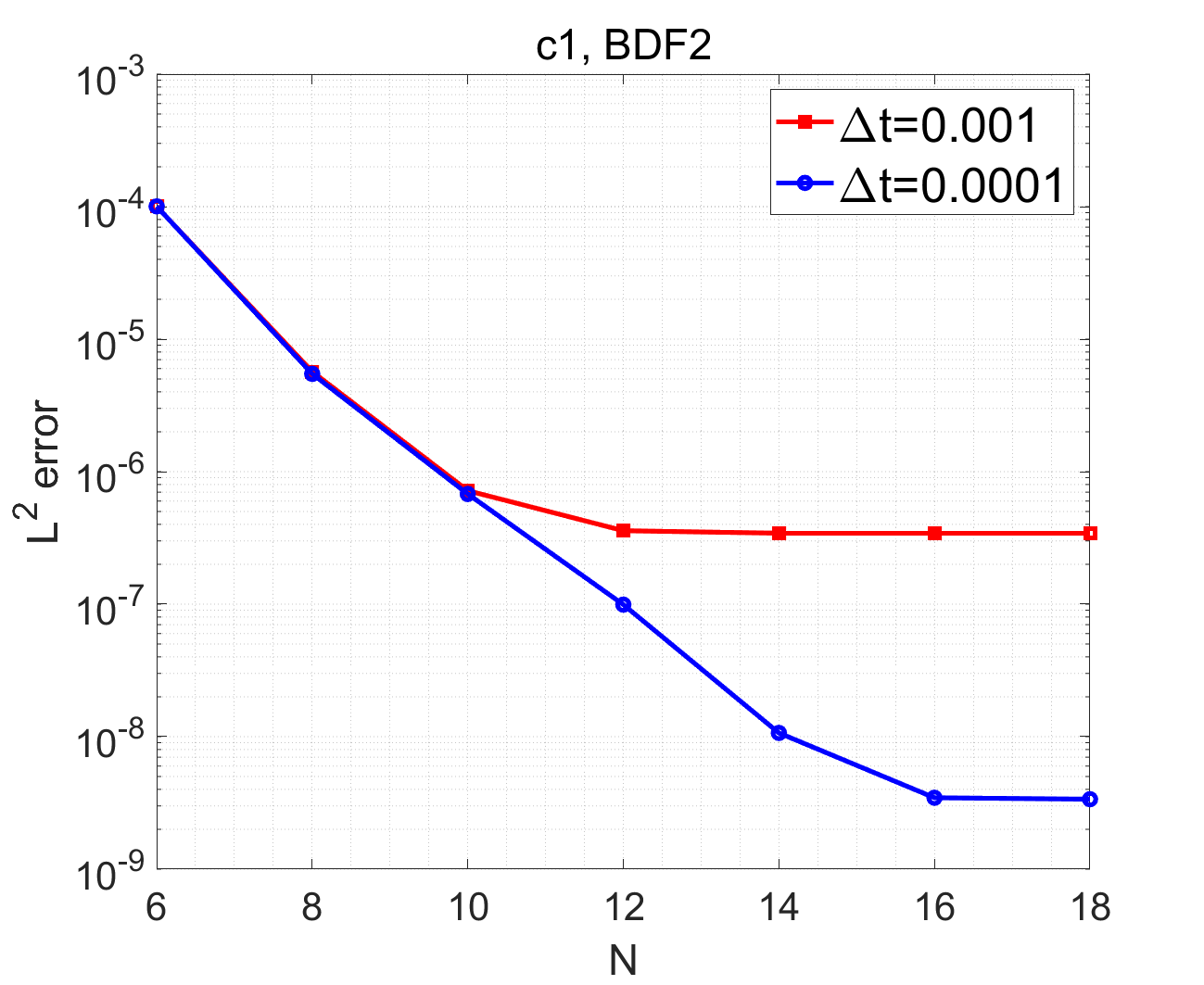}
}
\hspace{-7mm}
\subfigure{
\includegraphics[width=5.6cm]{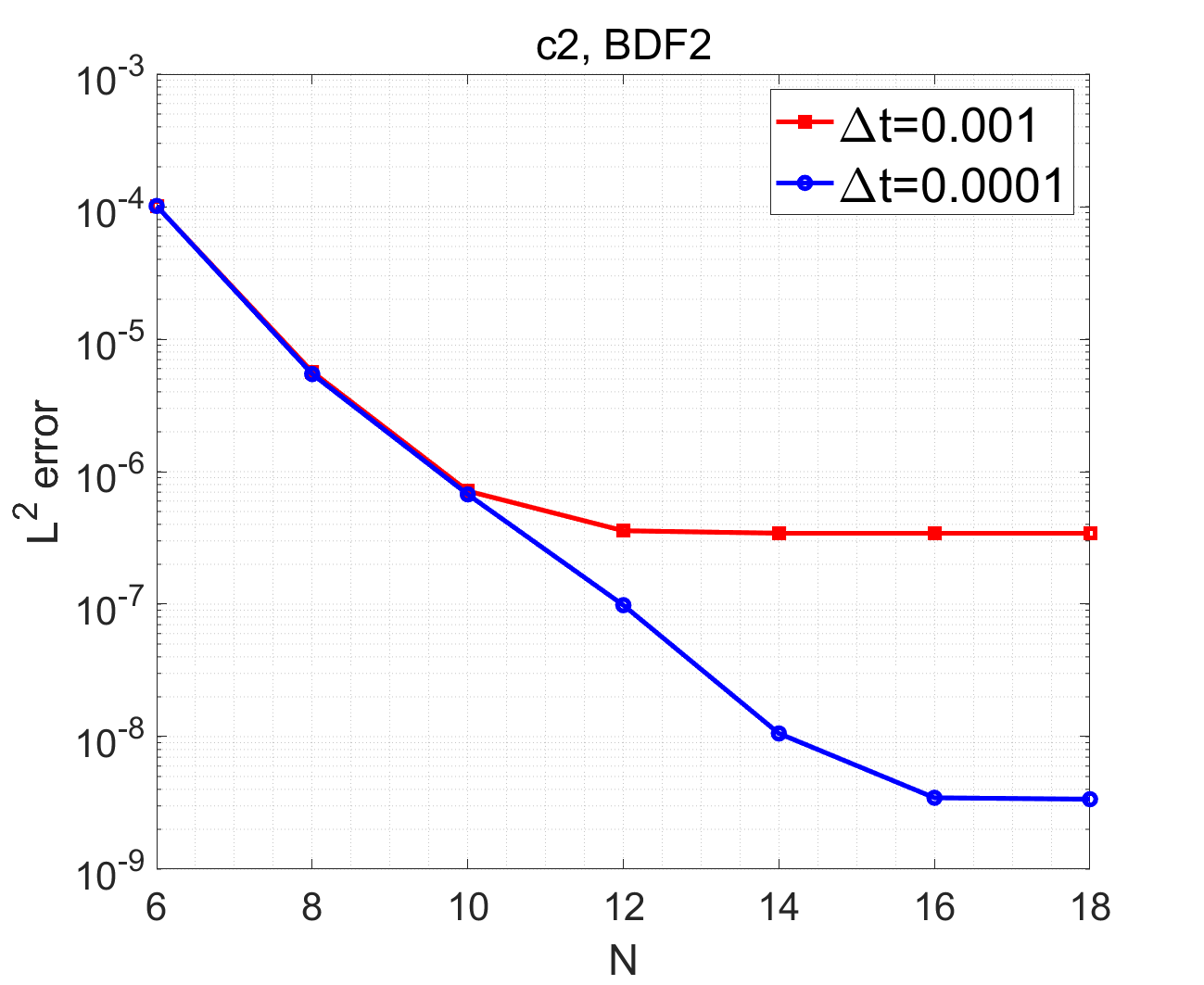}
}
\hspace{-7mm}
\subfigure{
\includegraphics[width=5.6cm]{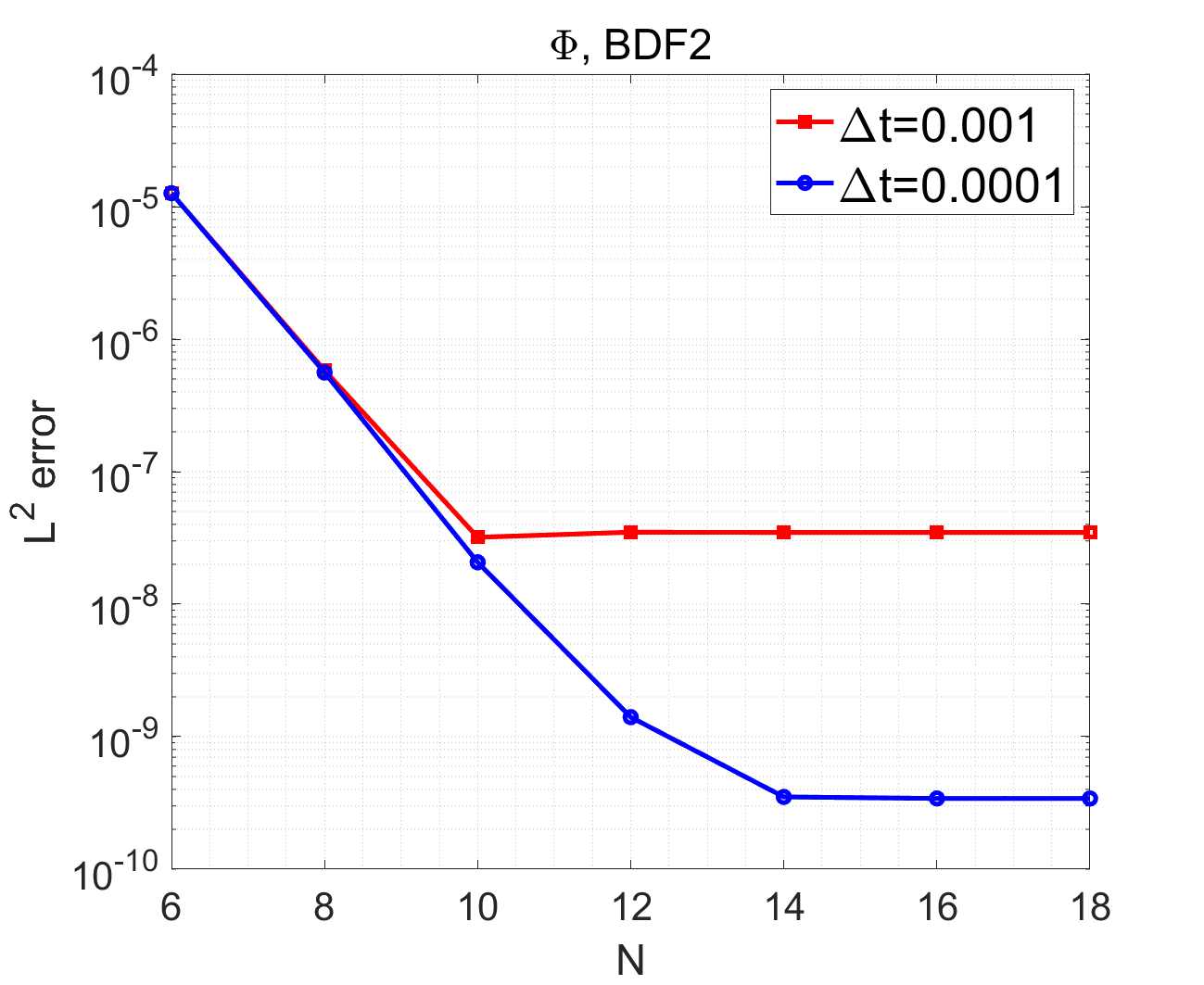}
}
\hspace{-7mm}
\subfigure{
\includegraphics[width=5.6cm]{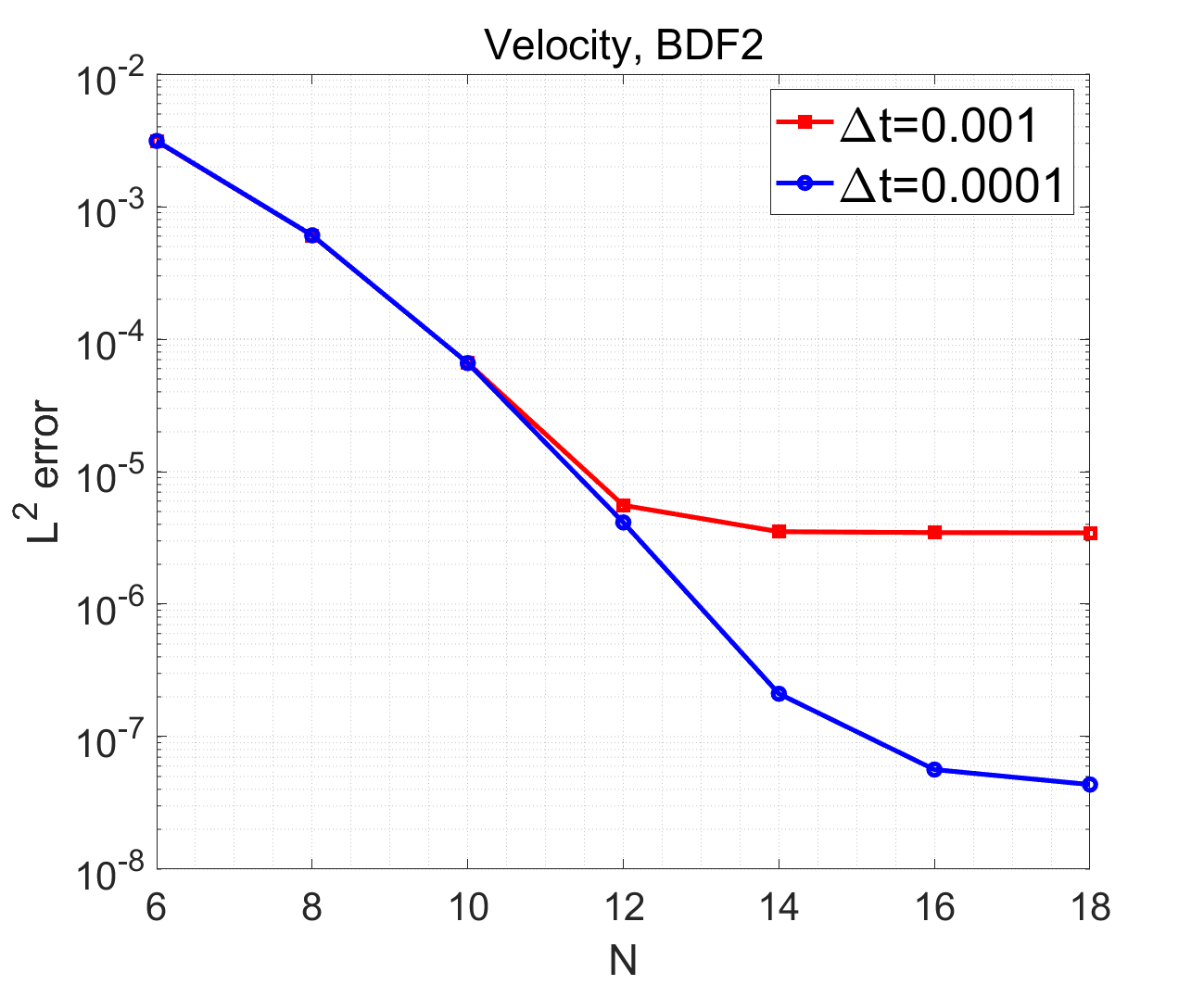}
}
\hspace{-7mm}
\subfigure{
\includegraphics[width=5.6cm]{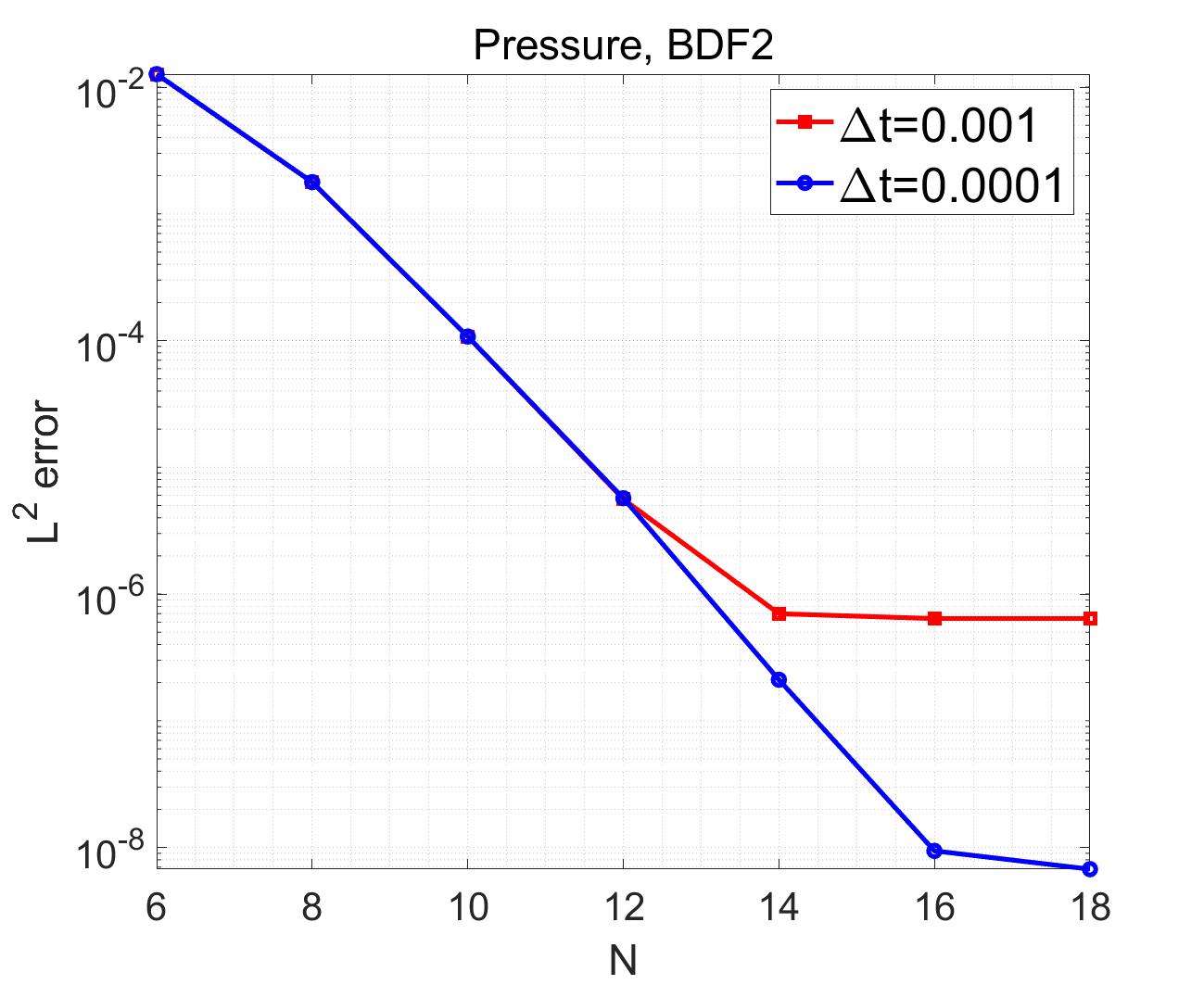}
}
\vspace{-5mm}
\caption{\small  (Example 1) Errors in $L^2$-norm as functions of $N$ in semi-log scale for $\Delta t= 0.001,\ 0.0001$ using the 2nd-order scheme.}\label{fig:ex1-ord2-space}
\end{figure}

\begin{figure}[htbp]
\centering
\subfigure{
\includegraphics[width=5.6cm]{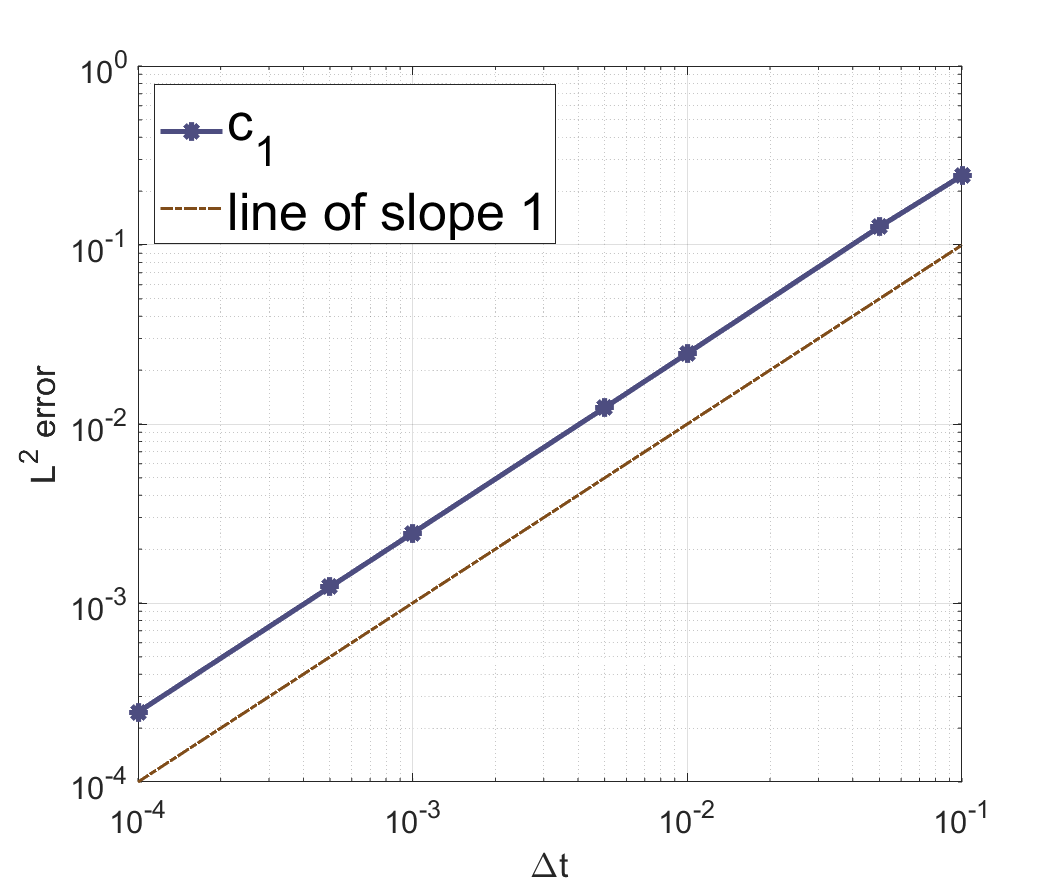}
}
\hspace{-7mm}
\subfigure{
\includegraphics[width=5.6cm]{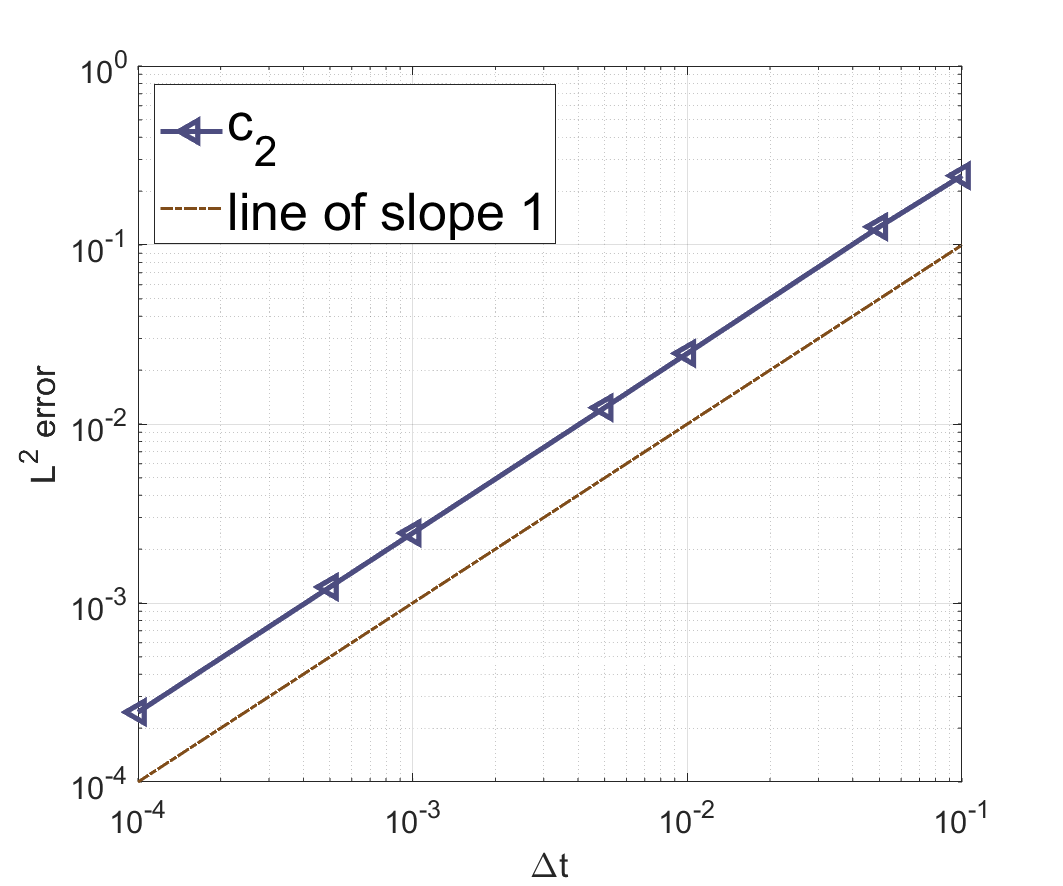}
}
\hspace{-7mm}
\subfigure{
\includegraphics[width=5.6cm]{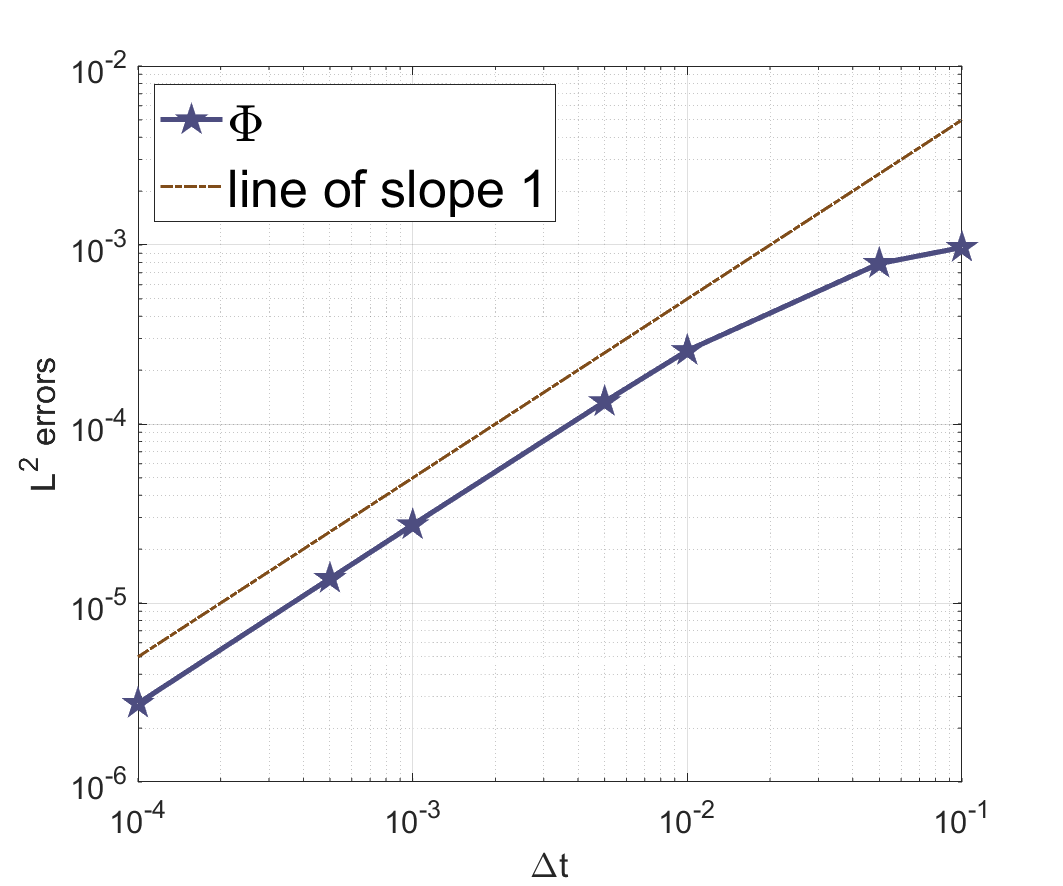}
}
\hspace{-7mm}
\subfigure{
\includegraphics[width=5.6cm]{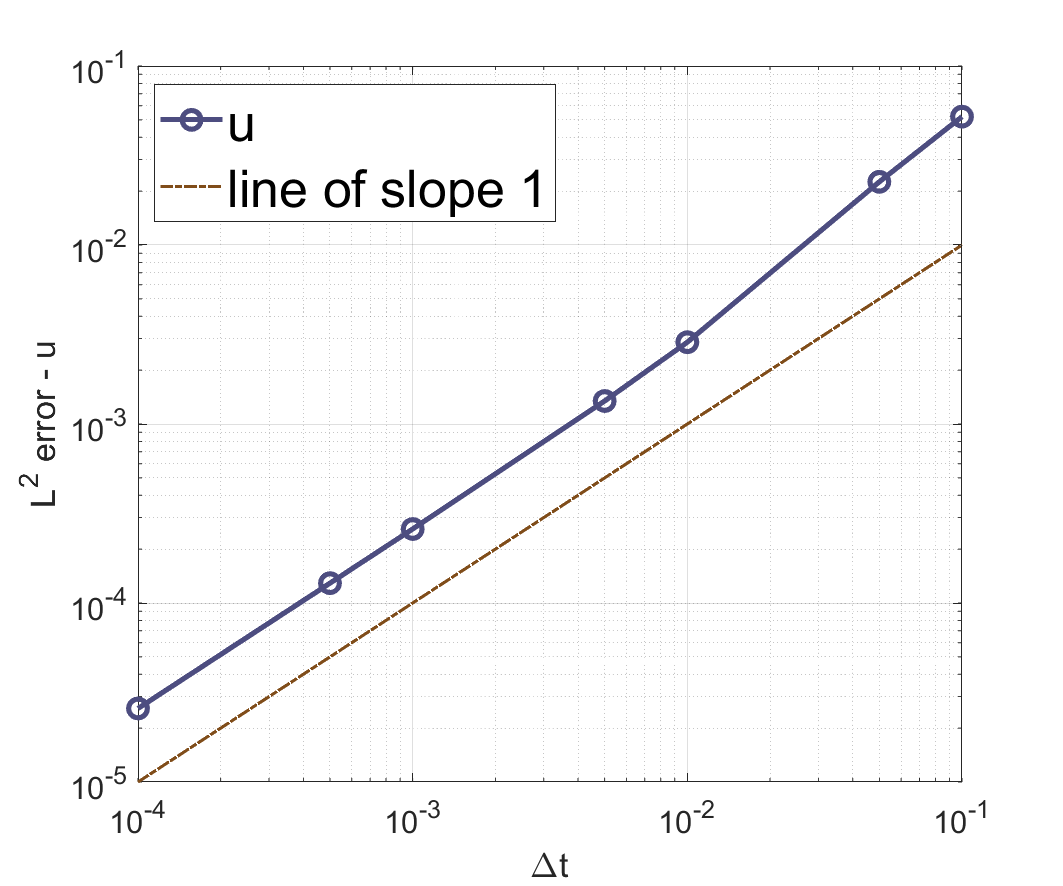}
}
\hspace{-7mm}
\subfigure{
\includegraphics[width=5.6cm]{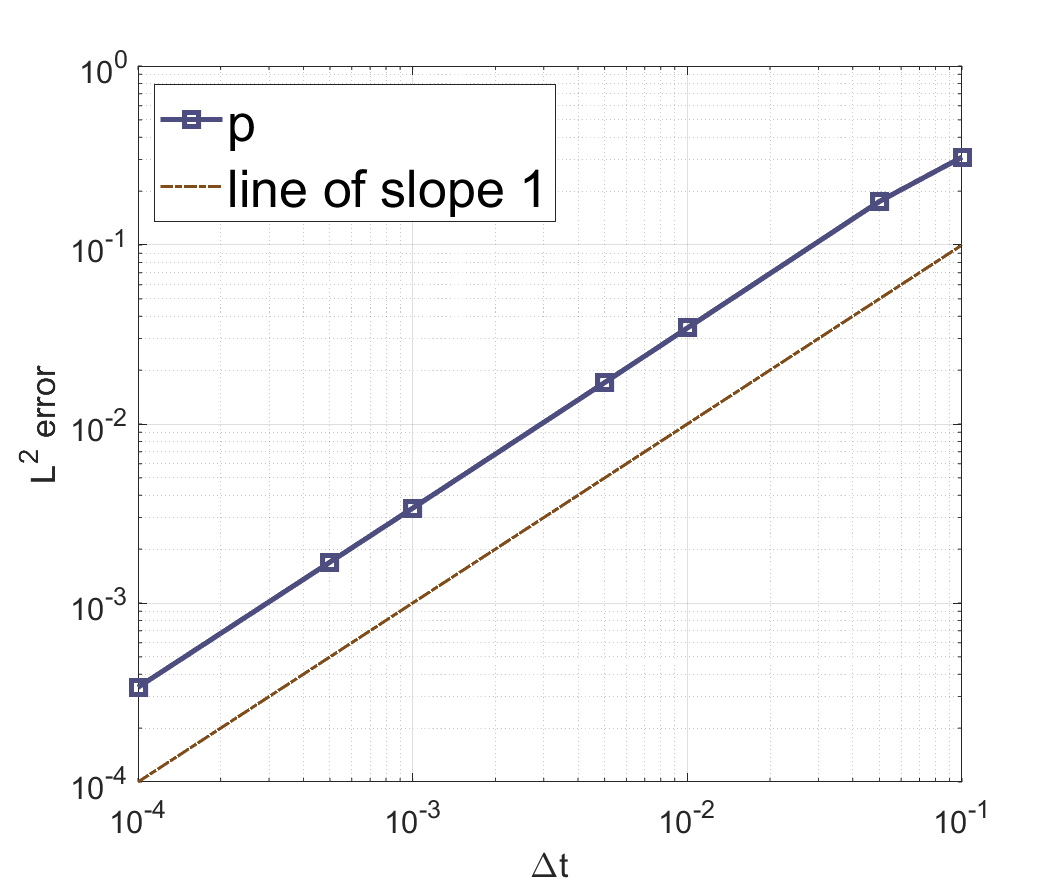}
}
\hspace{-7mm}
\subfigure{
\includegraphics[width=5.6cm]{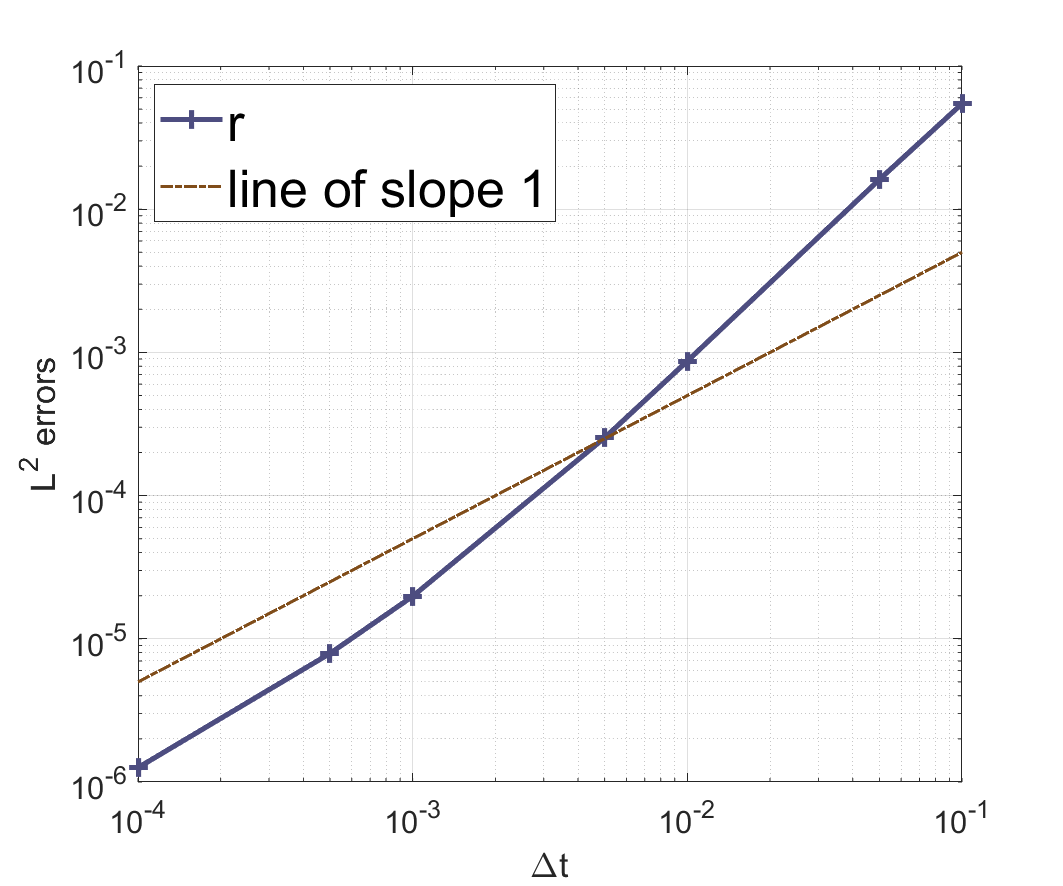}
}
\vspace{-7mm}
\caption{\small (Example 1) $L^2$-errors as a function of $\Delta t$ in log-log scale for the first order scheme.}\label{fig:ex1-ord1}
\end{figure}

 \begin{figure}[htbp]
\centering
\subfigure{
\includegraphics[width=5.6cm]{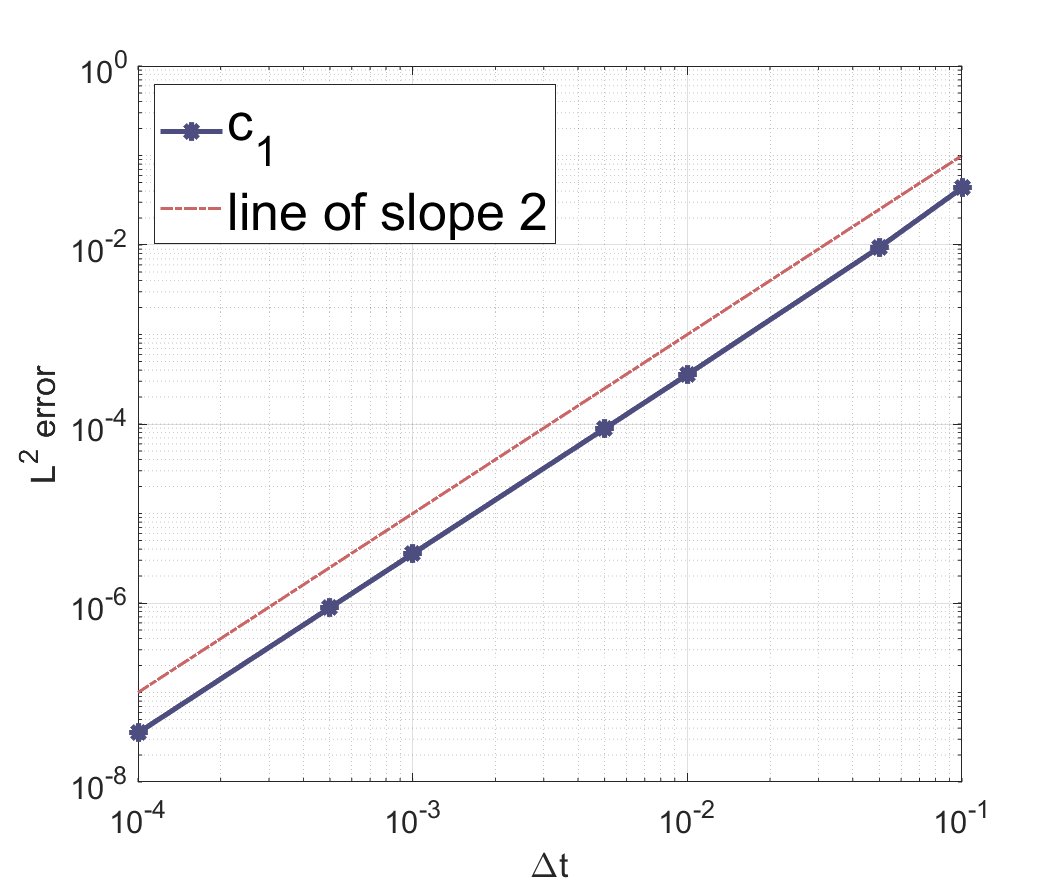}
}
\hspace{-7mm}
\subfigure{
\includegraphics[width=5.6cm]{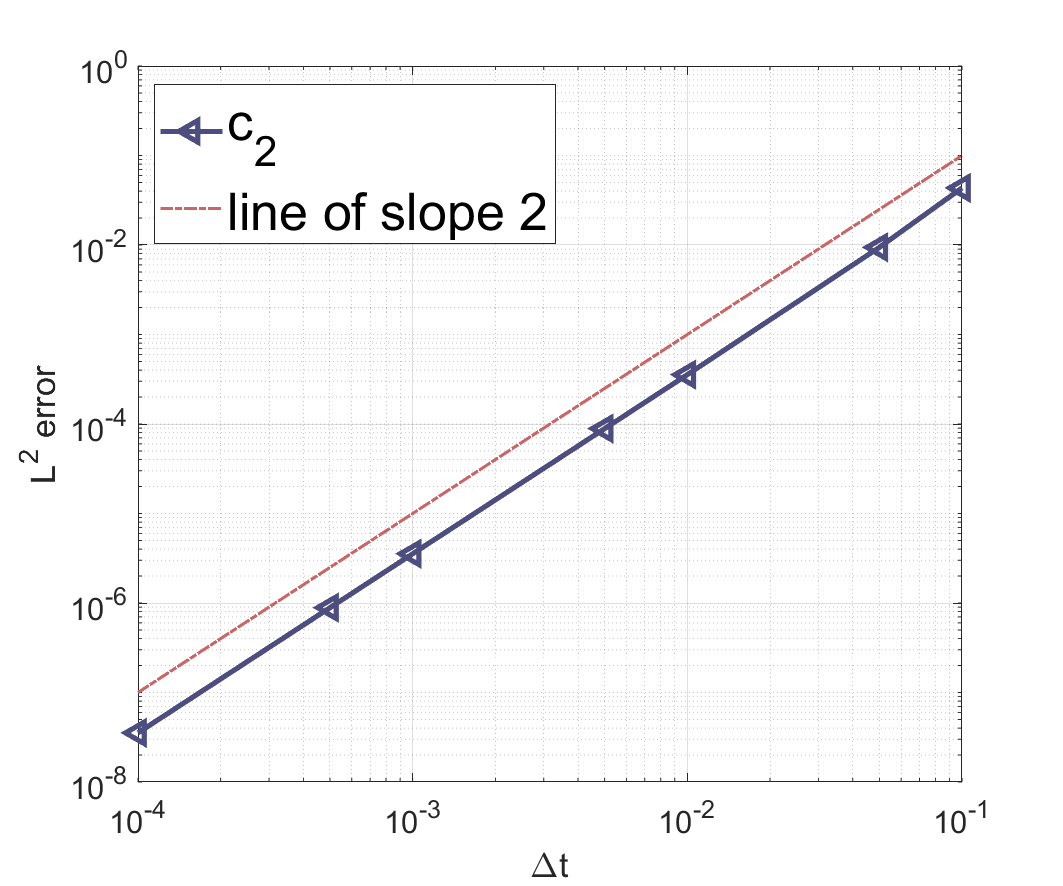}
}
\hspace{-7mm}
\subfigure{
\includegraphics[width=5.6cm]{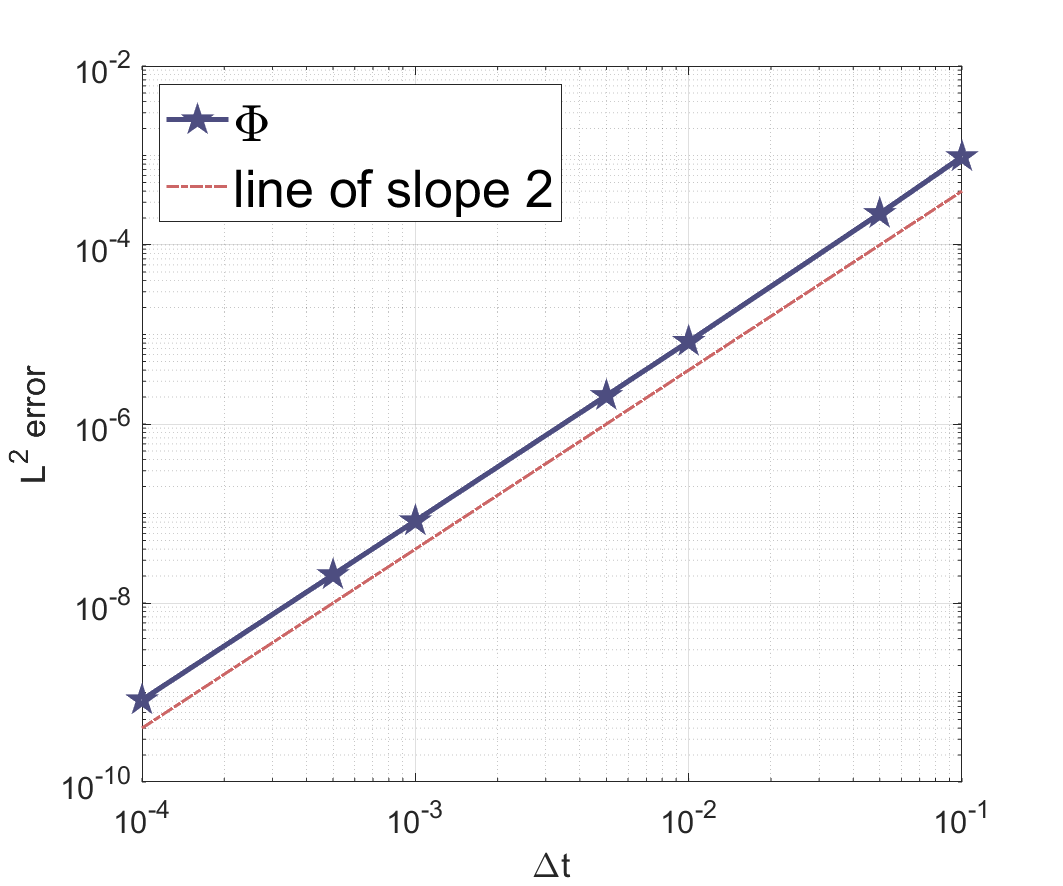}
}
\hspace{-7mm}
\subfigure{
\includegraphics[width=5.6cm]{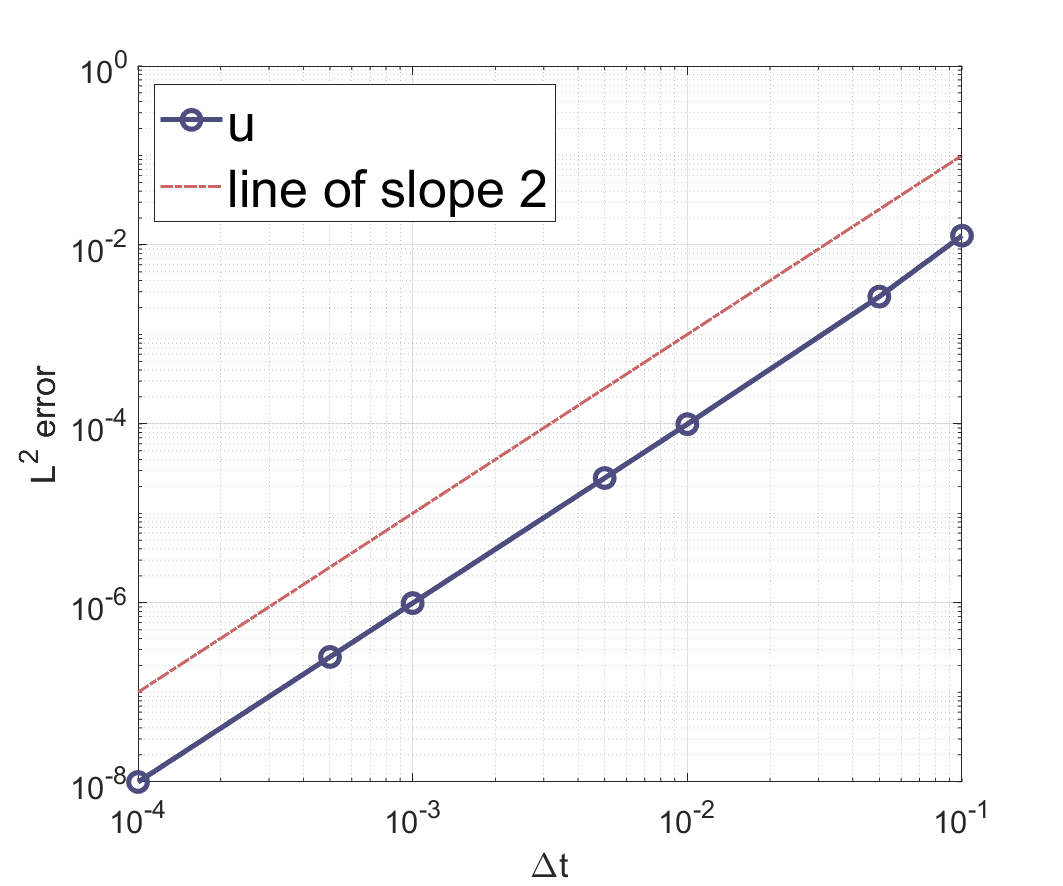}
}
\hspace{-7mm}
\subfigure{
\includegraphics[width=5.6cm]{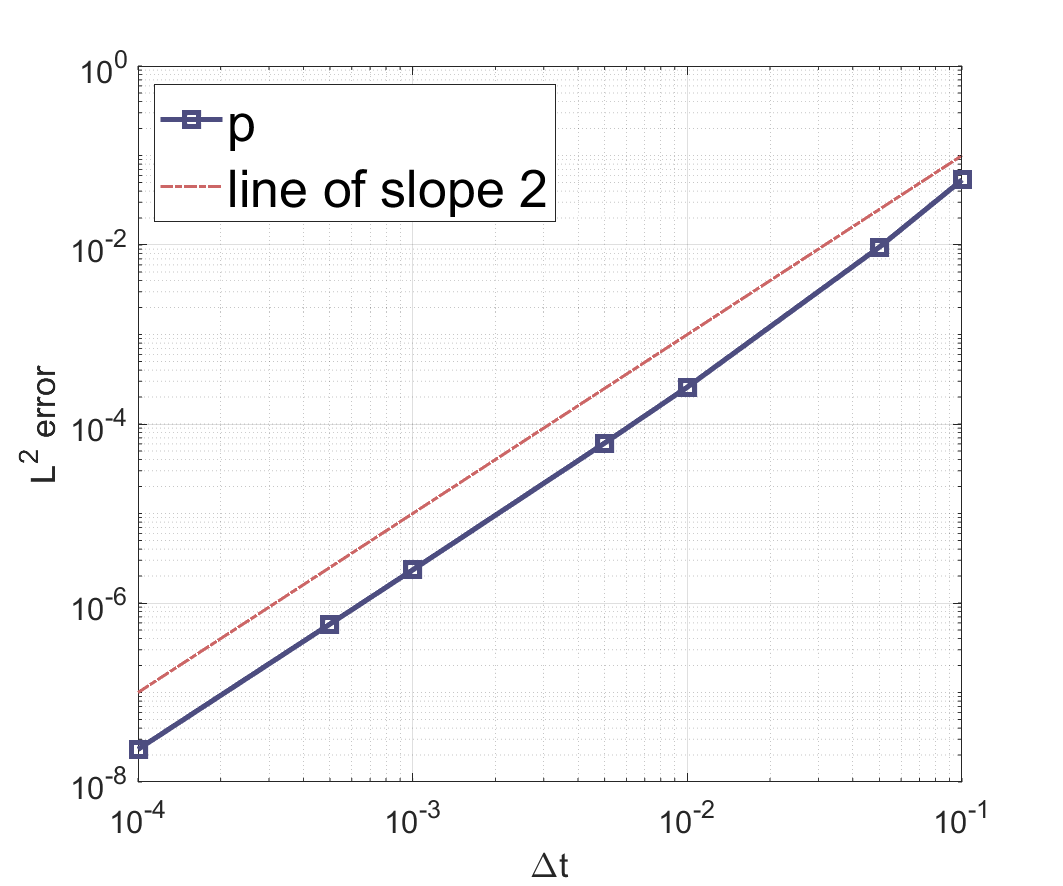}
}
\hspace{-7mm}
\subfigure{
\includegraphics[width=5.6cm]{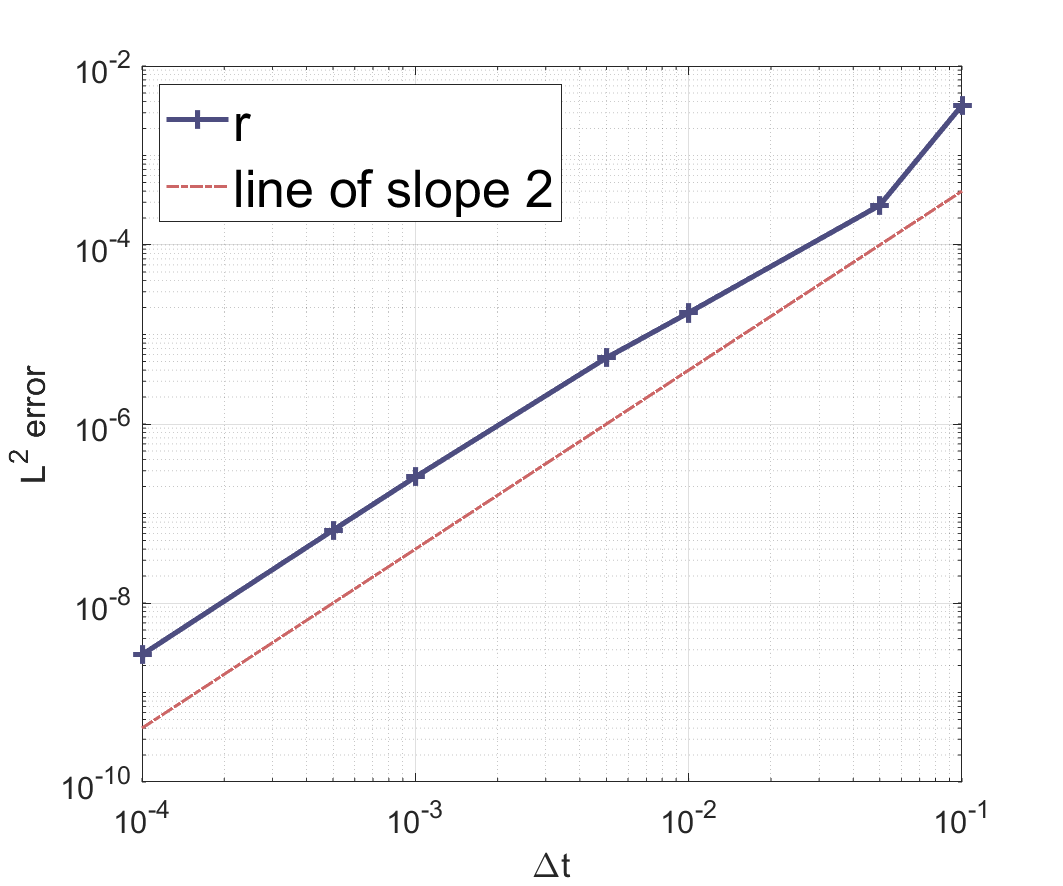}
}
\vspace{-7mm}
\caption{\small (Example 1) L2-errors as functions of $\Delta t$ in log-log scale for the second order scheme.}
\label{fig:ex1-ord2}

\end{figure}

\begin{table}[t]
  \begin{center}
    \caption{Velocity $L^2$-error and pressure $L^2$-error with respect to time step size at $T=1$, computed by \textbf{Scheme2-SM} and \textbf{Scheme2b-SM}, respectively.
    }
    \label{ex1-tab2}
\begin{tabular}{||c||c|c|c|c||}
\hline
\multirow{2}{*}{$\Delta t$} &
\multicolumn{2}{c|}{$L^2$ velocity error} &
\multicolumn{2}{c|}{$L^2$ pressure error} \\
\cline{2-5}
  & \textbf{Scheme2-SM} & \textbf{Scheme2b-SM} & \textbf{Scheme2-SM} & \textbf{Scheme2b-SM} \\
\hline
\hline
$10^{-1}$ & 1.26768E-02 & 1.26935E-02 & 5.33609E-02 & 5.33894E-02 \\
$10^{-2}$ & 9.92445E-05 & 9.90649E-05 & 2.58586E-04 & 2.57952E-04 \\
$10^{-3}$ & 9.92721E-07 & 9.92952E-07 & 2.34312E-06 & 2.34164E-06 \\
$10^{-4}$ & 9.93110E-09 & 9.93277E-09 & 2.31898E-08 & 2.31877E-08 \\
\hline
\end{tabular}
\vskip 5mm
\end{center}
\end{table}

An accuracy comparison between \textbf{Scheme2-SM} and \textbf{Scheme2b-SM} is given in Table \ref{ex1-tab2}. The velocity $L^2$-error and pressure $L^2$-error listed in the table indicates that
the two schemes are almost equal in term of the accuracy. The error comparison for the
other variables (not shown here) has given similar results.

\subsubsection*{Example 2} Set $z_1 = 1, z_2 = -1, D_1 = D_2 = 1$,
  $\epsilon = 1$,  $\nu=0.01$, and $N= 64$.
This example has a purpose to verify the positivity preserving, mass conserving and stability property.
We run \textbf{Scheme2b-SM} with the initial conditions:
\begin{align*}
& \mb u(\x,0)=(\pi{\sin}(2 \pi y ) {\sin}^{2}(\pi x),-\pi\sin (2 \pi x) {\sin}^{2} (\pi y ) ) \notag\\
& c_1(\x,0)=1.1+\cos(\pi x)\cos(\pi y)\\
& c_2(\x,0)=1.1-\cos(\pi x)\cos(\pi y).
\end{align*}

Figure \ref{fig:mass} plots time evolution of the discrete masses $\int_\O c_id\x, \ i=1,2$ computed
by using Scheme2b-SM, which demonstrates the mass conservation property of the scheme.
Figures \ref{fig:evolution-c1} - \ref{fig:evolution-u2} present the snapshots at $t=0,\ 0.1,\ 0.6,\ 1.0$ of the variables $c_1,\ c_2$, and two components of the velocity
$u_1$ and $u_2$, respectively.
It is observed from Figures \ref{fig:evolution-c1} and \ref{fig:evolution-c2} that the
concentrations $\{c_i\}$ preserve the positivity during the time evolution.

\begin{figure}[htbp]
\centering
\subfigure[]{
\includegraphics[width=7cm]{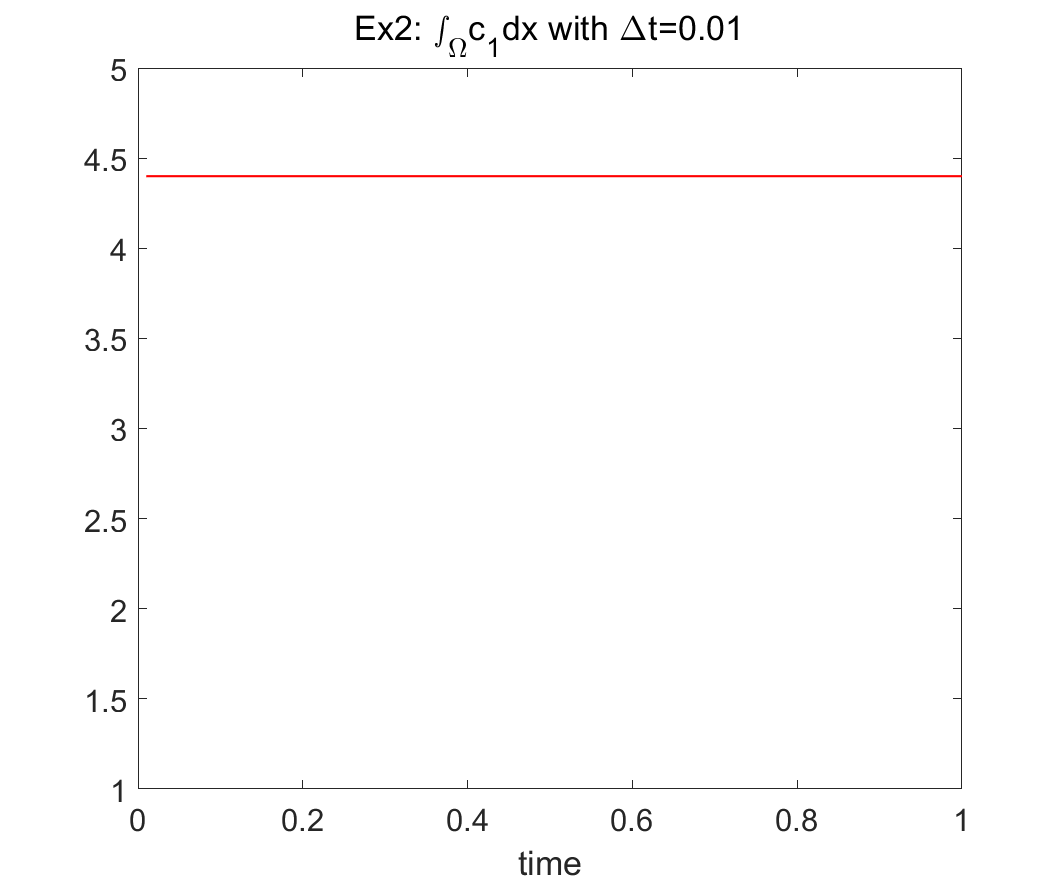}
}
\subfigure[]{
\includegraphics[width=7cm]{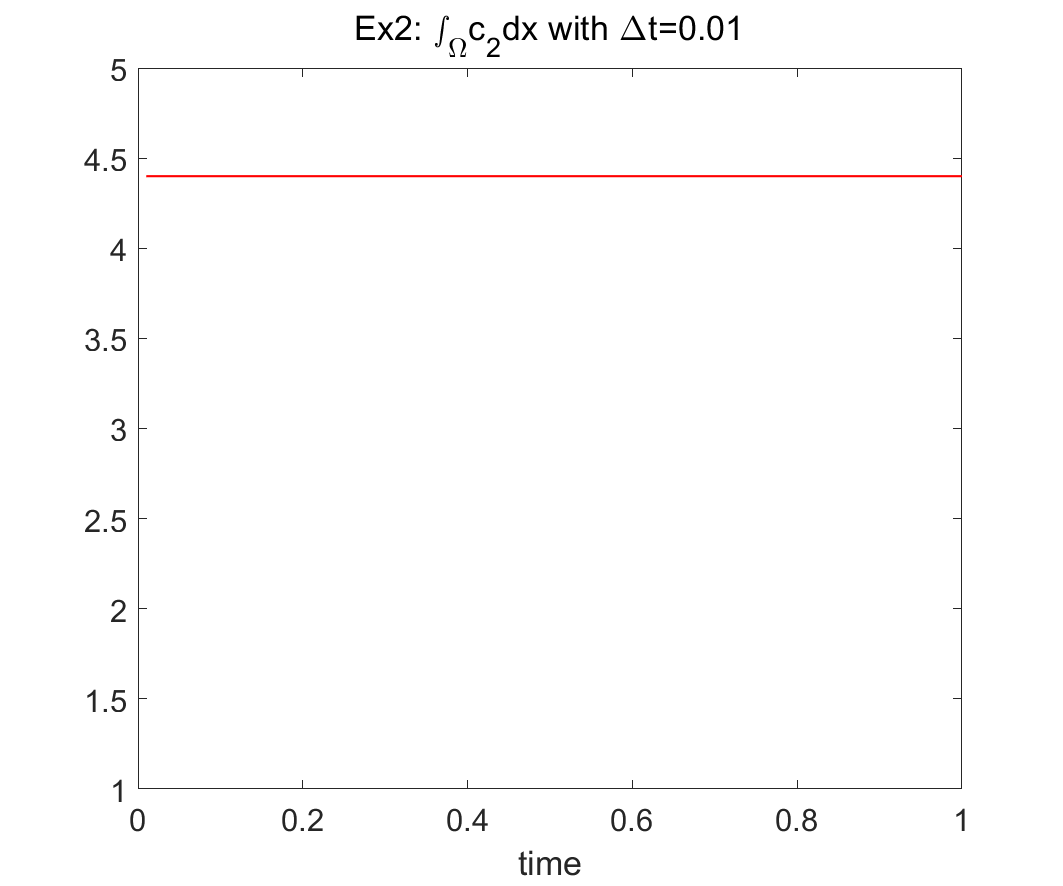}
}
\vspace{-6mm}
\caption{(Example 2) Time evolution of the discrete mass $\int_\O c_1d\x$ (a) and $\int_\O c_2d\x$ (b) computed with Scheme2b-SM.}
  \label{fig:mass}
\end{figure}

\begin{figure}[htbp]
\centering
\subfigure[$c_1$ at $t=0$]{
\includegraphics[width=4.2cm]{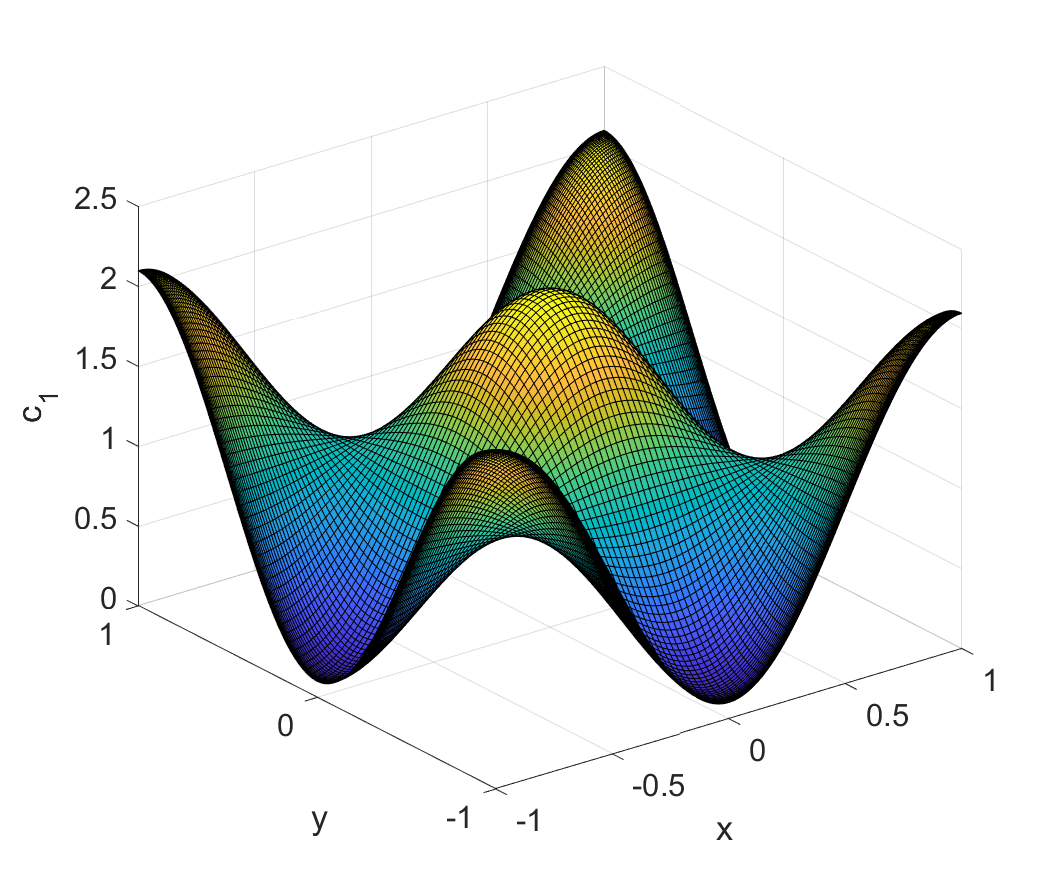}
}
\hspace{-8mm}
\subfigure[$c_1$ at $t=0.1$]{
\includegraphics[width=4.2cm]{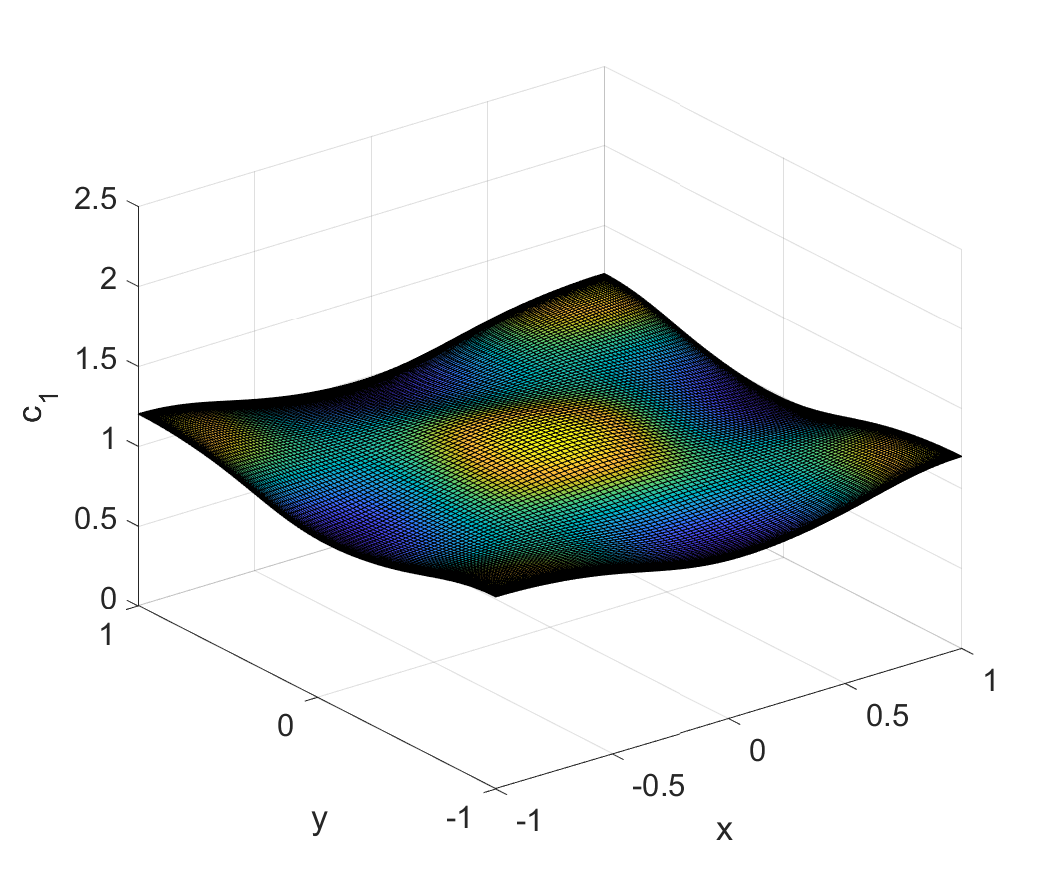}
}
\hspace{-8mm}
\subfigure[$c_1$ at $t=0.6$]{
\includegraphics[width=4.2cm]{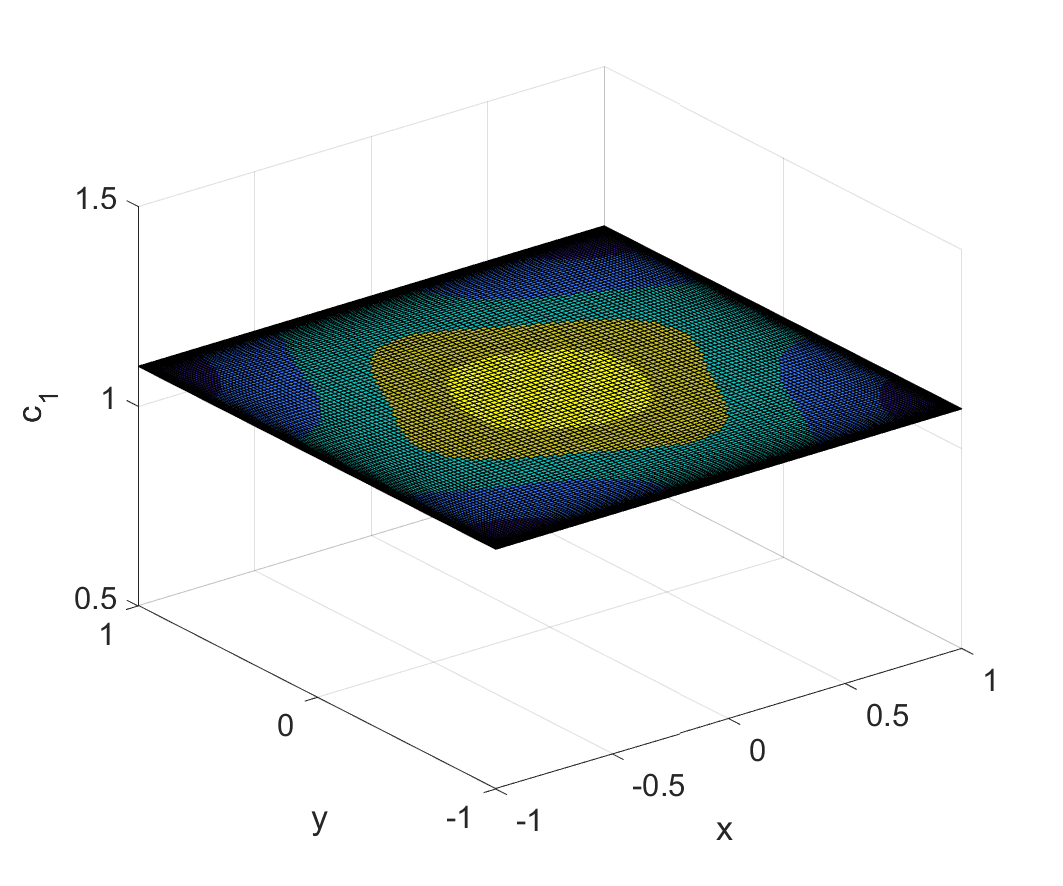}
}
\hspace{-8mm}
\subfigure[$c_1$ at $t=1$]{
\includegraphics[width=4.2cm]{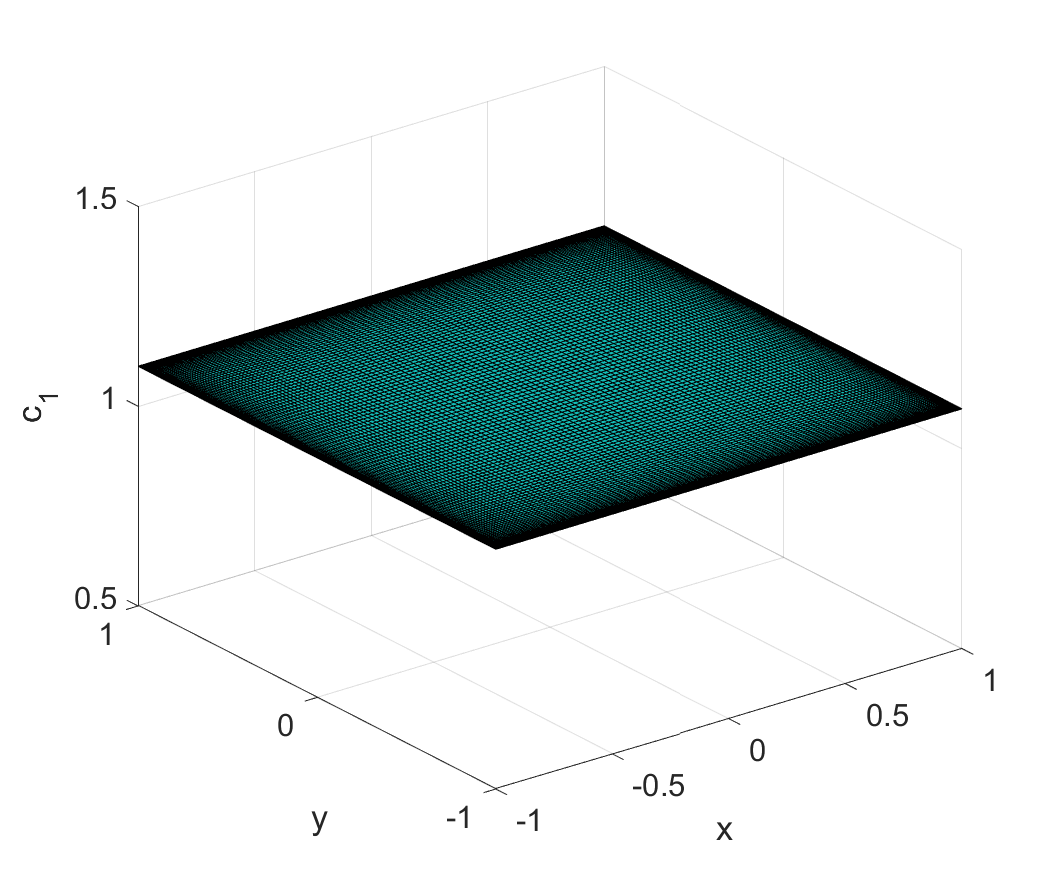}
}
\vspace{-5mm}
\caption{\small (Example 2) Snapshots of $c_1$.}
  \label{fig:evolution-c1}
\end{figure}

\begin{figure}[htbp]
\centering
\subfigure[$c_2$ at $t=0$]{
\includegraphics[width=4.2cm]{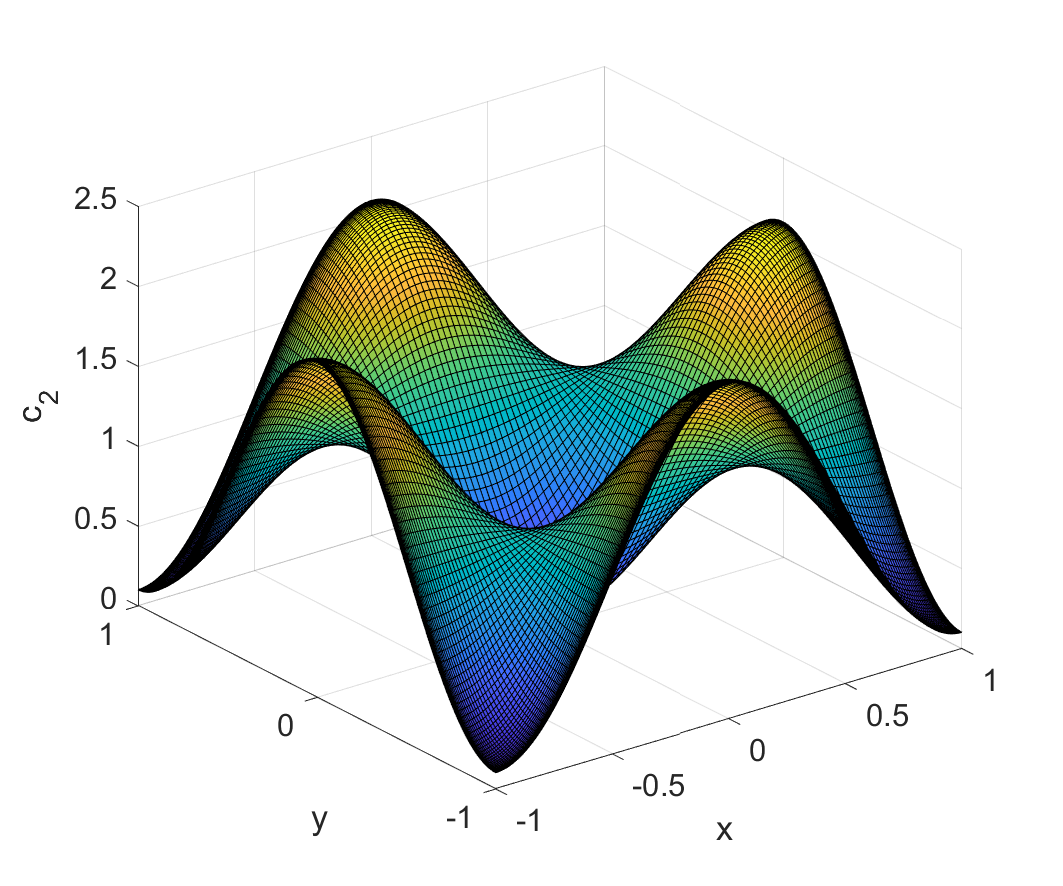}
}
\hspace{-8mm}
\subfigure[$c_2$ at $t=0.1$]{
\includegraphics[width=4.2cm]{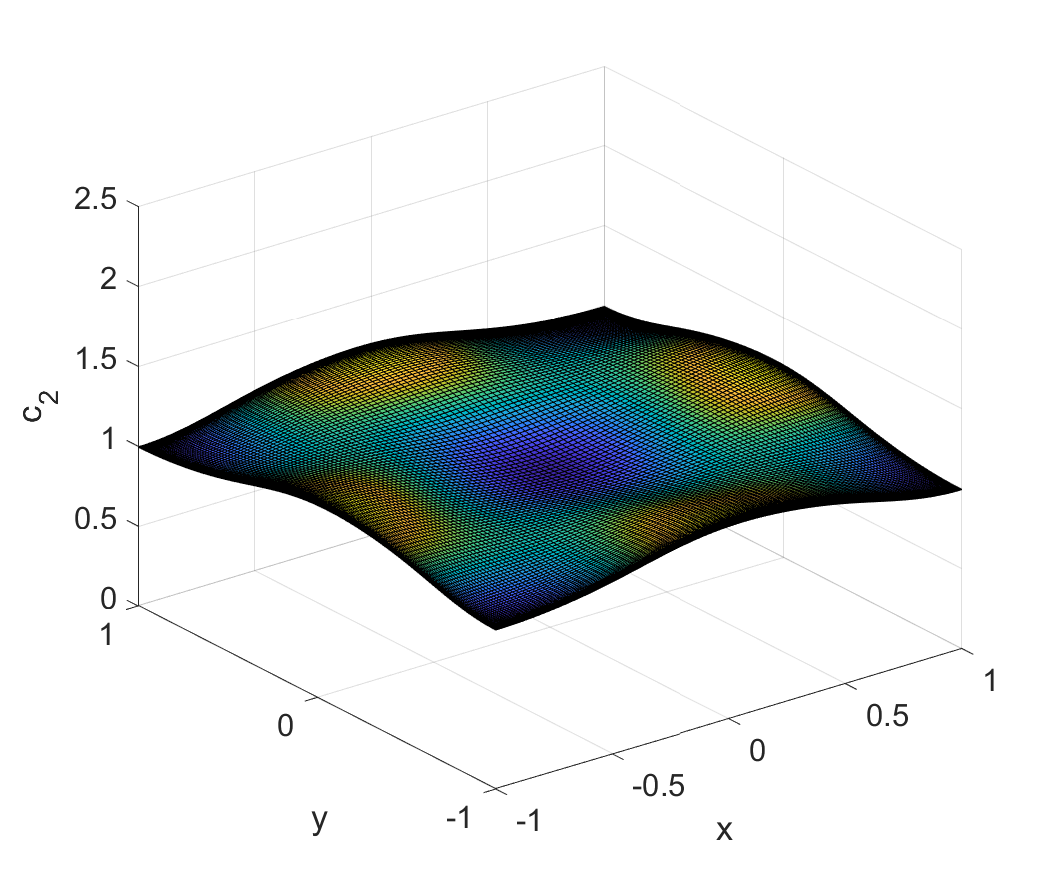}
}
\hspace{-8mm}
\subfigure[$c_2$ at $t=0.6$]{
\includegraphics[width=4.2cm]{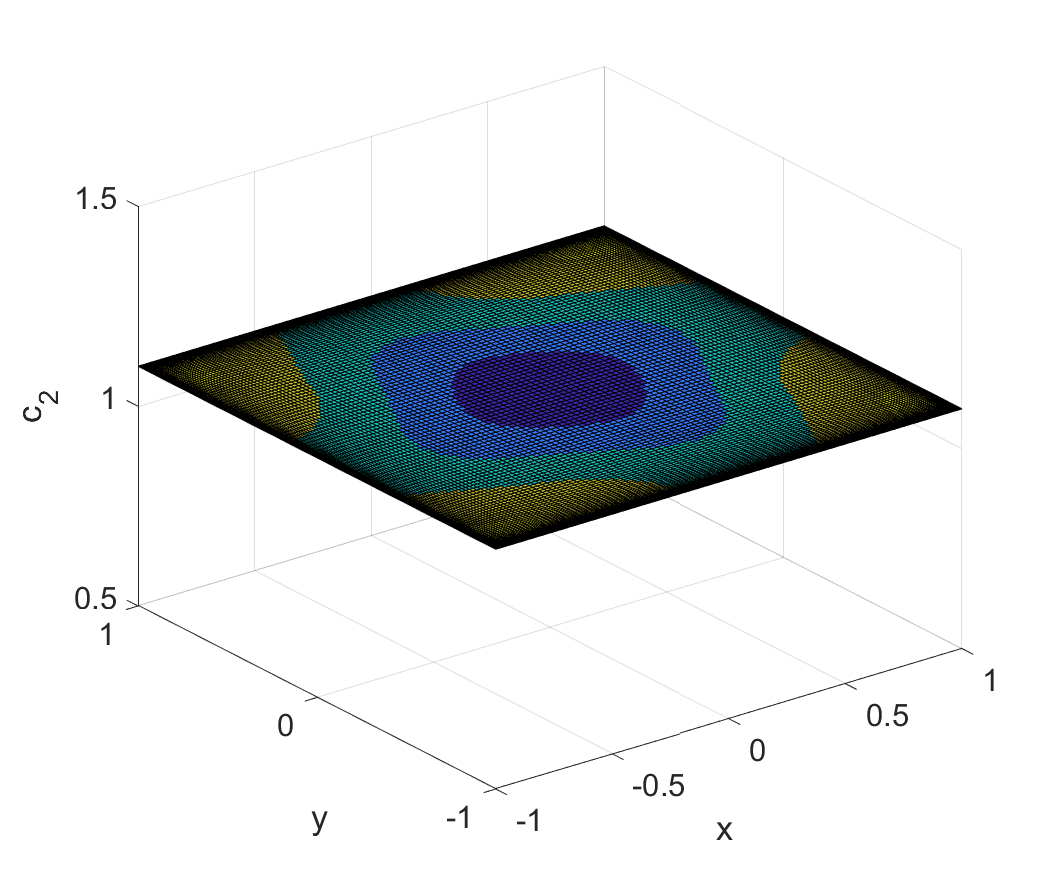}
}
\hspace{-8mm}
\subfigure[$c_2$ at $t=1$]{
\includegraphics[width=4.2cm]{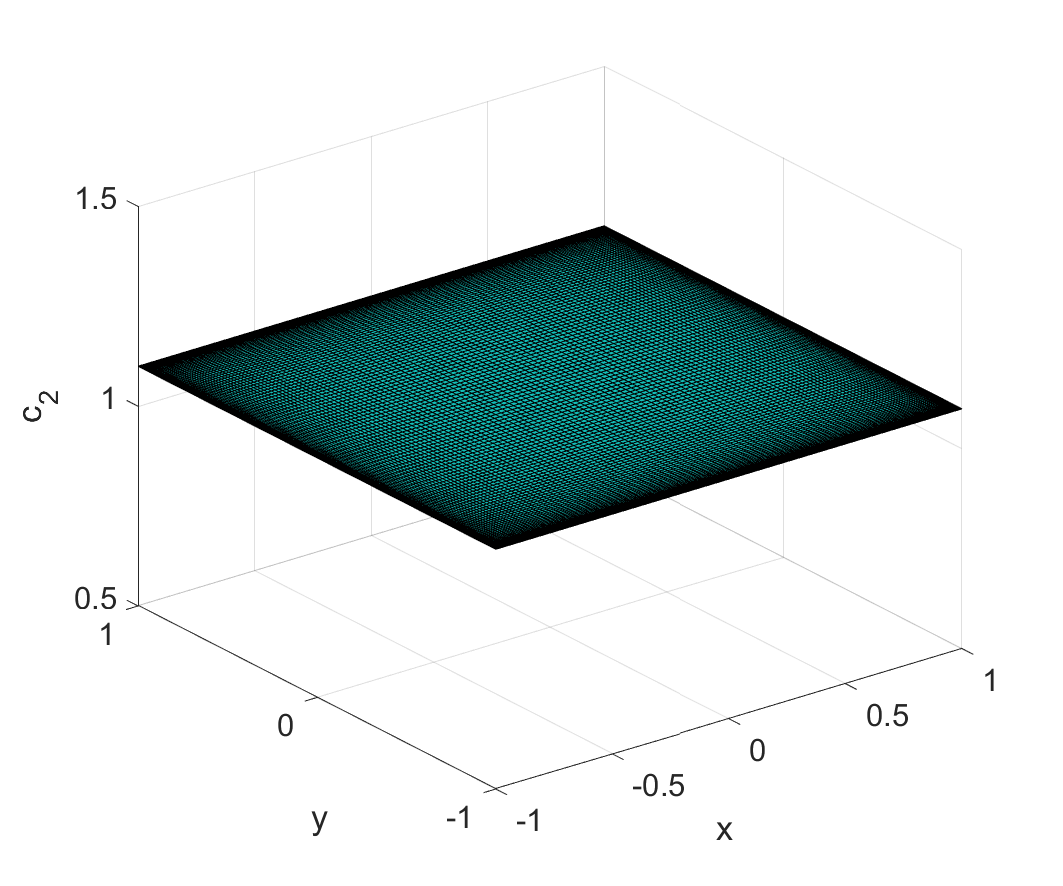}
}
\vspace{-5mm}
\caption{\small (Example 2) Snapshots of $c_2$.}
  \label{fig:evolution-c2}
\end{figure}

\begin{figure}[htbp]
\centering
\subfigure[$u_1$ at $t=0$]{
\includegraphics[width=4.2cm]{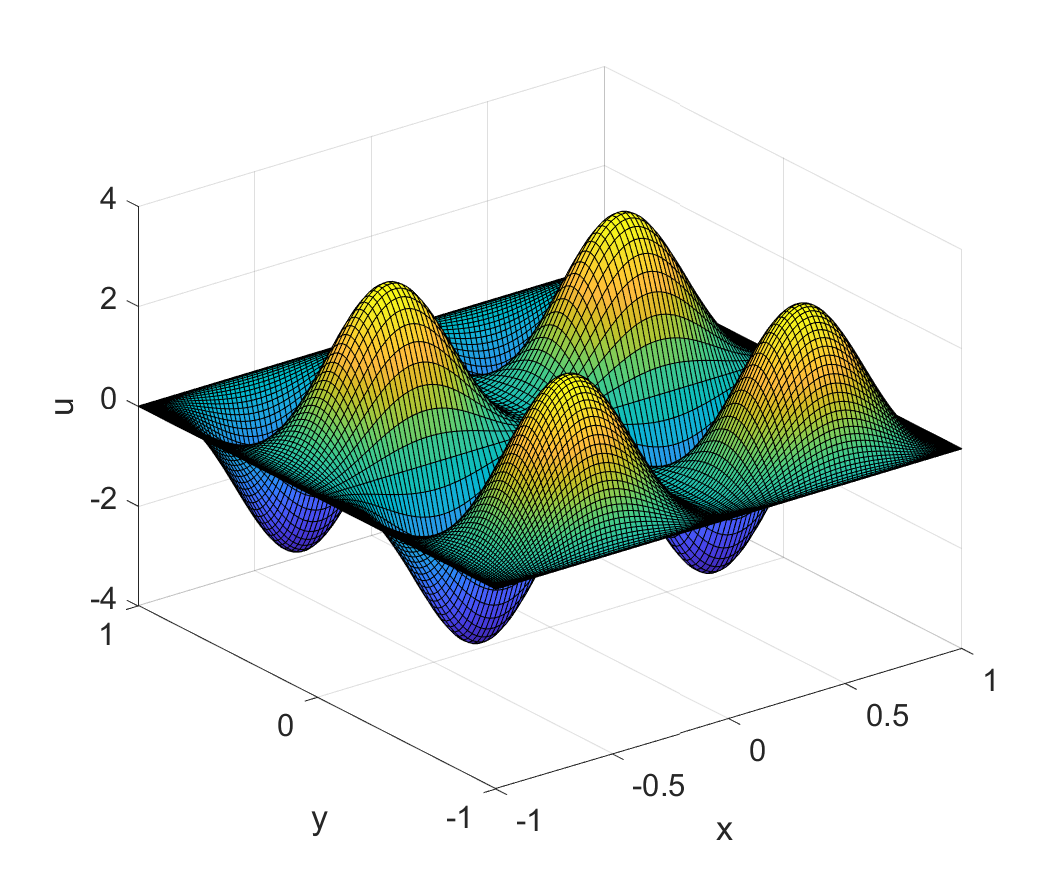}
}
\hspace{-8mm}
\subfigure[$u_1$ at $t=0.1$]{
\includegraphics[width=4.2cm]{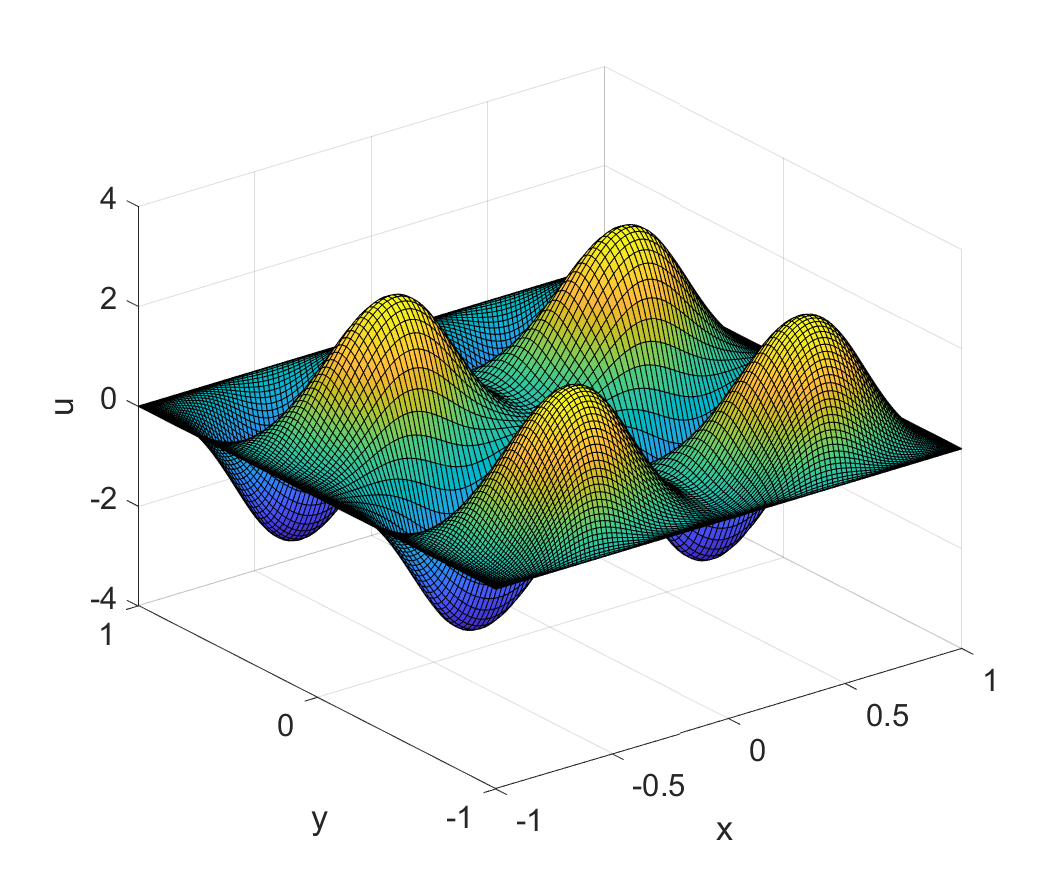}
}
\hspace{-8mm}
\subfigure[$u_1$ at $t=0.6$]{
\includegraphics[width=4.2cm]{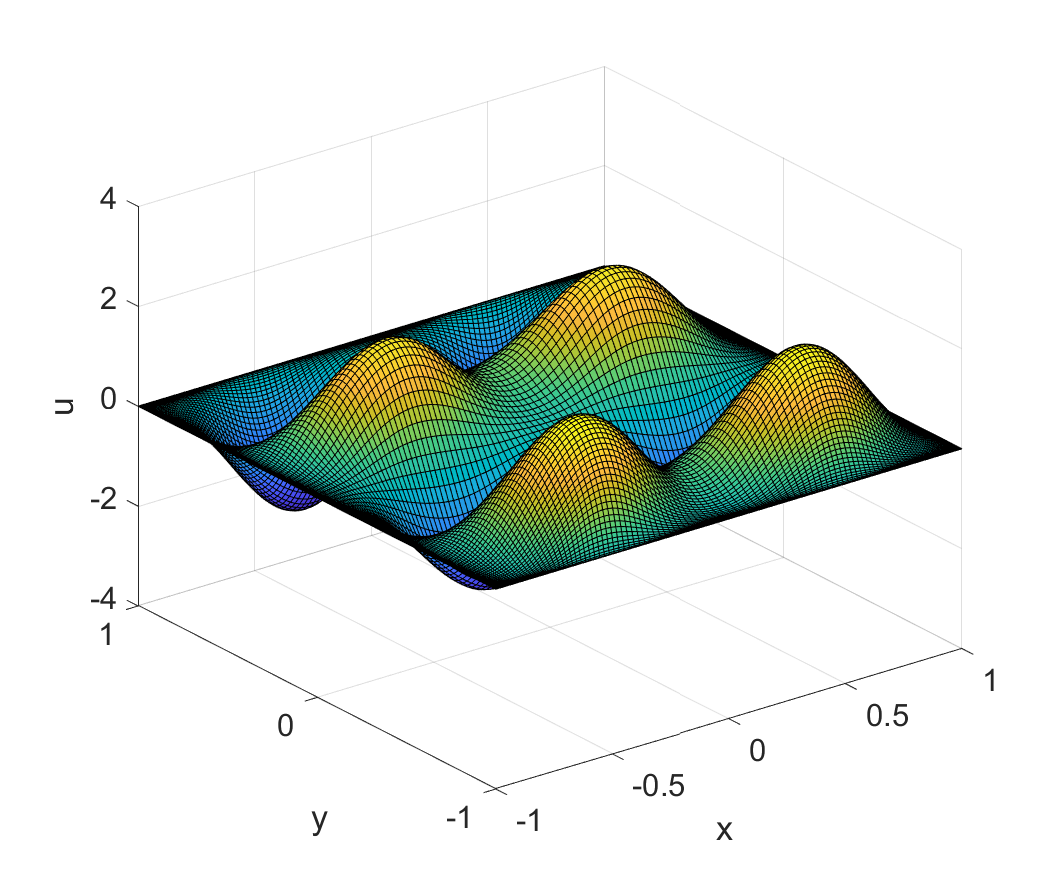}
}
\hspace{-8mm}
\subfigure[$u_1$ at $t=1$]{
\includegraphics[width=4.2cm]{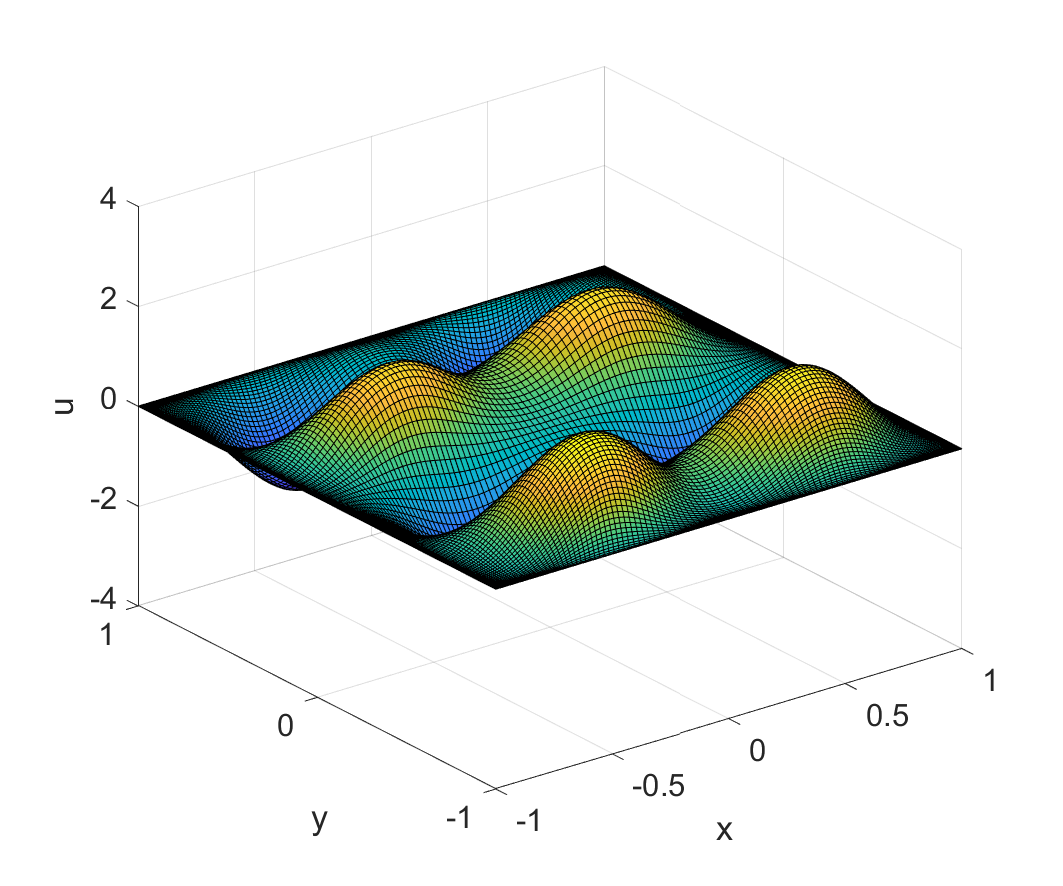}
}
\vspace{-5mm}
\caption{\small (Example 2) Snapshots of $u_1$.}
  \label{fig:evolution-u1}
\end{figure}

\begin{figure}[htbp]
\centering
\subfigure[$u_2$ at $t=0$]{
\includegraphics[width=4.2cm]{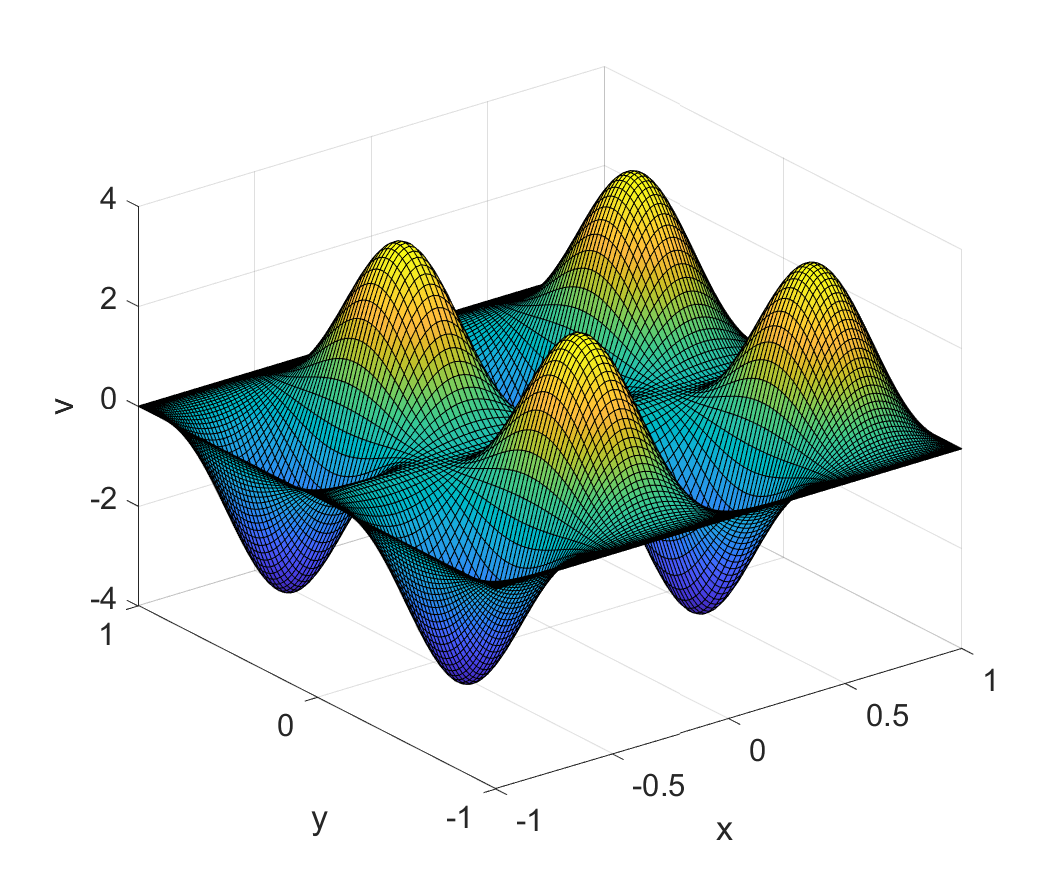}
}
\hspace{-8mm}
\subfigure[$u_1$ at $t=0.1$]{
\includegraphics[width=4.2cm]{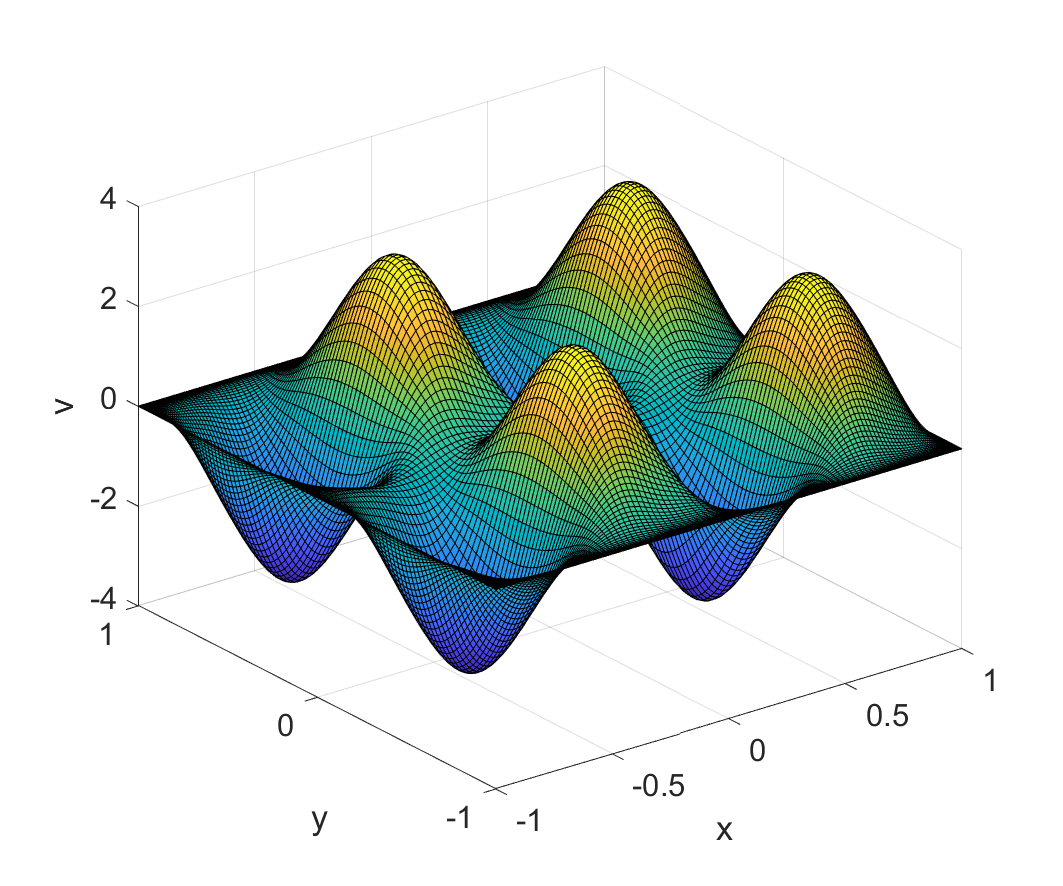}
}
\hspace{-8mm}
\subfigure[$u_2$ at $t=0.6$]{
\includegraphics[width=4.2cm]{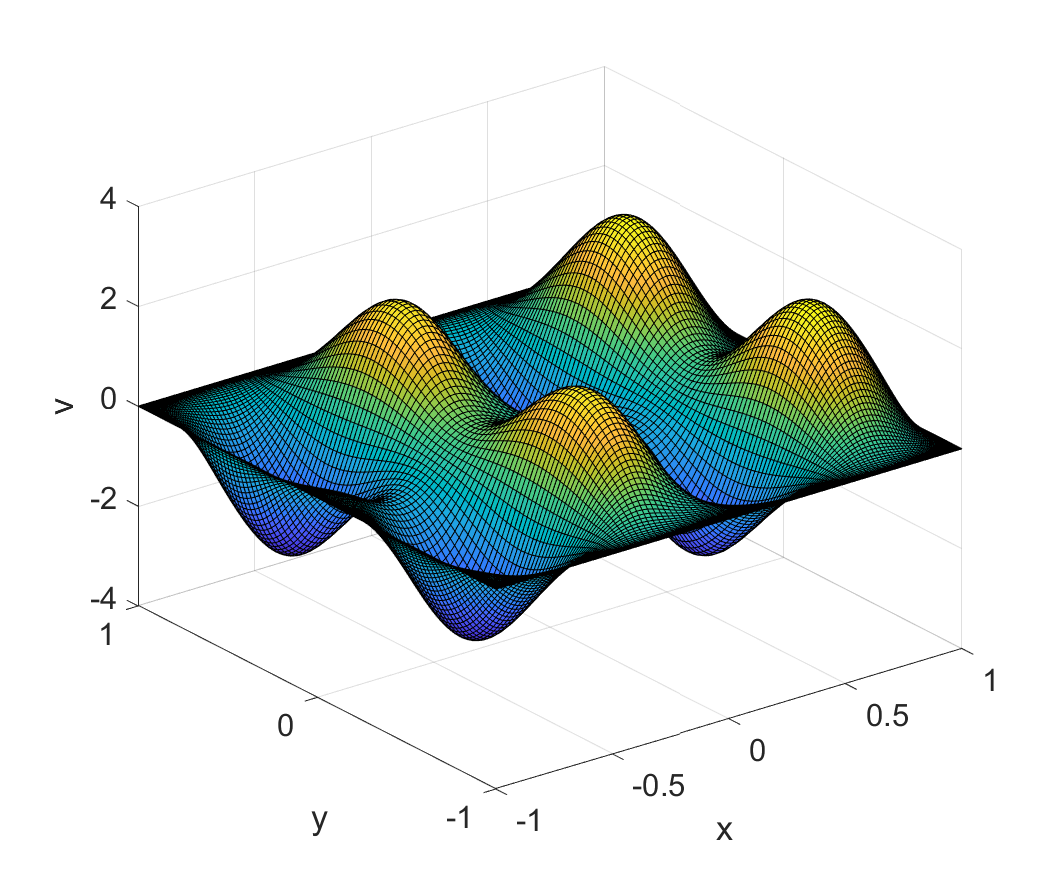}
}
\hspace{-8mm}
\subfigure[$u_2$ at $t=1$]{
\includegraphics[width=4.2cm]{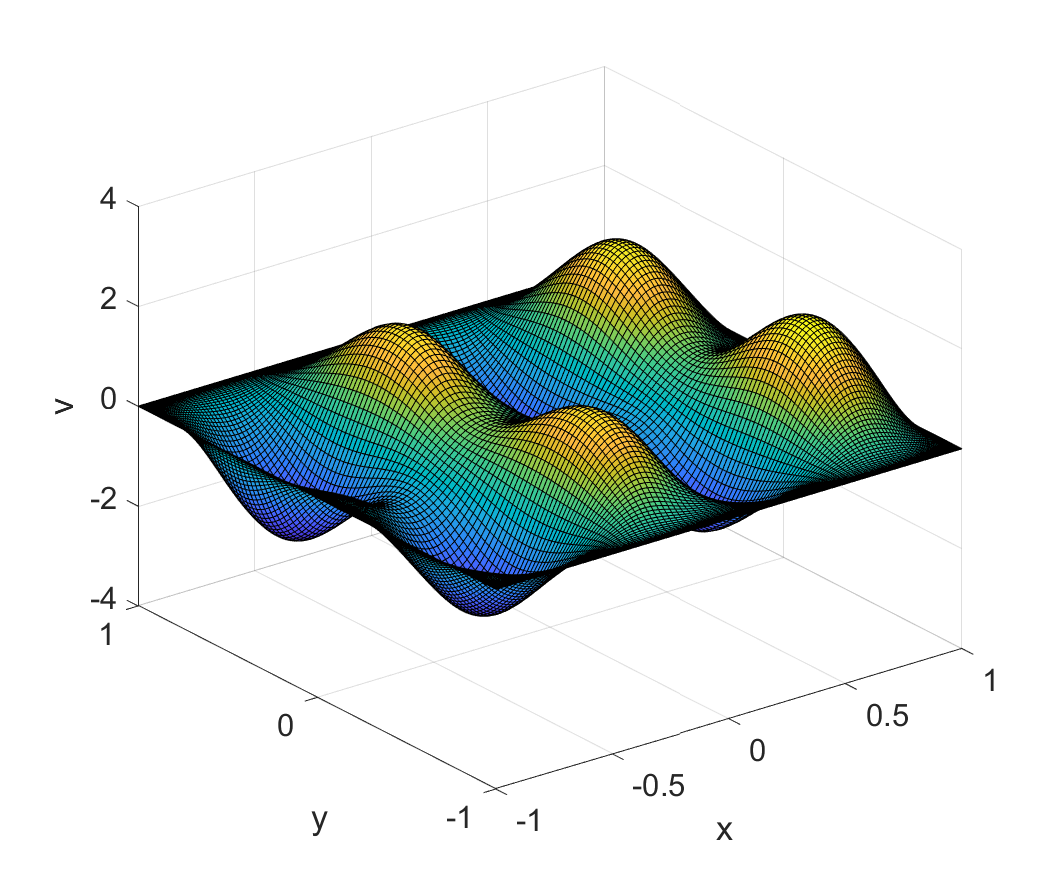}
}
\vspace{-5mm}
\caption{\small (Example 2) Snapshots of $u_2$ .}
  \label{fig:evolution-u2}
\end{figure}

\subsection{Case with three ions}

\subsubsection*{Example 3}
Set $z_1 = 1, z_2 = -1, z_3=2, D_1 = D_2 = D_3=1, \epsilon = 0.5,\ \nu=0.1$,
we verify the temporal convergence rates of
the proposed schemes using the following fabricated exact solution:
\begin{align*}
& \mb u=(\pi{\sin}(2 \pi y ) {\sin}^{2}(\pi x),-\pi\sin (2 \pi x) {\sin}^{2} (\pi y ) )e^{-t} \notag\\
& p=\sin( \pi x) \sin (\pi y)e^{-t}  \notag\\
& c_1=\cos(\pi x)\cos(\pi y)e^{-t} +2e^{-t} \\
& c_2=-2\cos(\pi x)\cos(\pi y)e^{-t}  +6e^{-t} \notag\\
& c_3=-\cos(\pi x)\cos(\pi y)e^{-t} +2e^{-t}  \notag\\
& \Phi=\frac{1}{\pi^2}\cos(\pi x)\cos(\pi y)e^{-t}.  \notag
\end{align*}
The source terms are obtained from the exact solution.
Figures \ref{fig:ex3-ord1} and \ref{fig:ex3-ord2} present the $L^2$ errors of the ions,
the electrostatic potential, the velocity, and the pressure as functions of the time step
size, computed from \textbf{Scheme1-SM} and \textbf{Scheme2-SM} respectively.
As observed from the figures, the convergence rates are respectively first order for \textbf{Scheme1-SM} and
second order for \textbf{Scheme2-SM}. This in a good agreement
with the theoretical prediction.
\begin{figure}[htbp]
\centering
\subfigure{
\includegraphics[width=5.6cm]{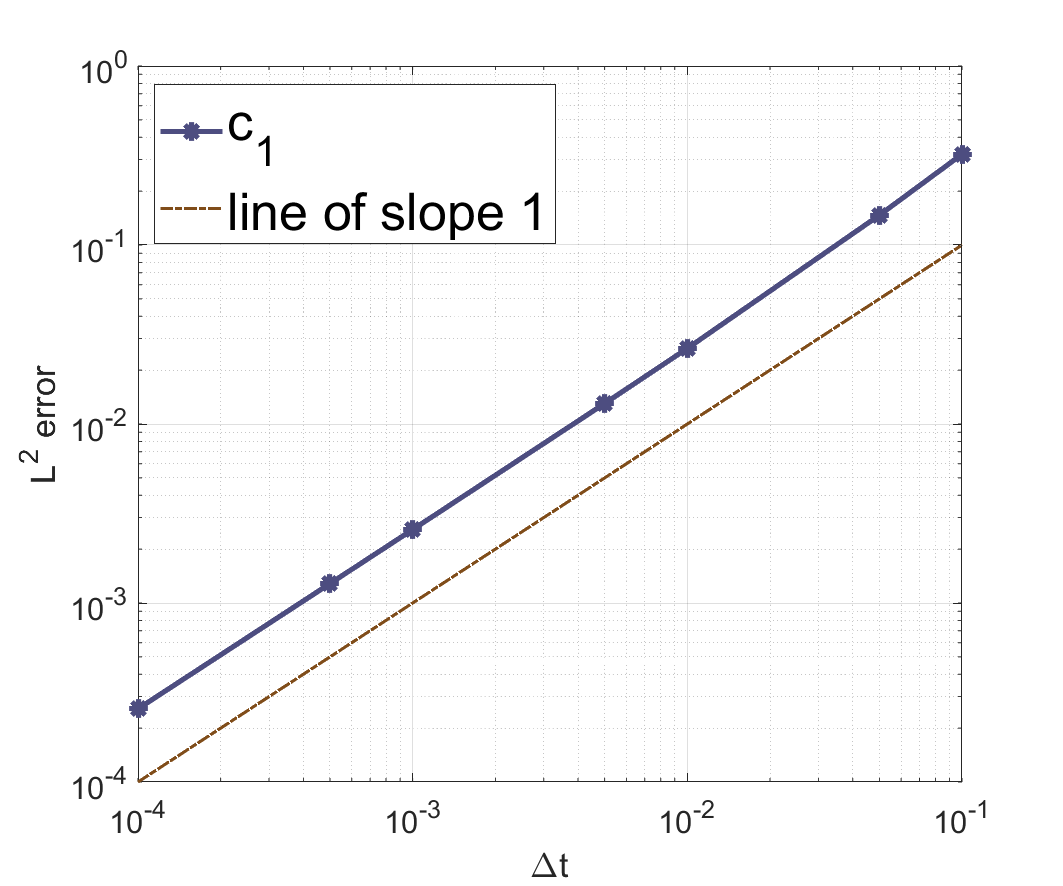}
}
\hspace{-7mm}
\subfigure{
\includegraphics[width=5.6cm]{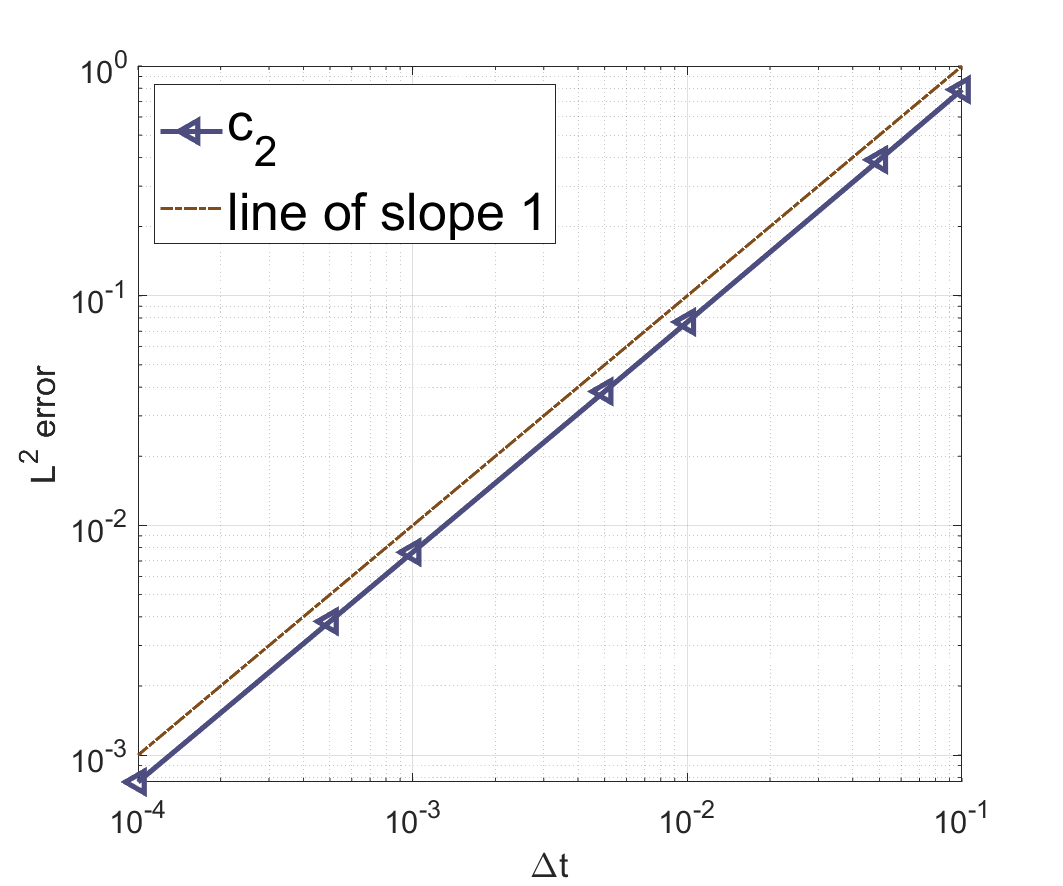}
}
\hspace{-7mm}
\subfigure{
\includegraphics[width=5.6cm]{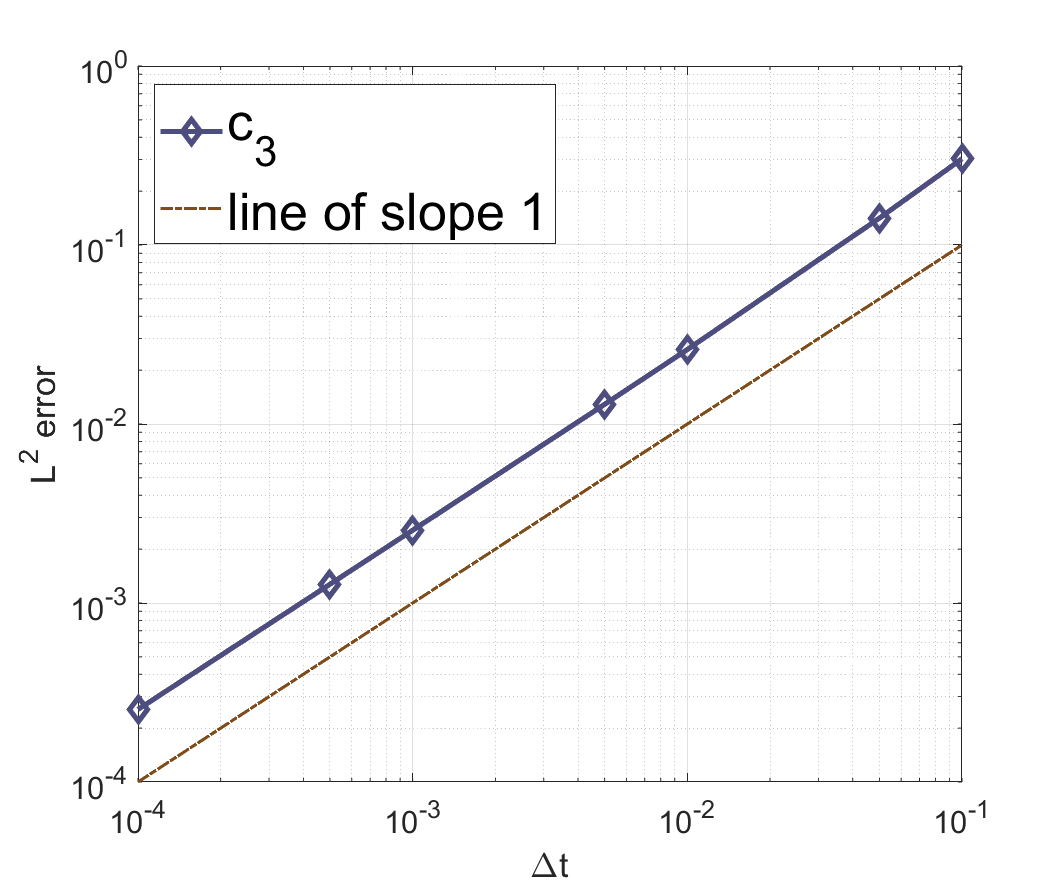}
}
\hspace{-7mm}
\subfigure{
\includegraphics[width=5.6cm]{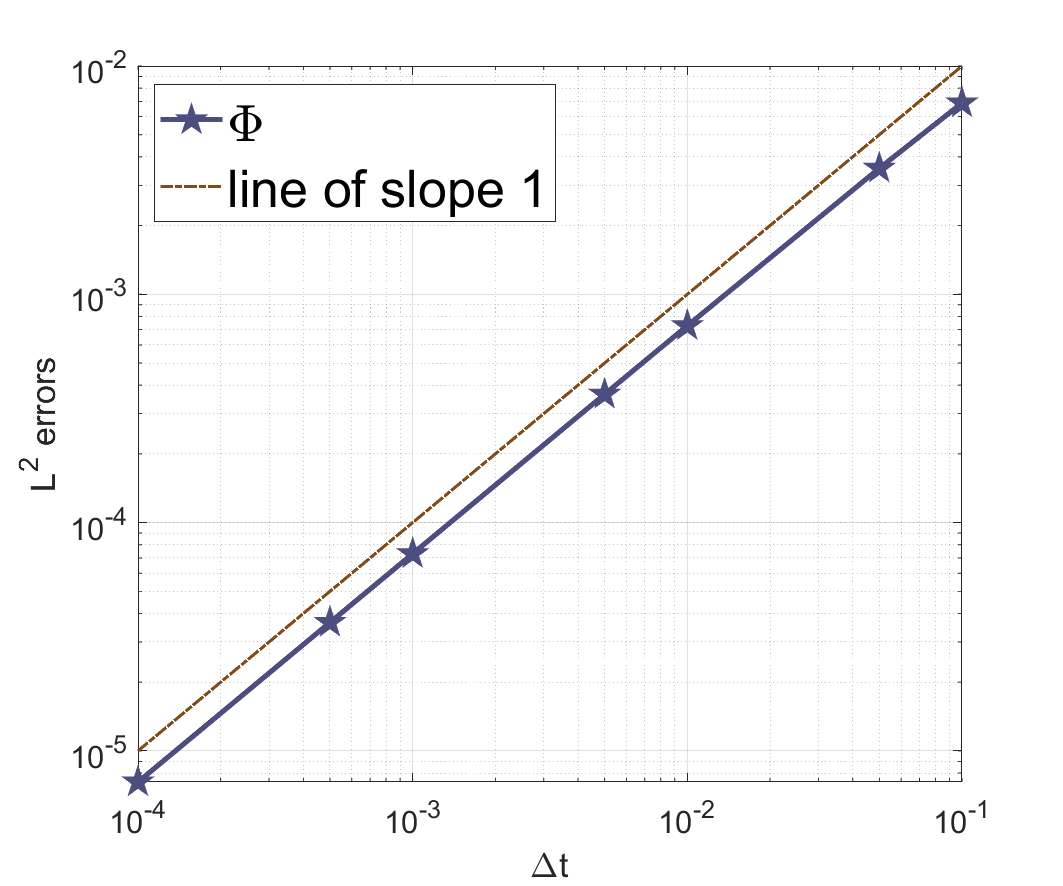}
}
\hspace{-7mm}
\subfigure{
\includegraphics[width=5.6cm]{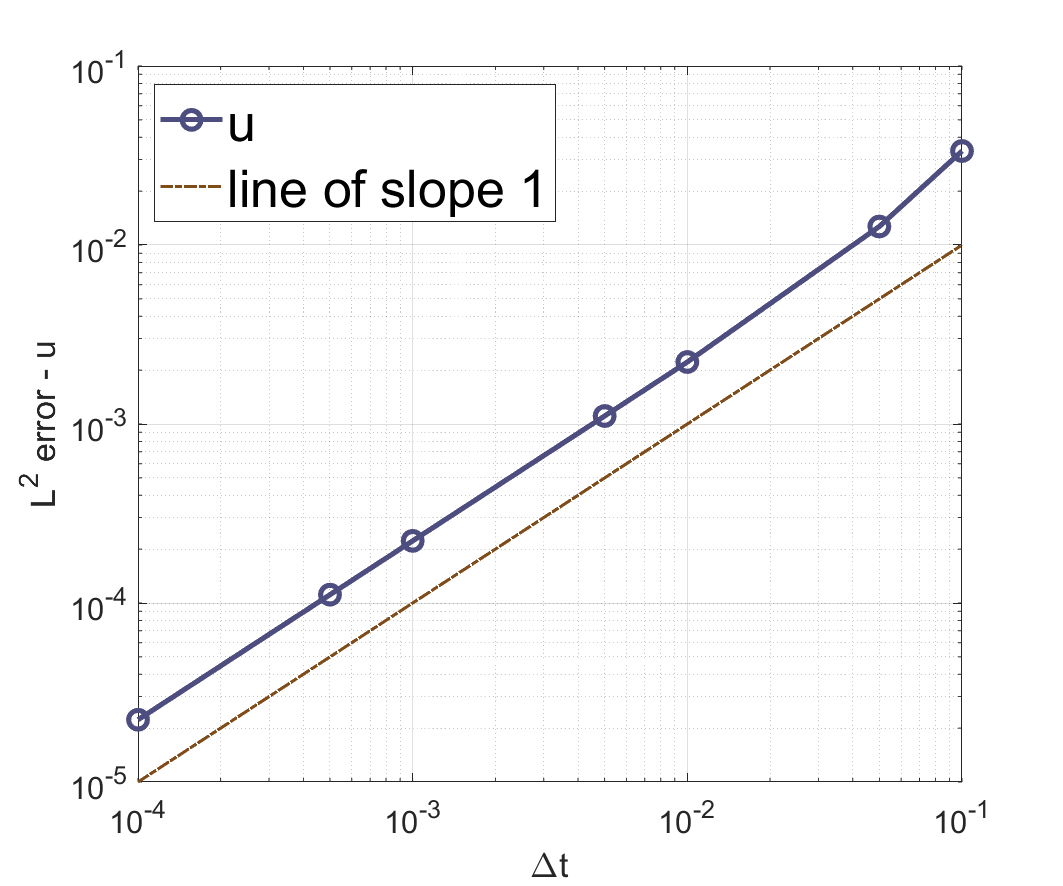}
}
\hspace{-7mm}
\subfigure{
\includegraphics[width=5.6cm]{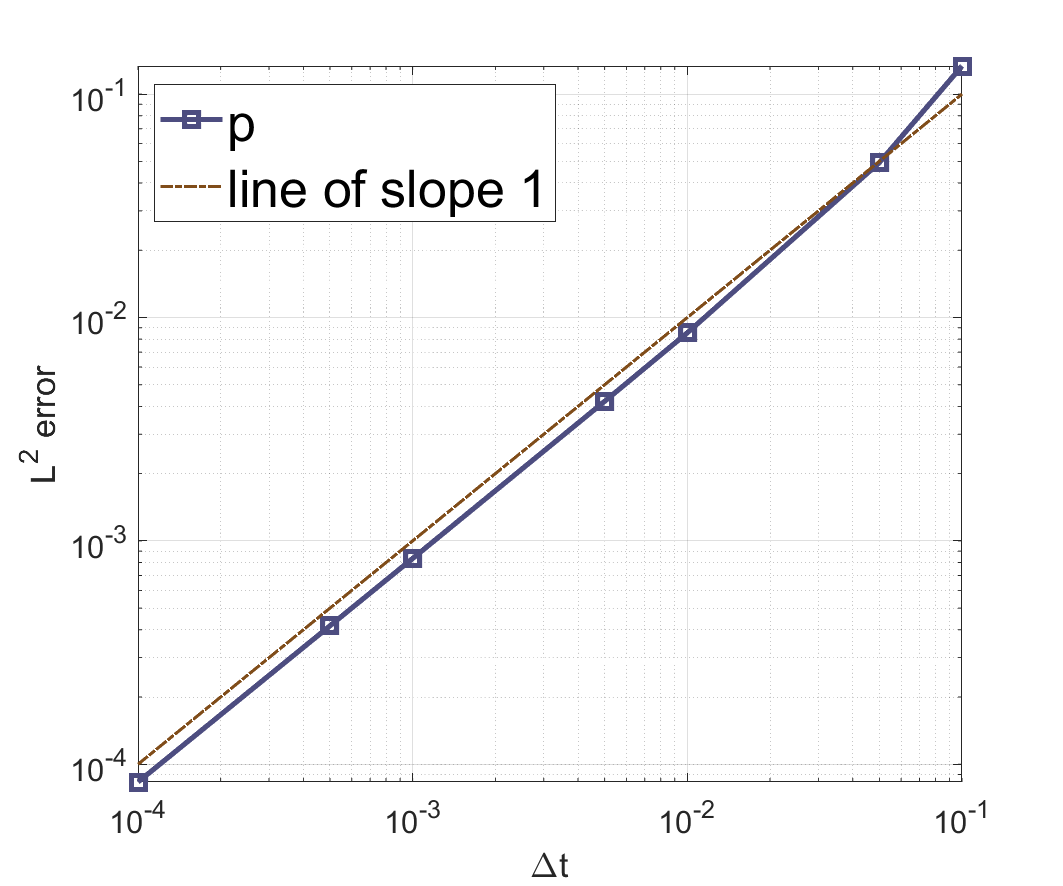}
}
\caption{(Example 3) $L^{2}$ error versus $\Delta t$ in log-log scale using the first order scheme with $N=64$ at each spatial direction.}
  \label{fig:ex3-ord1}
\end{figure}
\begin{figure}[htbp]
\centering
\subfigure{
\includegraphics[width=5.6cm]{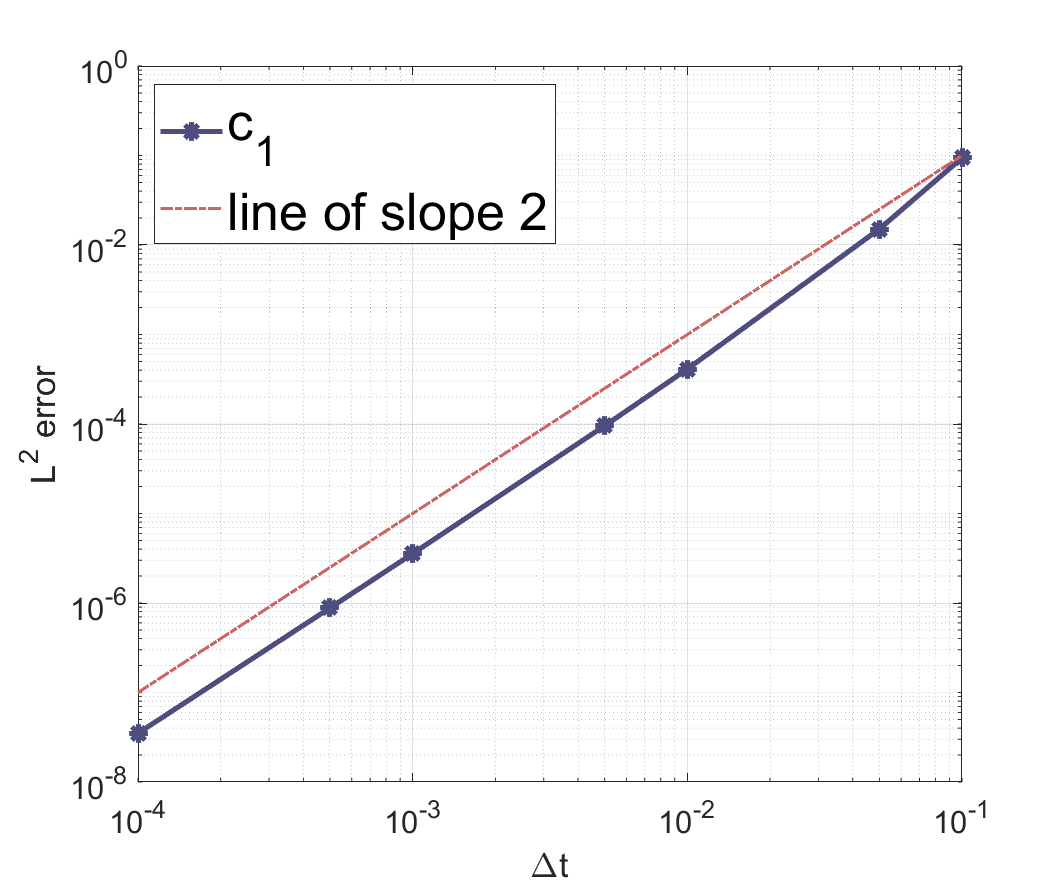}
}
\hspace{-7mm}
\subfigure{
\includegraphics[width=5.6cm]{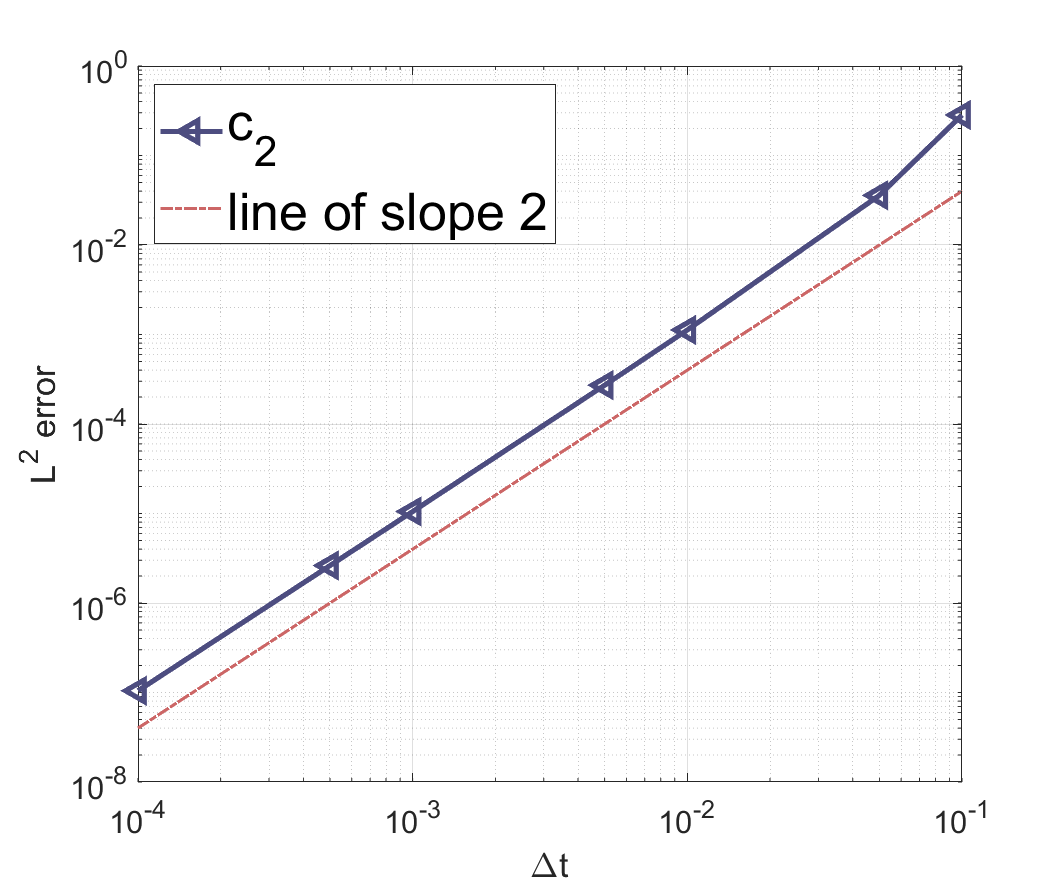}
}
\hspace{-7mm}
\subfigure{
\includegraphics[width=5.6cm]{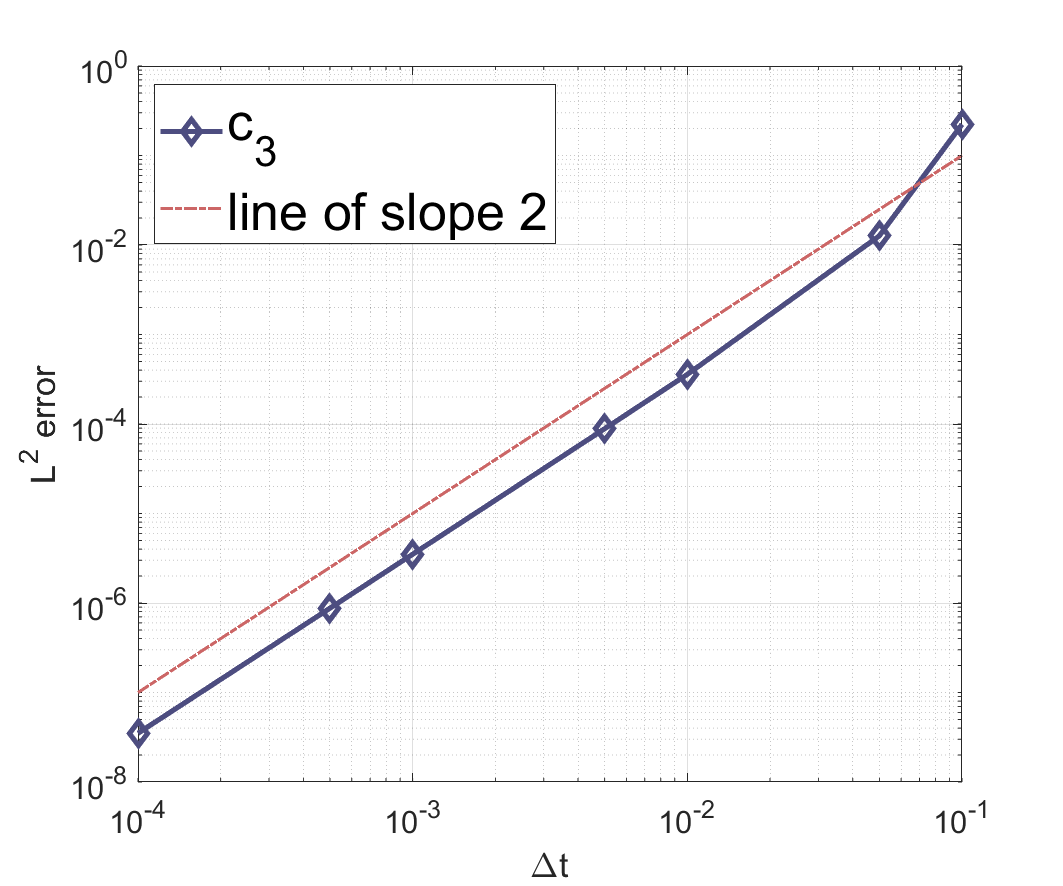}
}
\hspace{-7mm}
\subfigure{
\includegraphics[width=5.6cm]{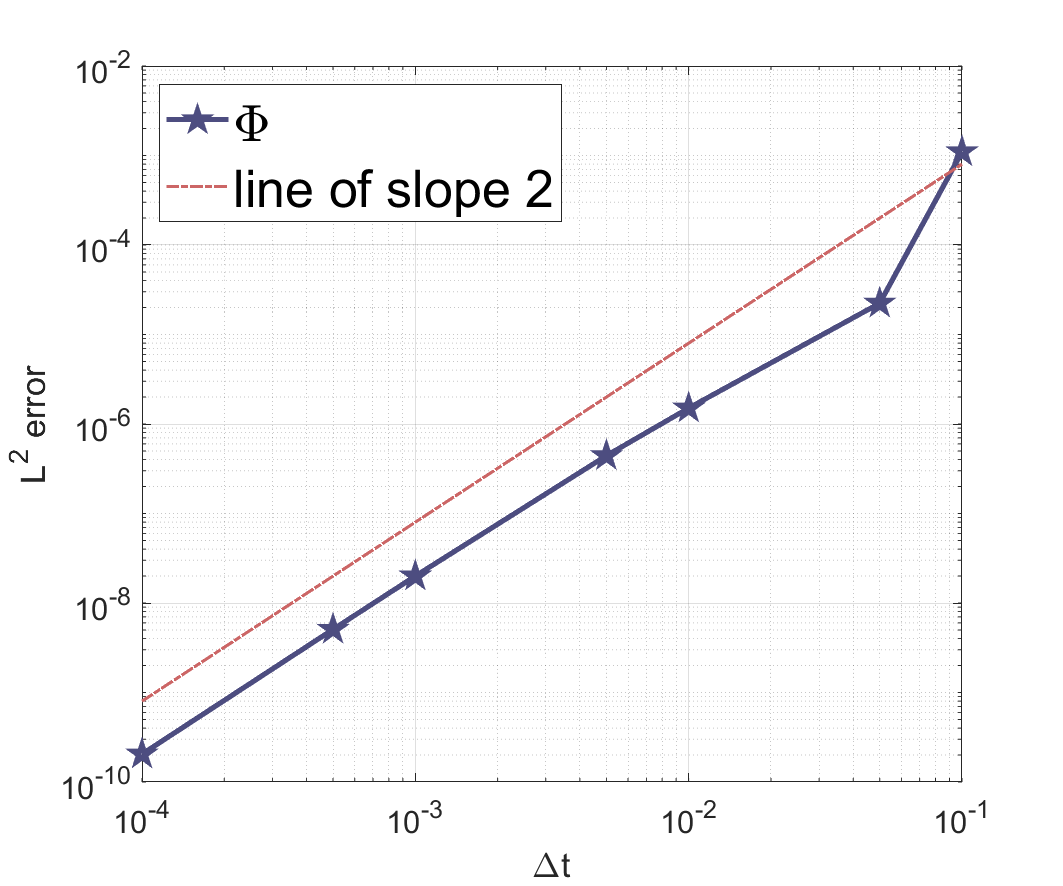}
}
\hspace{-7mm}
\subfigure{
\includegraphics[width=5.6cm]{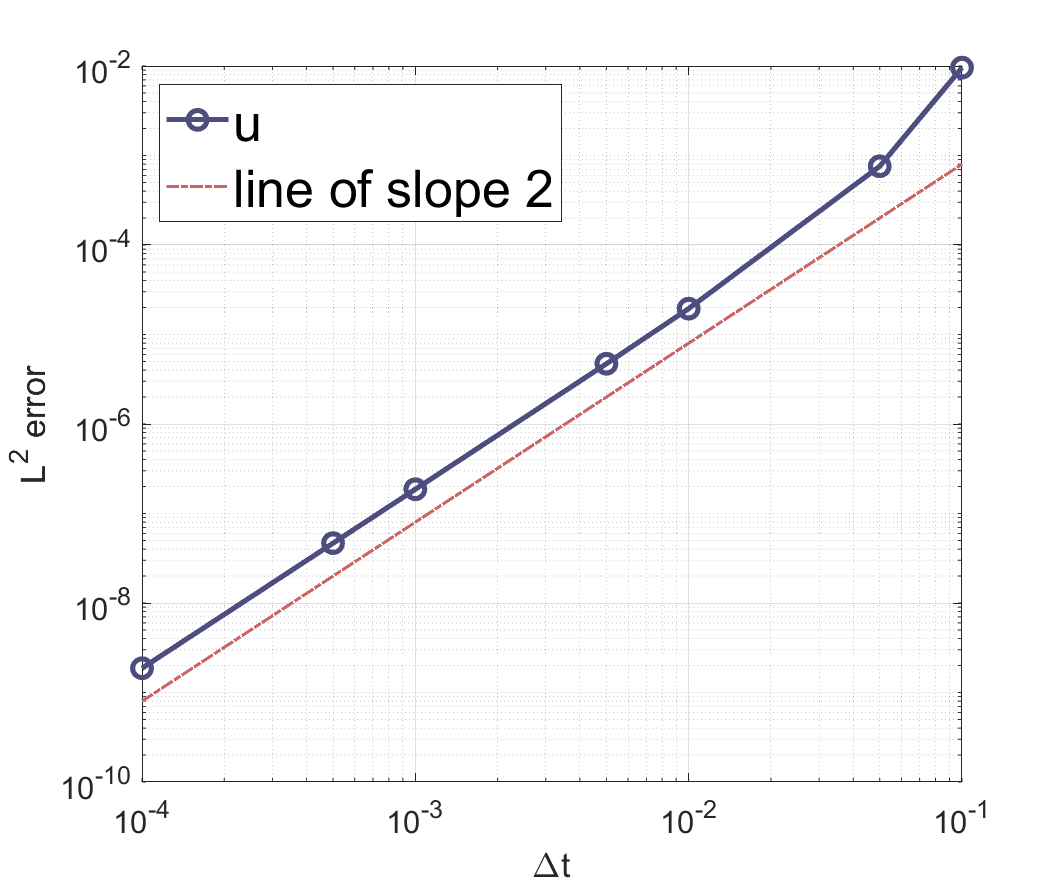}
}
\hspace{-7mm}
\subfigure{
\includegraphics[width=5.6cm]{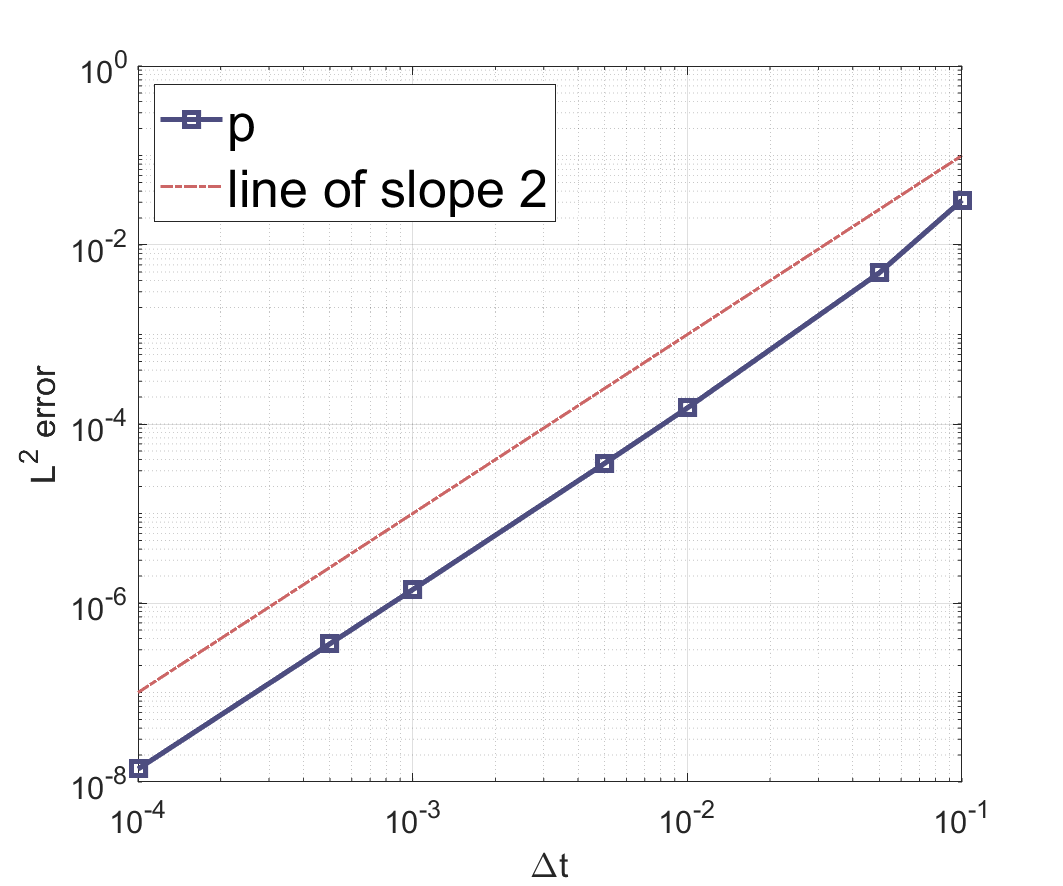}
}
\caption{(Example 3) $L^{2}$ errors versus $\Delta t$ in log-log scale using Scheme2-SM with $N=64$ at each spatial direction.}
  \label{fig:ex3-ord2}
\end{figure}

\section{Concluding Remarks}
\noindent

In this paper, we have developed
efficient time-stepping schemes for the Navier-Stokes-Nernst-Planck-Poisson equations.
The proposed schemes are constructed based on an auxiliary variable approach
for the Navier-Stokes equations and a delicate treatment of
the terms coupling the Navier-Stokes equations and the Nernst-Planck-Poisson
equations.
By introducing a dynamic equation for the auxiliary variable and
reformulating the original equations into an equivalent system,
we have constructed first- and second-order semi-implicit linearized schemes
for the underlying problem.
 A rigorous analysis was carried out, showing
that the overall schemes are unconditionally stable, and preserve positivity
and mass conservation of the ionic concentration solutions.
The implementation showed that it can be implemented in an efficient way: the computational complexity is equal to
solving several decoupled linear equations with constant coefficient at each time step. A number of numerical examples were provided to confirm the theoretical claims.
 We emphasize that the above attractive properties remain to be held at the full discrete level.
 As far as the best we know, this is the first second-order method
 which satisfies all the above properties for the Navier-Stokes-Nernst-Planck-Poisson equations at the discrete level.

\bibliographystyle{siam}
\bibliography{refs}

\end{document}